\documentclass[11pt]{amsart}
\usepackage{graphicx}
\usepackage{amssymb}
\usepackage{amssymb,amsmath,amsthm}
\usepackage[all]{xy}
\usepackage{xcolor,colortbl}
\usepackage{soul}
\usepackage{pdflscape}

\newtheorem{thm}{Theorem}[section]
\newtheorem{prop}[thm]{Proposition}
\newtheorem{lem}[thm]{Lemma}
\newtheorem{cor}[thm]{Corollary}
 
\theoremstyle{definition}

\newtheorem{remark}[thm]{Remark}

\numberwithin{equation}{section}
\newcommand{\sfA}{\mathsf{A}}

\newcommand{\sfD}{\mathsf{D}}
\newcommand{\sfE}{\mathsf{E}}

\newcommand{\sfc}{\mathsf{c}}
\newcommand{\frakm}{\mathfrak{m}}

\newcommand{\bbA}{{\mathbb{A}}}
\newcommand{\bbZ}{{\mathbb{Z}}}
\newcommand{\bbP}{{\mathbb{P}}}
\newcommand{\bbG}{{\mathbb{G}}}
\newcommand{\bbC}{{\mathbb{C}}}

\newcommand{\bbF}{{\mathbb{F}}}

\newcommand{\bbk}{{\Bbbk}}

\newcommand{\bfF}{{\mathbf{F}}}

\newcommand{\GL}{\operatorname{GL}}
\newcommand{\gl}{\mathfrak{gl}}

\newcommand{\PGL}{\operatorname{PGL}}
\newcommand{\Aut}{\operatorname{Aut}}
\newcommand{\Sym}{\operatorname{Sym}}

\newcommand{\id}{\operatorname{id}}
\newcommand{\Pic}{\operatorname{Pic}}

\newcommand{\He}{\operatorname{He}}
\newcommand{\Tr}{\operatorname{Tr}}

\newcommand{\PSU}{\operatorname{PSU}}

\newcommand{\cha}{\operatorname{char}}
\newcommand{\Span}{\operatorname{span}}

\newcommand{\Crit}{\operatorname{Crit}}
\newcommand{\disc}{\operatorname{disc}}
\newcommand{\PSL}{\operatorname{PSL}}

\newcommand{\SO}{\operatorname{SO}}

\newcommand{\cub}{\operatorname{cub}}
\newcommand{\bsm}{\left(\begin{smallmatrix}}
\newcommand{\esm}{\end{smallmatrix}\right)}

\newcommand{\calD}{\mathcal{D}}
\newcommand{\calO}{\mathcal{O}}
\newcommand{\calH}{\mathcal{H}}
\newcommand{\calX}{\mathcal{X}}

\newcommand{\calE}{\mathcal{E}}
\newcommand{\calL}{\mathcal{L}}
\newcommand{\calP}{\mathcal{P}}
\newcommand{\calT}{\mathcal{T}}
\newcommand{\calM}{\mathcal{M}}

\newcommand{\frakS}{\mathfrak{S}}
\newcommand{\frakA}{\mathfrak{A}}

\newcommand{\beq}{\begin{equation}}
\newcommand{\eeq}{\end{equation}}
\usepackage[
            pdftex,
           colorlinks=true,
            urlcolor=blue,       
            filecolor=blue,     
            linkcolor=blue,       
            citecolor=blue,         
            pdftitle= {Title},
            pdfauthor={Author},
            pdfsubject={Subject},
            pdfkeywords={Key}
            pagebackref,
            bookmarksopen=true]
            {hyperref}
\usepackage{hyperref}

\begin{document}

\title{Automorphisms of cubic surfaces in positive characteristic}
\author{Igor Dolgachev}
\address{Department of Mathematics, University of Michigan, 
Ann Arbor, MI 48109}
\email{idolga@umich.edu}
\author{Alexander Duncan}
\address{Department of Mathematics, University of South Carolina, 
Columbia, SC 29208}
\email{duncan@math.sc.edu}
\thanks{The second author was partially supported by
National Security Agency grant H98230-16-1-0309.}

\begin{abstract} 
We classify all possible automorphism groups of smooth cubic surfaces
over an algebraically closed field of arbitrary characteristic.
As an intermediate step we also classify automorphism groups of quartic
del Pezzo surfaces.
We show that the moduli space of smooth cubic surfaces is
rational in every characteristic,
determine the dimensions of the strata admitting each possible
isomorphism class of automorphism group, and find explicit normal forms
in each case.
Finally, we completely characterize when
a smooth cubic surface in positive characteristic, together with a group
action, can be lifted to characteristic zero.
\end{abstract}

\maketitle

\setcounter{tocdepth}{1}
\tableofcontents

\section{Introduction}

Let $X$ be a smooth cubic surface in $\bbP^3$ defined over an
algebraically closed field $\bbk$ of arbitrary characteristic $p$.
The primary purpose of this paper is to classify the possible
automorphism groups of $X$.
Along the way, we also classify the possible automorphism groups of del Pezzo
surfaces of degree 4.
This is progress towards the larger goal of classifying finite subgroups
of the plane Cremona group in positive characteristic
(see \cite{DI} for characteristic $0$).

In characteristic $0$, the first attempts at a classification of cubic
surfaces were undertaken by S.~Kantor~\cite{Kantor}, then
A.~Wiman~\cite{Wiman} in late nineteenth century.
In 1942, B.~Segre~\cite{Segre} classified all automorphism groups in
his book on cubic surfaces.
Unfortunately, all of these classification had errors.
The first correct classification was carried out using computers by
T.~Hosoh~\cite{Hosoh} in 1997.
A modern classification, which does not require computers, can be
found in~\cite{CAG}.

Our approach emphasizes the similarities across characteristics,
rather than the differences.
There is a well-known injective group homomorphism $\Aut(X) \to W(\sfE_6)$,
unique up to conjugacy, where $\Aut(X)$ is the group of automorphisms of
$X$ and $W(\sfE_6)$ is the Weyl group of the root system $\sfE_6$.
This homomorphism can also be described as the induced action of
$\Aut(X)$ on the configuration of 27 lines.
In this way, automorphism groups of cubic surfaces can be directly
compared to each other even if they are defined over different fields.

Let $\calM_{\cub}(\bbk)$ be the coarse moduli space of smooth
cubic surfaces over $\bbk$.
The space $\mathcal{M}_{\textrm{cub}}(\bbk)$ is rational in all characteristics
(Theorem~\ref{thm:rationality}).
The conjugacy classes of elements in Weyl groups of root systems were
classified by Carter~\cite{Carter}, although we use the
Atlas~\cite{ATLAS} labeling
1A, \ldots, 12C for the classes in $W(\sfE_6)$.
To each conjugacy class of $W(\sfE_6)$ we may associate the
closed subvariety of $\mathcal{M}_{\textrm{cub}}(\bbk)$
corresponding to surfaces admitting an automorphism of that class.
For example, the 1A stratum corresponds to the set of all cubic surfaces
since 1A is the class of the trivial automorphism.
This gives a stratification of $\mathcal{M}_{\textrm{cub}}(\bbk)$
by conjugacy classes of $W(\sfE_6)$.
Some of these strata will be empty (for example, 12C)
and some will coincide with others (for example, the 9A and 3C strata
coincide).
We will see that every nonempty stratum is irreducible and even
unirational (Corollary~\ref{cor:unirationalStrata}).

We see the strata for cubic surfaces in characteristic $0$ in
Figure~\ref{fig:cubicSpecialization},
where we choose only one label for each stratum.
Here an arrow $A \to B$ means that the stratum $A$ strictly contains the
stratum $B$ and those arrows implied by transitivity are omitted for clarity.
Each successive row corresponds to strata of dimension one less than the
preceding row with the 1A stratum having dimension 4 and the strata for 3C,
5A, 8A, and 12A all being single points.
See Figures~\ref{fig:cubicSpecializationP2},
\ref{fig:cubicSpecializationP3},
and \ref{fig:cubicSpecializationP5}
in the Appendix for other characteristics.

\begin{figure}[h]
\[ \xymatrix{
& 1A \ar[d] \\
& 2A \ar[dl] \ar[d] \ar[ddrr] \\
2B \ar[d] \ar[dr] &
3D \ar[dl] \ar[d] \ar[dr] \\
4B \ar[d] \ar[dr] &
6E \ar[dl] \ar[d] &
3A \ar[dl] \ar[d] &
4A \ar[dl] \ar[d] \\
5A & 3C & 12A & 8A\\
}
\]
\caption{Specialization of strata in $\mathcal{M}_{\textrm{cub}}$ when
$p \ne 2,3,5$.}
\label{fig:cubicSpecialization}
\end{figure}

We are interested in the full automorphism groups of smooth cubic surfaces.
Each automorphism group gives rise to a conjugacy class of
subgroups of $W(\sfE_6)$.
Remarkably and conveniently, the strata in this case are the same
as the strata for cyclic groups.
Thus, we can name the stratum of a conjugacy class of automorphism groups
by a cyclic subgroup that gives rise to the same stratum.
These labels may not be unique, but this is a feature when
comparing different characteristics.
For example, the 3C, 5A, 12A strata are all distinct in characteristic
$0$, but they coincide in characteristic $2$.

We can now state our main theorem:

\begin{thm} \label{thm:main}
Suppose $X$ is a smooth cubic surface in $\bbP^3$ defined over an
algebraically closed field $\bbk$ of arbitrary characteristic $p$.
Then the group of automorphisms of $X$ is one of the groups in
Table~\ref{tbl:cubicAutos} and all such groups occur.
Each row of the table corresponds to the stratum of
$\mathcal{M}_{\textrm{cub}}(\bbk)$ which admits an automorphism of the named class.
The heading ``dim'' gives the dimension of the stratum,
$\operatorname{char}(\bbk)$ gives the conditions on the characteristic,
and $\Aut(X)$ is the full automorphism group of a general
cubic surface in that stratum along with its order.
(See section~\ref{sec:groupTheory} for explanations of the group
theoretic notation.)
\end{thm}

\begin{table}[h]
\begin{center}
\renewcommand{\arraystretch}{1.3}
\begin{tabular}{|c|c|c|cc|cccccccccccccccc|}
\hline
Name & dim & $\operatorname{char}(\bbk)$ & $\Aut(X)$ & Order \\
\hline
1A & 4 & any & $1$ & 1 \\
\hline
2A & 3 & any & $2$ & 2 \\
\hline
2B & 2 & $\ne 2$ & $2^2$ & 4 \\
& & 2 & $2^4$ & 16 \\
\hline
3A & 1 & any & $\calH_3(3) \rtimes 2$ & 54 \\
\hline
3C & 0 & $\ne 2,3$ & $3^3 \rtimes \frakS_4$ & 648 \\
& & 2 & $\PSU_4(2)$ & 25920\\
\hline
3D & 2 & any & $\frakS_3$ & 6\\
\hline
4A & 1 & $\ne 2$ & $4$ & 4\\
& & 2 & $2^3 \rtimes \frakS_4$ & 192\\
\hline
4B & 1 & $\ne 2$ & $\frakS_4$ & 24 \\
& & 2 & \multicolumn{2}{c|}{\cellcolor{gray!20}(same as 4A)}\\
\hline
5A & 0 & $\ne 2,5$ & $\frakS_5$ & 120 \\
& & 2 & \multicolumn{2}{c|}{\cellcolor{gray!20}(same as 3C)}\\
\hline
6E & 1 & $\ne 2$ & $\frakS_3 \times \frakS_2$ & 12 \\
& & 2 & \multicolumn{2}{c|}{\cellcolor{gray!20}(same as 4A)}\\
\hline
8A & 0 & $\ne 2,3$ & $8$ & 8 \\
& & 3 & $\calH_3(3) \rtimes 8$ & 216 \\
\hline
12A & 0 & $\ne 2,3$ & $\calH_3(3) \rtimes 4$ & 108 \\
& & 3 & \multicolumn{2}{c|}{\cellcolor{gray!20}(same as 8A)} \\
& & 2 & \multicolumn{2}{c|}{\cellcolor{gray!20}(same as 3C)} \\
\hline
\end{tabular}
\end{center}
\caption{Automorphism groups of cubic surfaces.}
\label{tbl:cubicAutos}
\end{table}

In addition to classifying the automorphism groups and the dimensions of
their strata, we also find explicit equations for the cubic surfaces
in each class
(see the summary in Table~\ref{tbl:Normal Forms}).
The specific choices of normal forms are certainly subjective, but they all have
the nice property that the number of parameters is exactly the same as
the dimension of the corresponding stratum of the moduli space.
In particular, they show that all of the strata are unirational.

There are several approaches that one could take to carry out the
classification; for some other approaches see Remarks~\ref{rem:CohenWales},
\ref{rem:Hosoh}, and \ref{rem:Saito}.
We have taken a conceptual geometric approach;
as a consequence, we believe many of our intermediate results and their
corollaries will be of independent interest.
We discuss these in the remainder of the introduction,
while also providing a tour of the cubic surfaces and their
automorphisms.

\subsection{Lifting to characteristic 0}

Let $X$ be a smooth projective irreducible variety over an algebraically
closed field $\bbk$ of positive characteristic and $G$ a finite group of
automorphisms of $X$.
We say that a group $G$ is \emph{tame} (resp. \emph{wild}) if the
characteristic $p$ does not divide (resp. divides) the order of $G$.
One expects that the behavior of tame groups to be similar across
different characteristics.
The order of $|W(\sfE_6)|$ is $51840=2^7 \times 3^4 \times 5$, so one only expects
unusual automorphism groups of cubic surfaces in characteristics $2,3,5$.

More precisely, we say
a pair $(X,G)$ can be \emph{lifted to characteristic $0$}
if there exists a complete discrete valuation
ring $R$ of characteristic $0$ with residue field $\bbk$,
and a smooth projective $A$-scheme $\calX$ with
an action of $G$ over $A$, such that the fiber over the closed
point is $G$-equivariantly isomorphic to $X$.
When $X$ is a smooth rational surface (for example, a smooth cubic)
and $G$ is tame, the pair can be lifted to characteristic $0$
by a result of Serre (see \S{}5~of~\cite{SerreBourbaki}).

An obvious obstruction to doing this for wild groups is that some
automorphism groups in characteristic $2$ and $3$ do not act on
any cubic surface in characteristic $0$.
This is the only obstruction:

\begin{thm} \label{thm:liftingIntro}
Let $X$ be a smooth cubic surface defined over an algebraically closed field
$\bbk$ of
positive characteristic and let $G$ be a group of automorphisms of $X$.
Then $(X,G)$ can be lifted to characteristic $0$ if and only if
$G \subset W(\sfE_6)$ is realized as a group of automorphisms on some
smooth cubic surface in characteristic $0$.
\end{thm}

In positive characteristics, it is natural to ask if there
are actions of infinitesimal group schemes on cubic surfaces,
but this is impossible (Theorem~\ref{thm:infinitesimal}).

\subsection{Reflections and Eckardt points}

Recall that a \emph{reflection} of a vector space
is a linear transformation of order $2$ that fixes a hyperplane pointwise.
An automorphism $g : X \to X$ of a cubic surface $X$ is a \emph{reflection}
if $g$ is the restriction of an automorphism of $\bbP^3$ that can be
represented by a reflection in $\GL_4(\bbk)$.
Note that this definition makes sense even when the characteristic is $2$.
Reflections are of class 2A in $W(E_6)$.

An \emph{exceptional line} on a cubic surface $X$ is a line in the
ambient $\bbP^3$ that is contained in $X$. 
If three exceptional lines on $X$ all pass through the same
point, it is classically called an \emph{Eckardt point}.
There is a well-known bijective correspondence between Eckardt points
and reflections that leave $X$ invariant.

Most of the automorphism groups of cubic surfaces are generated by
reflections and can be completely described by the geometric
arrangements of their Eckardt points.
If two reflections commute, then their product has class 2B and their
corresponding Eckardt points lie on a common exceptional line.
Conversely, an involution of class 2B implies the existence of those two
Eckardt points.

If two reflections do not commute, then their product has class 3D and
the line $\ell$ between the corresponding Eckardt points does not lie inside
the cubic surface.  The third point on the $\ell$ is also an Eckardt
point and the three points generate a group isomorphic to the symmetric
group $\frakS_3$, which recovers the original automorphism of class 3D.
Similarly, geometric interpretations of the configurations of Eckardt
points exist for 3A, 3C, 4B, 5A, and 6E classes as well
(Section~\ref{sec:reflectionGroups}).

When the characteristic is not $2$, then there are at most 2 Eckardt
points on an exceptional line.
However, when $p=2$, there are exactly 5 Eckardt points on any
exceptional line as soon there is more than 1 Eckardt point.
We will see that this can be explained by the fact that the exceptional
line blows down to the \emph{canonical point} of a quartic del Pezzo surface
$Y$ described in \cite{pencils}, and thus the full automorphism group of
$Y$ must also act on $X$.
This distinction between the number of Eckardt points
is the fundamental reason why the automorphism groups are so much
larger, and why so many strata coincide,
when the characteristic is $2$.

It is well known that reflections generate a subgroup of index $1$, $2$,
or $4$ in the automorphism group of a cubic surface when $p=0$;
we show this holds in all characteristics (Corollary~\ref{cor:reflectionIndex}).
The reflections fail to generate only if there are automorphisms
with classes 4A, 8A, and 12A, which always have a power of class 2A.

\subsection{Forms of Sylvester and Emch}

A classical theorem of Sylvester states that a general cubic surface
in characteristic $0$ can be defined by
\[
\sum_{i=0}^4 x_i^3 = \sum_{i=0}^4 c_ix_i = 0
\]
in $\bbP^4$ where $c_0, \ldots, c_4$ are parameters.
When the parameters are all equal, one obtains the \emph{Clebsch
diagonal cubic surface}, which has an obvious $\frakS_5$-action.
The Clebsch cubic is the unique surface admitting an automorphism of
class 5A, which we take as its definition when $p \ne 0$.

More generally, by selectively setting certain parameters equal to one
another, we obtain families of cubic surfaces admitting automorphism
groups isomorphic to subgroups of $\frakS_5$.  In this way we obtain
general representatives of cubic surfaces realizing automorphisms of class
1A, 2A, 2B, 3D, 4B, 6E and 5A that correspond to the permutations
$()$, $(12)$, $(12)(34)$, $(123)$, $(1234)$, $(123)(12)$, and $(12345)$
respectively.

When $p=5$, the Clebsch cubic surface does not exist (Sylvester's
normal form yields a singular surface); this is essentially the only
difference between $p=5$ and higher characteristics.
When $p=3$, the cubics in Sylvester's normal form are totally singular.
When $p=2$, we will see that all cubics from Sylvester's normal form
are isomorphic to one another!

In Section~\ref{sec:generalForms}, we rectify this bad behavior
in characteristics $2$ and $3$ by instead
discussing a lesser known normal form due to Emch~\cite{Emch}.
We prove that, in all characteristics, a general cubic surface
can be defined by
\[
x_0^3 + x_1^3 + x_2^3 + x_3^3
+ c_0x_1x_2x_3 + c_1x_0x_2x_3 + c_2x_0x_1x_3 + c_3x_0x_1x_2 =0
\]
in $\bbP^3$ for parameters $c_0, \ldots, c_3$.

When these parameters are all equal to $0$ we obtain the
\emph{Fermat cubic surface}
\[
x_0^3 + x_1^3 + x_2^3 + x_3^3 = 0\ .
\]
This is the unique surface realizing an automorphism of class 3C.
When $p=3$, the surface is not reduced and the 3C stratum is empty.
When $p \ge 5$, the automorphism group is generated by permutations and
multiplying the coordinates by third roots of unity.
When $p=2$, the automorphism group of the Fermat cubic is the simple
group $\PSU_4(2)$ of order $25920$.
In fact, we will see that the possible automorphism groups of \emph{any}
cubic surface in
characteristic $2$ are closely related to the subspaces of the non-degenerate 
Hermitian space $\bbF_4^4$.

\subsection{Cyclic Surfaces}

A cubic surface $X$ is \emph{cyclic} if there exists a Galois cover of
degree $3$ over $\bbP^2$.  The deck transformation of the cover has
class $3A$.
These surfaces are discussed in more detail in Section~\ref{sec:3A}.

When $p \ne 3$, these can always be defined by the form
\begin{equation} \label{eq:cyclicIntro}
x_0^3 + x_1^3 + x_2^3 + x_3^3 + cx_0x_1x_2 = 0
\end{equation}
in $\bbP^3$ where $c$ is a parameter.
In this case, the covering map $X \to \bbP^2$ is the restriction of the
projection $\bbP^3 \to \bbP^2$ given by
\[(x_0:x_1:x_2:x_3) \to (x_0:x_1:x_2) \ . \]
The deck transformation $g$ of the cover is generated by
$x_3 \mapsto \zeta_3 x_3$ for a primitive third root of unity $\zeta_3$.
The ramification locus $C$ is a smooth curve of genus $1$
whose equation is obtained by removing the $x_3^3$ from
\eqref{eq:cyclicIntro}. 

A subgroup $G$ of the automorphism group of $X$ fits into an exact sequence
\[
1 \to \langle g \rangle \to G \to \Aut_{\bbP^2}(C) \to 1
\]
where $\Aut_{\bbP^2}(C)$ is the embedded automorphism group of
$C \subset \bbP^2$.
The group $\Aut_{\bbP^2}(C)$ is isomorphic to a semidirect product
$(\bbZ/3\bbZ)^2 \rtimes \bbZ/2\bbZ$ for a general curve of genus $1$.
This explains why $\Aut(X)$ has order $54$ for a general $X$ on the 3A
stratum.
For special values of $c$, the curve $C$ has additional
automorphisms, which give rise to additional automorphisms of $X$.
For example, if $c=0$ then one obtains the Fermat surface.
When $p \ge 5$, an additional automorphism of order $4$ occurs on the
unique surface of class 12A; and one of order $3$ occurs on
the surface of class 3C (the Fermat surface).
When $p = 2$, there is only one isomorphism class of genus $1$ curve
with a larger automorphism group: this is another explanation for why
the 12A and 3C strata coincide when $p=2$.

When $p \ne 3$, there are exactly 9 Eckardt points lying in the plane
fixed by the deck transformation $g$.  These points are the flexes
of the ramification curve $C$ and can be identified with the $3$-torsion
subgroup of a group structure on $C$.
There are 12 lines that pass through exactly 3 of these points
and each point lies on exactly 4 of these lines.
Together these points and lines form the
\emph{Hesse configuration} $(9_4\ 12_3)$.

When $p=3$, we find a different normal form for cyclic surfaces
(Proposition~\ref{prop:normalForms3Ap3}).
Here the ramification curve $C$ is a cuspidal cubic instead of an
elliptic curve.
Nevertheless, the automorphism group for the general cyclic surface is
still abstractly isomorphic to the situation when $p \ne 3$.
Again we find 9 Eckardt points in the Hesse configuration, which can be
identified with a subgroup of a group structure on $C$.
For one special surface in the 3A stratum, there is an additional
automorphism of $C$ of order 8 that stabilizes this subgroup.
This explains why the 8A and 12A strata coincide when $p=3$.

\subsection{Structure of the paper}

We fix notation and review the basics of cubic surfaces in
Section~\ref{sec:prelim}.
The following sections establish several general facts which will be
useful in carrying out the classification.
In Section~\ref{sec:DP4},
we establish some facts about del Pezzo surfaces of degree $4$
in arbitrary characteristic; in particular, we classify all the possible
automorphism groups.
In Section~\ref{sec:phenomena},
we discuss some interesting phenomena for cubic surfaces that only occur
in special characteristics --- much of this section is of
independent interest.
In Section~\ref{sec:Fermat}, we study the Fermat cubic surface,
which has an unusually large automorphism group.
In Section~\ref{sec:generalForms}, we discuss normal forms for general
cubic surfaces by Sylvester and Emch --- the latter will turn out to be
better behaved in general characteristics.
In Section~\ref{sec:rationality}, we show that the moduli space of cubic
surfaces is rational in all characteristics.

The paper then turns to the classification in earnest.
For each conjugacy class in $W(\sfE_6)$, we describe the cubic surfaces
that admit an automorphism with that class.
Sections~\ref{sec:ccs} through \ref{sec:higherOrder} go through each
possibility, describe automorphism groups acting on these surfaces, and
provide normal forms in each case.
In Section~\ref{sec:EckardtCollections}, we consider the possible
reflection groups acting on cubic surfaces, which we use
in Section~\ref{sec:proof} to finish the proof of the classification.
Finally, in Section~\ref{sec:lifting}, we prove
Theorem~\ref{thm:liftingIntro}.

\section*{Acknowledgements}
The authors would like to thank J.-P.~Serre and an anonymous referee for
several helpful comments.

\section{Preliminaries}
\label{sec:prelim}

Throughout this paper $\bbk$ will be an algebraically closed field
and $p$ will denote its characteristic.
Moreover, $X$ will be a smooth cubic surface and $\Aut(X)$ its
automorphism group.
Note that, since the usual embedding $X \hookrightarrow \bbP^3$ is the
anti-canonical embedding, we may assume $\Aut(X) \subseteq \PGL_4(\bbk)$.

It is known (see \cite{CAG}, Corollary 8.2.40, where the proof does not use the assumption on the characteristic) that for any del Pezzo surface $\calD$ of degree $\le 5$ the automorphism group $\Aut(\calD)$ has a faithful linear
representation in the orthogonal complement $K_{\calD}$ of the canonical class
in the Picard group $\Pic(\calD)$.
This defines an embedding of $\Aut(\calD)$ into the Weyl group of the root
lattice of type $\sfE_{9-d}$, where, by definition, $\sfE_5 = \sfD_5, \sfE_4 = \sfA_4$. In our case where $X$ is a del Pezzo surface of degree $3$, we have a faithful embedding of $\Aut(X)$ in $W(\sfE_6)$. 
In particular, since $|W(\sfE_6)|=51840$, the possible prime orders of
automorphisms of $X$ are $2,3$ and $5$.

It has been known for more than 150 years that a nonsingular cubic
surface contains exactly $27$ lines.
We will refer to these as \emph{exceptional lines} to differentiate them
from other lines in $\bbP^3$.
They are exceptional curves of the first kind on $X$,
i.e. smooth rational curves $E$ satisfying $E^2=-1$ and $E\cdot K_X = -1$.
We will identify them with their divisor classes in $\Pic(X)$.

A smooth cubic surface $X$ can be obtained by blowing up six points in
general position in $\bbP^2$ --- in this case, this means that no
three points are collinear and the six points do not lie on a conic.
The exceptional curves of the blow-up are six skew exceptional lines on $X$. 
Conversely, given six skew lines $E_1, \ldots, E_6$ on $X$,
we may blow them down to 6 points $p_1, \ldots, p_6$ on $\bbP^2$
(this is called a \emph{geometric marking} of $X$).
Given a geometric marking, the remaining 21 lines will be labeled by
$F_{ij}$ for the strict transforms of the lines $\overline{p_ip_j}$, and
$G_i$ for the strict transforms of the unique conic passing through the
five points excluding $p_i$.

A \emph{double-sixer} is a set of 12 lines divided into two sets of 6
skew lines such that each line is incident to exactly 5 lines in the other set.
An example is $E_1, \ldots, E_6, G_1, \ldots, G_6$.
There are exactly $36$ double-sixers in a cubic surface.

Given any two incident exceptional lines, there is a unique third exceptional
line incident to both.
We call such a collection a \emph{tritangent trio} --- there are exactly
45 tritangent trios on a cubic surface.
A \emph{tritangent plane} is a hyperplane section of $X$ which consists
of three lines, which are a tritangent trio.
If the three exceptional lines are concurrent (in other words, they
share a common point), then the common point is called an \emph{Eckardt
point.}

We will see in Theorem~\ref{thm:Eckardt} below that Eckardt points are
in bijection with reflections in the automorphism group of $X$.
In Lemma~\ref{lem:excLineEck} below, we will find that when $p \ne 2$ there are
either 0, 1, or 2 Eckardt points on an exceptional line; but when
$p = 2$ there are 0, 1 or 5 Eckardt points on an exceptional line.
A \emph{trihedral line} is a line $\ell$ which is \emph{not} contained
in $X$ and contains exactly three Eckardt points.
We will see in Corollary~\ref{cor:triadLineEck}
that, if a line $\ell$ contains two Eckardt
points and does not lie on $X$, then $\ell$ must contain a third Eckardt
point.

Below, we will see that Eckardt points are in bijection with
automorphisms of class 2A, pairs of Eckardt points sharing a common
exceptional line are in bijection with automorphisms of class 2B,
and trihedral lines are in bijection with cyclic groups generated by an
automorphism of class 3D.
Moreover, many of the other automorphism classes correspond to certain 
geometric arrangements of these objects.

\subsection{Group Theoretic notation}
\label{sec:groupTheory}

Throughout, we use the following notation,
some of which is borrowed from \cite{ATLAS}:
\begin{itemize}
\item $A \rtimes B$ is a semidirect product where $A$ is normal.
\item $n^m$ is the abelian group $(\bbZ/n\bbZ)^m$.
\item $\frakS_n$ is the symmetric group on $n$ letters.
\item $\frakA_n$ is the alternating group on $n$ letters.
\item $D_{2n}$ the dihedral group of order $2n$.
\item $\calH_n(q)$ is the Heisenberg group of $n\times n$ upper-triangular
matrices with coefficients in the field of $q$ elements where all
diagonal entries are $1$.
\item $\PGL_n(q)$ is the projective general linear group over a
vector space of dimension $n$ with $q$ elements.
\item $\PSL_n(q)$ is the projective special linear group over a
vector space of dimension $n$ with $q$ elements.
\item $\PSU_n(q)$ is the projective special unitary group over a
Hermitian vector space of dimension $n$ with $q^2$ elements.
\end{itemize}

We will also need the following well-known classification:

\begin{thm} \label{thm:suz}
Let $G$ be a finite subgroup of $\PGL_2(\bbk)$ where $\bbk$ is an
algebraically closed field of characteristic $p$.
Then $G$ is isomorphic to one of the following groups
\begin{enumerate}
\item $G \cong D_{2n}$, a dihedral group of order $2n$
where $n > 1$ is coprime to $p$.
\item $G \cong \frakA_4$, the alternating group on $4$ letters.
\item $G \cong \frakS_4$, the symmetric group on $4$ letters; provided $p \ne 2$.
\item $G \cong \frakA_5$, the alternating group on $5$ letters.
\item $G \simeq \PSL_2(p^n)$ for some $n>0$; provided $p > 0$.
\item $G \simeq \PGL_2(p^n)$ for some $n>0$; provided $p > 0$.
\item $G \simeq A \rtimes \mu_n$ where $n$ is coprime to $p$,
$\mu_n$ is the group of $n$th roots of unity of $\bbk$,
and the group $A$ is a $\mu_n$-stable
finite subgroup of the additive group of $\bbk$.
\end{enumerate}
Moreover, all of these possibilities occur.
\end{thm}

\begin{proof}
The classification in characteristic $0$ is due to F.~Klein
\cite{Klein} and coincides with the classification of
finite subgroups of $\SO(3)$ that can be found in many textbooks
(e.g. Theorem~9.1~of~\cite{Artin}). 
The classification in positive characteristic
is essentially due to L.~Dickson.
We recover it as follows.
Since $\Bbbk$ is algebraically closed, the groups
$\PGL_2(\Bbbk)$ and $\PSL_2(\Bbbk)$ coincide,
so it suffices to find subgroups of the latter.
Note that if the lift to $\GL_2(\bbk)$ is reducible, then it must
be upper triangular in some basis, and thus is described by the final case
in the statement of the theorem.
Otherwise, by \cite{Winter} we may assume that $G$ is contained in
$\PSL_2(q)$ for some $q=p^n$.
From Theorems~3.6.25 and 3.6.26 of~\cite{Suzuki}, we obtain all the
remaining cases.
\end{proof}

\begin{remark}
Groups from the last case (7) are always conjugate to a subgroup of the
affine group
\[
(x:y) \mapsto (mx+ay:y)
\]
for $m \in \bbk^\times$ and $a \in \bbk$.
This case contains cyclic groups of order coprime to $p$
(in particular, all finite cyclic groups when $p=0$).
This also contains all finite elementary abelian $p$-groups for $p>0$
when $n=1$.
Note that conjugacy is a delicate question here as, even up to
conjugacy, there are infinitely many elementary abelian $p$-subgroups
of the additive group of an infinite field of positive characteristic.
\end{remark}

\begin{remark}
A group isomorphic to $2^2$ occurs in all characteristics:
it is the degenerate dihedral group $D_4$ when $p \ne 2$ and is an
elementary abelian $2$-group when $p=2$.
Individual groups may occur multiple times in the above list due to
exceptional isomorphisms.
For example, $A_5 \cong \PSL_2(4) \cong \PSL_2(5)$,
and $\PSL_2(\bbF) \cong \PGL_2(\bbF)$ when $\cha(\bbF)=2$.
\end{remark}

\subsection{Groups acting on del Pezzo surfaces}

We know that the automorphism group $\Aut(X)$ of a del Pezzo surface of degree $d\le 5$ is a finite group $G$ that acts faithfully on $\Pic(X)$ preserving the intersection product. It is known that  the orthogonal complement $L$ of the canonical class $K_X$ equipped with the integral quadratic form defined by the intersection product is isomorphic to the root lattice of type $A_4, D_5,E_6,E_7,E_8$ if $d = 5,4,3,2,1$, respectively. The induced action of $G$ on $L$ is faithful and identifies $G$ with a subgroup of the Weil group $W(L)$ of the lattice of the corresponding type. In our case the lattice $L$ is the root lattice of type $E_6$. All of these facts can be found, for example, in \cite{CAG}, Chapter 8 and their proof there does not use any assumption on the characteristic.  We will identify elements $g\in \Aut(X)$ with the corresponding elements $g^{*}\in W(E_6)$.

 The conjugacy classes of elements of finite order in $W(L)$ can be
classified \cite{Carter}.  Table \ref{tableconj}, at the end of the
article, contains the classification of conjugacy classes of elements in
$W(E_6)$ that can be found in loc.~cit., in \cite{ATLAS}, or in
\cite{Manin}. We also include information about the characteristic
polynomial of the action of an element $w\in W(E_6)$ on the root space
(here $\Phi_d$ is the cyclotomic polynomial of degree $d$).
In particular, we include the trace of each element.
The fourth column gives the order of the centralizer subgroup of the conjugacy class.

\subsection{Infinitesimal Group Schemes}

\begin{thm} \label{thm:infinitesimal}
Let $X$ be a del Pezzo surface of degree $\le 5$.  There does not exist an
infinitesimal group scheme with a faithful action on $X$.
\end{thm}

\begin{proof}
Since the tangent space of such a group scheme is a linear subspace of
the space of global sections of the tangent bundle $\Theta_{X/\Bbbk}$ of
$X$, it is enough to show that $X$ has no non-zero regular vector
fields.
Following a suggestion of J.-P. Serre, we deduce it for del
Pezzo surfaces of high degree from the following well-known fact:

\smallskip

\noindent
{\bf Fact}: \emph{Let $\pi:Y'\to Y$ be the blow-up of a point $y$ on a
smooth algebraic variety $Y$ over $\Bbbk$.
Then a regular vector field on $Y$ lifts to a regular vector field on
$Y'$ if and only if it vanishes at $y$}. 

\smallskip

A coordinate-free proof of this fact can be given along the following lines. The exact sequence of sheaves of differentials
$$0\to \pi^*\Omega_{Y/\Bbbk}^1\to \Omega_{Y'/\Bbbk}^1\to \Omega_{Y'/Y}^1\to 0$$
defines, after passing to the duals, the exact sequence
\beq\label{theta}
0\to \Theta_{Y'/\Bbbk}\to \pi^*\Theta_{Y/\Bbbk}\to \mathcal{E}xt_{\calO_{Y'}}^1(\Omega_{Y'/Y}^1,\calO_{Y'})\to 0.
\eeq
We have $\Omega_{Y'/Y}^1 = i_*\Omega_{E/\Bbbk}^1$, where
$i:E\hookrightarrow Y'$ is the closed embedding of the exceptional
divisor $E$.
The exact sequence
\[0\to \mathcal{I}_E/\mathcal{I}_E^2\to \Omega_{Y'/\Bbbk}^1\otimes
\calO_E\to \Omega_{E/\Bbbk}^1\to 0\]
and the fundamental local isomorphism  formula $\omega_E\cong
\mathcal{E}xt_{\calO_{Y'}}^1(\calO_E,\omega_{Y'})$ obtained by applying the functor $\mathcal{H}om(-, \omega_{Y'})$ to the exact sequence 
$0\to \calO_{Y'}(-E)\to \calO_{Y'}\to \calO_{E} \to 0$ and using the adjunction formula, shows that the sheaf 
$\mathcal{E}xt_{\calO_{Y'}}^1(\Omega_{Y'/Y}^1,\calO_{Y'})$ is a subsheaf of 
\[\mathcal{E}xt_{\calO_{Y'}}^1(\Omega_{Y'/\Bbbk}^1\otimes \calO_E,\calO_{Y'})\cong 
\Theta_{Y'/\Bbbk}\otimes \omega_{E}\otimes \omega_{Y'}^{-1} \cong
\Theta_{Y'/\Bbbk}(E)\otimes \calO_E = \Theta_{Y'/\Bbbk}\otimes
\pi^*(T_{Y,y})\]
where $T_{Y,y}$ is the Zariski tangent space of $Y$ at $y$.
Taking the global sections in the exact sequence \eqref{theta},
we obtain the exact sequence
\[0\to H^0(Y',\Theta_{Y'/\Bbbk})\to H^0(Y',\pi^*\Theta_{Y/\Bbbk})\to
H^0(Y',\Theta_{Y'/\Bbbk}\otimes \pi^*(T_{Y,y}))\]
that shows that a vector field on $Y'$ is obtained as a lift  of a vector field on $Y$ vanishing at the point $y$. 

\smallskip

Returning to our case, we use that $X$ is isomorphic to the blow-up of
$k\ge 4$ distinct points on $\bbP^2$. So, it is enough to show in the
case $k = 4$; in other words, for del Pezzo surfaces of degree $5$. We
may assume that the points have coordinates $(1:0:0),(0:1,0),(0:0:1),
(1:1:1)$. Since $H^0(\bbP^2,\Theta_{\bbP^2/\Bbbk})$ is isomorphic to the
Lie algebra of ${\rm PGL}_3(\Bbbk)$ and vanishing at one point gives two
conditions, we see by direct computation that there are no non-zero
vector fields vanishing at the four points. 
\end{proof}

\begin{remark}
In fact, the proof shows that to obtain a non-trivial regular vector
field starting from $\bbP^2$, one has to blow-up either less than 3
points or blow-up infinitely near points. In particular, to construct a
smooth rational projective surface with infinitesimal automorphism group
one has to blow-up points in such a way that the automorphism group of
the surface is trivial but there exists a non-zero regular vector field.
An example of such a surface in characteristic $p = 2$ can be found in
\cite{Neuman}. Here one can find also a direct proof of the Fact in the
case of surfaces (loc.cit. Lemma (3.1)).

A regular vector field $D$ on a smooth projective algebraic varety over
an algebraically closed field of characteristic $p > 0$ is called
$p$-closed if $D^p = aD$, where $a\in \Bbbk$.  There is a natural
bijective correspondence between $p$-closed regular fields and finite
infinitesimal group schemes of order $p$ isomorphic to $\mu_p$ if $a\ne
0$ or $\alpha_p$ if $a = 0$. Also, if a variety admits a regular vector
field, then it also admits a $p$-closed vector field. It would be very
interesting to classify all finite infinitesimal group schemes that may
act on a rational surface.
\end{remark}

\section{Del Pezzo surfaces of degree 4}
\label{sec:DP4}

A smooth cubic surface is isomorphic to the blow-up of one point on a
del Pezzo surface of degree $4$ that does not lie on any of the $16$ lines
contained within (in its anti-canonical embedding as a quartic surface in
$\bbP^4$).
Since we will use this several times,
we will classify automorphism groups of del Pezzo surfaces of degree $4$
before moving to cubic surfaces.

When $p \ne 2$, the story is not much different than the well-known case of
$p=0$ discussed in, for example, \S{}8.6~of~\cite{CAG}.
The case of $p=2$ extends the work in \cite{pencils}. 

A del Pezzo surface $Y$ of degree $4$ is isomorphic to a complete intersection of two quadrics $V(q_1)\cap V(q_2)$ in $\bbP^4$. If $p\ne 2$, the equations can be reduced to the form
\begin{align}\label{delpezzo4-1}
q_1&= x_1^2+x_2^2+x_3^2+bx_4^2 = 0,\\ \notag
q_2& = x_0^2+x_2^2+ax_3^2+x_4^2 =0,\notag
\end{align}
otherwise
\begin{align}\label{delpezzo4-2}
q_1 &= (ab+b+1)x_2^2 + a x_3^2 + x_2y_1 + x_3y_2 = 0 \\ \notag
q_2 &= b x_1^2 + (ab+a+1)x_2^2 + x_1y_1 + x_2y_2 = 0. \notag
\end{align}
where $a$ and $b$ are parameters.
In both cases 
\begin{equation}\label{roots}
\{ (1:0), (0:1), (1:-1), (-a:1), (1:-b) \}\subset \bbP^1(\Bbbk)
\end{equation}
are the roots of a degree five binary form  $\Delta$. They correspond to parameters $(\lambda:\mu)\in \bbP^1$ such that the quadric $V(\lambda q_1+ \mu q_2)$ is singular.

In the case $p\ne 2$, the equations have an obvious symmetry defined by
negations of the projective coordinates.
They generate a subgroup of $\Aut(Y)$ isomorphic to an elementary
abelian group $2^4$.
One can choose generators to be reflections that change the sign of only one
coordinate. In the case $p = 2$, the surface $Y$ has a similar but less obvious
group of symmetries.
The subgroup $2^4$ of $\Aut(Y)$ in either case is generated by
reflections in $\Bbbk^5$ defined by
\[
\rho_i(v) = v-\frac{b_{\lambda q_1+\mu q_2}(z_i,v)}{(\lambda q_1+\mu q_2)(z_i)}z_i,
\]
where $z_i$
are vectors representing the singular points of quadrics in the pencil
$\lambda q_1+\mu q_2 = 0$ and $b_{\lambda q_1+\mu q_2}$ is the
associated polar symmetric bilinear form.
Here $(\lambda:\mu)$ must be chosen in such a way that the corresponding
quadric does not vanish at any of the points $z_i$;
the transformation does not depend on this choice.

The group of automorphisms $\Aut(Y)$ acts on the coordinates $(\lambda:\mu)$
of the pencil of quadrics leaving the polynomial $\Delta$ invariant.
In this way we realize $\Aut(Y)$ as the semi-direct product
\[ \Aut(Y) \cong 2^4\rtimes  G, \]
where $G$ is a finite subgroup of $\PGL_2(\bbk)$ that leaves the zeros of the binary form $\Delta$ invariant.
Note that this agrees with the known structure of the Weyl group $W(\sfD_5)\cong 2^4\rtimes \frakS_5$.
The group $G$ corresponds to the subgroup of $\frakS_5$ that permutes the roots of $\Delta$.

\begin{thm} \label{thm:dp4autos}
The coarse moduli space of del Pezzo surfaces of degree $4$ is
isomorphic to the coarse moduli space of $5$ distinct points in $\bbP^1$.
The automorphism group of a given surface $Y$
is isomorphic to $2^4 \rtimes G$ where $G$ is the subgroup of
$\PGL_2(\bbk)$ leaving invariant the corresponding set of $5$ points.
The possible groups are enumerated in Table~\ref{tbl:5pointAutos}. 
\end{thm}

\begin{table}[h]
\begin{center}
\begin{tabular}{|c|c|c|c|c|}
\hline
$p$ & $G$ & $(a,b)$ & $\lambda$ \\
\hline
$\ne 2,3,5$ & 1 & {\tiny general} & $(1,1,1,1,1)$ \\
 & $2$ & $a=b$ & $(2,2,1)$ \\
 & $4$ & $(i,i)$ & $(4,1)$ \\
 & $\frakS_3$ & $(\zeta,\zeta)$ & $(3,2)$ \\
 & $D_{10}$ & $(\phi,\phi)$ & $(5)$ \\
\hline
2 & 1 & {\tiny general} & $(1,1,1,1,1)$ \\
 & $2^2$ & $a=b$ & $(4,1)$ \\
 & $\frakA_5$ & $(\zeta,\zeta)$ & $(5)$ \\
\hline
3 & 1 & {\tiny general} & $(1,1,1,1,1)$ \\
 & $2$ & $a=b$ & $(2,2,1)$\\
 & $4$ & $(i,i)$ & $(4,1)$ \\
 & $D_{10}$ & $(\phi,\phi)$ & $(5)$ \\
\hline
5 & 1 & {\tiny general} & $(1,1,1,1,1)$ \\
 & $2$ & $(a,a)$ & $(2,2,1)$ \\
 & $\frakS_3$ & $(\zeta,\zeta)$ & $(3,2)$ \\
 & $5\rtimes 4$ & $(i,i)$ & $(5)$ \\
\hline
\end{tabular}
\end{center}
\caption{Automorphisms of $5$ distinct points in $\bbP^1$.}
\label{tbl:5pointAutos}
\end{table}

\begin{proof}
If $p = 0$, the possible groups $G$ and corresponding equations of
$Y$ are listed in \cite{CAG}, 8.6.4.
We will extend these results using a more uniform approach that works
for all characteristics.

Let us explain the remaining columns of Table~\ref{tbl:5pointAutos}.
The $(a,b)$ column give parameters for \eqref{delpezzo4-1} or
\eqref{delpezzo4-2} realizing the surface $Y$ on which $G$ acts,
while the $\lambda$ column lists the sizes of the orbits of
$G$ on the $5$ points.

Let $\zeta$, $i$ and $\phi$ be elements of $\bbk$ which are solutions to
the equations
$\zeta^2+\zeta+1=0$, $i^2+1=0$, and $\phi^2-\phi-1$.
When $p=0$, note $\zeta$ is a primitive cube root of unity,
$i$ is a primitive $4$th root of unity, and 
$\phi=\frac{1+\sqrt{5}}{2}$ is the Golden ratio.
When $p=2$, we may assume $\zeta=\phi$ and note that $i=1$ is not permitted
for $a$ or $b$ in \eqref{delpezzo4-2} since the 5 points would not be
distinct.
When $p=3$, we have $\zeta=1$ which is invalid for $a$ or $b$ in
\eqref{delpezzo4-1}.
When $p=5$, we have $\phi=3$ and $i=2$ or $i=3$.

Consider the following matrices:
\begin{equation*}
g_2 := \begin{pmatrix} 0 & 1 \\ 1 & 0\end{pmatrix},\
g_3 := \begin{pmatrix} \zeta & 0 \\ 0 & \zeta^{-1}\end{pmatrix},\
g_4 := \begin{pmatrix} \frac{1+i}{2} & \frac{1-i}{2} \\ \frac{1-i}{2} &
\frac{1+i}{2} \end{pmatrix},\
g_5 := \begin{pmatrix} -\phi & -1 \\ 1 & 0 \end{pmatrix}.
\end{equation*}
With the exception of $g_4$ when $p=2$ and $g_3$ when $p=3$,
each matrix $g_n$ is well-defined and has order $n$ in $\PGL_2(\bbk)$.

We will explicitly describe the groups $G$ that show up in the table.
When $a=b$, the matrix $g_2$ leaves invariant the $5$ points from
\eqref{roots}.  If $p=2$ and $a=b$, then the matrix
\[
g_2' = \begin{pmatrix} a & 1 \\ 1 & a \end{pmatrix}
\]
also leaves invariant the $5$ points.
Thus when $a=b$, we have exhibited the group
$G=\langle g_2 \rangle \cong 2$ when $p \ne 2$
and exhibited $G = \langle g_2, g_2' \rangle \cong 2^2$ when $p=2$.
If $a=b=\zeta$ and $p \ne 3$, then
$\frakS_3 \cong\langle g_2, g_3 \rangle \subseteq G$.
If $a=b=i$ and $p \ne 2$, then $4 \cong \langle g_4 \rangle \subseteq G$.
If $a=b=\phi$, then $D_{10} \cong \langle g_2, g_5 \rangle \subseteq G$.
If $p=2$, then $\zeta=\phi$ and so in the case $a=b=\zeta=\phi$
we obtain $\frakA_5 \cong \langle g_3, g_5 \rangle \subseteq G$.
If $p=5$, then in the case where $a=b=i$ 
we obtain $5 \rtimes 4 \cong \langle g_4, g_5 \rangle \subseteq G$;
note that $\phi$ is equal to one of the two possible values of $i$.

It remains to demonstrate that these are the only subgroups that occur.
First, observe that, in any characteristic, there is no non-trivial
automorphism of $\bbP^1$ which fixes three or more points.
Thus, as a subgroup of $\frakS_5$, the group $G$ cannot contain any
transpositions.
Up to conjugacy, these are the only subgroups of
$\frakS_5$ containing no transpositions:
\begin{equation} \label{eq:goodGroups}
1,\ 2,\ 3,\ 4,\ 2^2,\ 5,\ \frakS_3,\ D_{10},\ 5 \rtimes 4,\ \frakA_5 .
\end{equation}
(Note there are other subgroups isomorphic to $2$, $2^2$ and
$\frakS_3$ which \emph{do} have transpositions.)

This means that $G$ contains elements only of orders $1$, $2$, $3$, $4$
and $5$; and that an element of order $2$ must act as a double
transposition in $\frakS_5$.
In characteristic $2$, there are no elements of order $4$ in
$\PGL_2(\bbk)$.
In characteristic $3$, any element of order $3$ can fix at most one
point on $\bbP^1$, so $5$ points cannot be a disjoint union of orbits.
Since cyclic subgroups of fixed order are all conjugate in
$\PGL_2(\bbk)$ (consider diagonalizations or Jordan canonical forms),
we conclude that every non-trivial cyclic subgroup is generated by one
of $g_2, g_3, g_4, g_5$ up to conjugacy. 

Moreover, we claim that up to choice of coordinates the $5$ points can
be put into the form \eqref{roots} with one of the matrices
$g_2, g_3, g_4, g_5$ given above.
For order $2$, there is only one fixed point which we may
assume is $-1$; up to scaling, we may assume one of the orbits is
$\{0,\infty\}$ and so the other must be $\{-a,-a^{-1}\}$ for some choice of
$a$.
For the other orders, we only have to show that the group and set of
invariant points is unique up to conjugacy; that they can be put in form
\eqref{roots} then follows.
For order $3$, we may assume $p \ne 3$ and the group acts via
$x \mapsto \zeta x$, so the two fixed points are $0$ and $\infty$;
scaling $x$ commutes with the group action, so we may assume the
non-trivial orbit contains $1$ and so also $\zeta, \zeta^2$.
For order $4$, we may assume $p \ne 2$ and the group acts via
$x \mapsto ix$ fixing either $0$ or $\infty$; scaling $x$ commutes with
the group action and so we may assume the non-trivial orbit is
$\{1,-1,i,-i\}$.  The change of coordinates $x \to x^{-1}$ normalizes
the group action and leaves invariant the non-trivial orbit
$\{1,-1,i,-i\}$ and so we may assume the fixed point is $0$.
For order $5$, if we assume $p \ne 5$, then the group acts diagonally
and we may scale to ensure the orbit consists of the 5 fifth roots of unity.
If $p=5$, then we may assume the group acts via $x \mapsto x+1$
and, possibly after translation, the orbit is $\{0,1,2,3,4\}$.

Due to our choice of coordinates for the $5$ points and $g_2, g_3, g_4,
g_5$ above, we see that the presence of an element of order $3$ forces
$G$ to contain $\frakS_3$.
Using similar arguments, one concludes that the rows of
Table~\ref{tbl:5pointAutos} are as desired.
It only remains to exclude the following possibilities:
$2^2$ and $\frakA_5$ when $p \ne 2$, and
$5 \rtimes 4$ when $p \ne 5$.
Observe that $2^2$ must fix one of the $5$ points;
when $p \ne 2$,
the group $2^2$ must act faithfully on the tangent space at the fixed
point, which is impossible.
The group $\frakA_5$ contains a subgroup isomorphic to $2^2$ and so it
also does not occur when $p \ne 2$.
Finally, the group $5 \rtimes 4$ acts on $\bbP^1$ only when $p=5$ by
Theorem~\ref{thm:suz}.
\end{proof}

Let $Y$ be a del Pezzo surface of degree $4$.
An \emph{exceptional line} of $Y$ is a line of the ambient space $\bbP^4$ that is
contained in $Y$.  Similarly to cubic surfaces these are precisely the
$(-1)$-curves of $Y$. 
It is well known that there are exactly 16 exceptional lines
(see, for example, Theorem~26.2~of~\cite{Manin}).
The surface $Y$ can be obtained by blowing up 5 points in general
position on the plane, which always lie on a unique conic.
Indeed, del Pezzo surfaces of degree $4$ are determined up to
isomorphism by a set of $5$ distinct points on $\bbP^1$.

As we stated earlier, the automorphism group $\Aut(Y)$ has an injection
into the Weyl group $W(\sfD_5) \cong 2^4\rtimes \frakS_5$.
One representation of $W(\sfD_5)$ is an action on $\bbZ^5$ where the group
$\frakS_5$ acts by permuting basis elements and the group
$2^4$ acts by multiplying an even number of coordinates by $-1$.
For $A \subseteq \{1,2,3,4,5\}$, we will use the notation $\iota_{A}$
to denote the involution multiplying the corresponding coordinates by $-1$.
When $A$ has $4$ (resp. $2$) elements, we say $\iota_{A}$ is an \emph{involution of
the first} (resp. \emph{second}) \emph{kind}.
The group $2^4$ is realized on every del Pezzo surface of
degree $4$ and the group $2^4$ acts simply transitively on the set of 16
lines.

We recall from \cite{Carter} that
conjugacy classes of elements of the group $W(\sfD_5)$ correspond to
\emph{signed cycle types}.
They are listed in Table~\ref{tbl:D5ccs}.
The trace and characteristic polynomial are for the standard
$5$-dimensional representation.
The column ``realizable'' indicates whether, and in what
characteristics, the corresponding class can be realized on a del Pezzo
surface of degree $4$ --- this is obtained from
Theorem~\ref{thm:dp4autos}.
The class in $W(\sfE_6)$ indicates the corresponding conjugacy class
under the standard embedding $\sfD_5 \hookrightarrow \sfE_6$;
this is obtained by comparing characteristic polynomials with those from
Table~\ref{tbl:cyclicWE6} below.
Elements in the conjugacy class $\bar{1}\bar{1}\bar{1}\bar{1}1$
are involutions of the first kind,
while elements in the class $\bar{1}\bar{1}111$ are involutions of
the second kind.

\begin{table}
\begin{center}
\begin{tabular}{|c|c|r|l|c|l|}
\hline
Order & Name & Trace & Char. Poly & Realizable & Class in $W(\sfE_6)$\\
\hline
1 & $11111$ & $5$ & $\Phi_{1}^{5}$ & yes & 1A \\
\hline
2 & $111\bar{1}\bar{1}$ & $1$ & $\Phi_{1}^{3}\Phi_{2}^{2}$ & yes & 2B \\
2 & $1\bar{1}\bar{1}\bar{1}\bar{1}$ & $-3$ & $\Phi_{1}^{1}\Phi_{2}^{4}$ & yes & 2A \\
2 & $2111$ & $3$ & $\Phi_{1}^{4}\Phi_{2}^{1}$ &  & 2C \\
2 & $21\bar{1}\bar{1}$ & $-1$ & $\Phi_{1}^{2}\Phi_{2}^{3}$ &  & 2D \\
2 & $221$ & $1$ & $\Phi_{1}^{3}\Phi_{2}^{2}$ & yes & 2B \\
\hline
3 & $311$ & $2$ & $\Phi_{1}^{3}\Phi_{3}^{1}$ & $p \ne 3$ & 3C \\
\hline
4 & $2\bar{2}\bar{1}$ & $-1$ & $\Phi_{1}^{1}\Phi_{2}^{2}\Phi_{4}^{1}$ & yes & 4B \\
4 & $41$ & $1$ & $\Phi_{1}^{2}\Phi_{2}^{1}\Phi_{4}^{1}$ & $p \ne 2$ & 4D \\
4 & $\bar{2}11\bar{1}$ & $1$ & $\Phi_{1}^{2}\Phi_{2}^{1}\Phi_{4}^{1}$ &  & 4D \\
4 & $\bar{2}\bar{1}\bar{1}\bar{1}$ & $-3$ & $\Phi_{2}^{3}\Phi_{4}^{1}$ &  & 4C \\
4 & $\bar{2}\bar{2}1$ & $1$ & $\Phi_{1}^{1}\Phi_{4}^{2}$ & yes & 4A \\
\hline
5 & $5$ & $0$ & $\Phi_{1}^{1}\Phi_{5}^{1}$ & yes & 5A \\
\hline
6 & $32$ & $0$ & $\Phi_{1}^{2}\Phi_{2}^{1}\Phi_{3}^{1}$ &  & 6G \\
6 & $3\bar{1}\bar{1}$ & $-2$ & $\Phi_{1}^{1}\Phi_{2}^{2}\Phi_{3}^{1}$ & $p \ne 3$ & 6F \\
6 & $\bar{3}1\bar{1}$ & $0$ & $\Phi_{1}^{1}\Phi_{2}^{2}\Phi_{6}^{1}$ & $p \ne 3$ & 6C \\
\hline
8 & $\bar{4}\bar{1}$ & $-1$ & $\Phi_{2}^{1}\Phi_{8}^{1}$ & $p \ne 2$ & 8A \\
\hline
12 & $\bar{3}\bar{2}$ & $0$ & $\Phi_{2}^{1}\Phi_{4}^{1}\Phi_{6}^{1}$ &  & 12C \\
\hline
\end{tabular}
\end{center}
\caption{Conjugacy classes in $W(\sfD_5)$.}
\label{tbl:D5ccs}
\end{table}

It will be useful to explicitly describe the action of $W(\sfD_5)$
on the $16$ lines.
Let $E_1, \ldots, E_5$ be $5$ skew lines obtained by blowing up $5$
points $p_1, \ldots, p_5$ on the plane model; let $F_{ij}$ be the strict
transforms of the lines $\overline{p_ip_j}$ and let $G$ be the strict
transform of the unique conic containing all $5$ points (this would be
$G_6$ in the cubic surface).

One identifies the $16$ lines with sequences in $\bbF_2^5$ with an odd
number of non-zero entries.
The lines $E_i$ are identified with the standard basis
vectors.
The lines $F_{ij}$ are identified with the vectors which have
$0$'s in the $i$th and $j$th spot, but $1$'s elsewhere.
The line $G$ is identified with the vector $(1,1,1,1,1)$.
Two lines are incident if their sum has $4$ ones and skew if their sum
has $2$ ones.

The group $\frakS_5$ acts on $\bbF_2^5$ by permuting the coordinates.
The normal subgroup $2^4$ is identified with the subgroup of $\bbF_2^5$
consisting of an even number of non-zero entries.
The involutions of the first kind correspond to $4$ ones, while the
involutions of the second kind correspond to $2$ ones.
The action of $2^4$ on the 16 lines is simply by addition in $\bbF_2^5$.

This makes computations quite transparent.  For example, we now
immediately see that
involutions of the first kind partition the $16$ lines into $8$ orbits of
incident lines, while involutions of the second kind partition the lines
into $8$ orbits of skew lines.

The classification of possible groups of automorphisms of a quartic del Pezzo surface shows that  the conjugacy classes
$2111$,
$21\bar{1}\bar{1}$,
$\bar{2}11\bar{1}$ and
$\bar{2}\bar{1}\bar{1}\bar{1}$
are never realized geometrically on a del Pezzo surface of degree $4$.
Similarly, the realizability of the other classes are obstructed or
permitted, accordingly.

We will need a fairly detailed understanding of the fixed loci of
involutions of the first and second kind. First, note that in the case
$p = 2$, substituting $y_1 = y_2 = 0$ in the equations \eqref{delpezzo4-2}
we get a point 
\[
p_0 = ( a\sqrt{b}+a+\sqrt{a} : \sqrt{ab} : b\sqrt{a}+b+\sqrt{b} : 0 : 0 ) \ .
\] 
This point is called the \emph{canonical point} of $Y$ in \cite{pencils}.
It is independent of the choice of projective coordinates. 

\begin{prop} \label{prop:inv12kind}
If $p \ne 2$, the fixed loci of the 5 involutions of the first kind form 5
elliptic curves; the curves in each pair meet at 4 distinct points, which are
the fixed locus of an involution of the second kind.

If $p=2$, the fixed locus of an involution of the second kind is the
canonical point.
For an involution $\tau$ of the first kind there are two
possibilities:
\begin{itemize}
\item There exists an involution not of the first or second kind commuting
with $\tau$.  In this case, the fixed locus of $\tau$ is a pair
of smooth rational curves that are tangent to one another at the
canonical point.
\item Otherwise, the fixed locus is a rational curve with a unique
cuspidal singularity at the canonical point.
\end{itemize}
\end{prop}

\begin{proof}
For $p \ne 2$, this is easy to see from the diagonal normal form
\eqref{delpezzo4-1}.  Involutions of the first kind correspond to
changing the sign of only one coordinate, say $x_4$.  Thus the fixed
locus in $\bbP^4$ consists of the point $(0:0:0:0:1)$ and the hyperplane
$x_4=0$.  The point does not lie on $Y$ and the hyperplane cuts out a
smooth complete intersection of quadrics in $\bbP^3$, which is an
elliptic curve.  An involution of the second kind corresponds to
changing the signs of exactly two coordinates.  This fixes a line and a
plane in the ambient $\bbP^4$.  The line does not intersect any point of
$Y$, while the plane cuts out exactly $4$ distinct points.

Let us investigate the case where $p =2$ in more detail. 

The reflection $r_1$ corresponding to $(1:0)$ is given by
$r_1(x_3)=x_3+\frac{1}{a}y_2$ and acts as the identity on the other
coordinates.  The reflection $r_2$ corresponding to $(0:1)$ acts
similarly via $r_2(x_1)=x_1+\frac{1}{b}y_1$.
The reflection $r_3$ corresponding to $(1:1)$ acts via
\begin{align*}
r_3(x_i)&=x_i+\frac{1}{c}(y_1+y_2) \text{ for } i=1,2,3 \\
r_3(y_j)&=y_j \text{ for } j=1,2
\end{align*}
where $c=ab+a+b+1$.

Let $H$ be the fixed point locus of the reflection $r_3$; it is defined
by $y_1+y_2=0$.  Let $C = H \cap Y$.
The scheme $C$ is integral if and only if every member of the restricted
pencil is an irreducible conic.
The general element of the pencil $(\lambda q_1 + \mu q_2)|_H$
when restricted to $H$ is
\[
\mu b x_1^2 +
((\lambda +\mu)(ab+1) + \lambda b + \mu a)x_2^2
+ \lambda a x_3^2 +
(\mu x_1 + (\lambda+\mu)x_2 + \lambda x_3)y_1 \ .
\]
If this is reducible then it has the form
\[
(\mu x_1 + (\lambda+\mu)x_2 + \lambda x_3)(Sx_1+ Rx_2 + Tx_3+y_1) \ .
\]
Immediately, we see that $S=b$ and $T=a$ by considering the coefficients
of $x_1^2$ and $x_3^2$;
and thus $b\lambda=a\mu$ by considering the coefficient of the $x_1x_3$
term.
Note $a\ne 0,1$ and $b \ne 0,1$ or else the half-discriminant $\Delta$ has
multiple zeroes.  Thus $\lambda, \mu \ne 0$.
From the $x_1x_2$ term we conclude that $R=a+b$.
From the $x_2^2$ term we obtain $(\lambda+\mu)(ab+a+b+1)=0$,
or equivalently, $(a+b)(a+1)(b+1)=0$.
Since $a, b \ne 1$, we conclude that
$C$ is reducible if and only if $a=b$.
Note that, from the description of automorphisms above, $a=b$ if and
only if there exists a non-trivial involution commuting with $\tau$ not
of the first or second kind.

If $C$ is reducible, then there are two components;
one is contained in the plane $x_1+x_3=y_1+y_2=0$
and the other in the plane $ a(x_1+x_2)+y_1=y_1+y_2=0$.
One obtains smooth conics in each plane that intersect only at
the canonical point where they share the tangent line $x_1+x_3=y_1=y_2=0$.

If $C$ is irreducible, then one checks that the canonical point is the
unique singular point, which is a cusp.
Since $C$ has arithmetic genus $1$, we conclude that $C$ is a rational curve.
\end{proof}

\section{Differential structure in special characteristics}
\label{sec:phenomena}

Given a hypersurface $X$ in $\bbP^n$ of degree $d$ defined by a homogeneous form $F$, the
\emph{Hessian hypersurface} $\He(Y)$ of $X$ is the hypersurface defined by
the determinant of the matrix $D$ of second partial derivatives of $F$:
\[
H(F) := \det \left( \frac{\partial^2 F}{\partial x_i \partial x_j} \right) \ .
\]
It is a hypersurface of degree $(d-2)(n+1)$ unless the determinant vanishes identically.
The definition is independent of choice of coordinates for $F$.
Indeed, if $A$ is an invertible matrix corresponding to a change of
coordinates, then $D$ is taken to $A^TDA$; now $\det(D)$ vanishes if and
only if $\det(A^TDA)$ vanishes.

In this section, $X$ is a smooth cubic surface defined by a homogeneous form $F$ of
degree $3$, so $\He(X)$ is a hypersurface of degree $4$.

In characteristic $0$, we have the following result, essentially due to
Dardanelli and van Geemen~\cite{Dardanelli}, which partitions the cubic
surfaces based on the structure of their Hessian surfaces:

\begin{thm}
\label{thm:DG}
Let $X$ be a smooth cubic surface over an algebraically closed field of
characteristic $0$ and let $\He(X)$ be its Hessian hypersurface.
Then exactly one of the following holds:
\begin{enumerate}
\item[(a)] $\He(X)$ is irreducible with $10$ ordinary double points,
\item[(b)] $\He(X)$ is irreducible with $7$ singularities,
\item[(c)] $\He(X)$ is irreducible with $4$ singularities,
\item[(d)] $\He(X)$ is a union of a hyperplane and an irreducible surface,
\item[(e)] $\He(X)$ is a union of $4$ hyperplanes.
\end{enumerate}
\end{thm}

In~\cite{Dardanelli}, they also find normal forms for the cubics in each
of these classes; they correspond to (not necessarily closed)
subvarieties of the moduli space of smooth cubic surfaces.
The cubics in class (a) correspond to those which have a
non-degenerate Sylvester form
(see Section~\ref{sec:generalForms} below).
The cubics in classes (d) and (e) correspond to cyclic surfaces (see
Section~\ref{sec:3C}), with (e) being the Fermat cubic
surface (see Section~\ref{sec:Fermat}).

\begin{remark} \label{rem:DGapproach}
The Hessian surfaces constrain the possible automorphism groups of the
corresponding cubic surfaces.  One could classify cubic surfaces along
with their automorphisms by considering each class in turn and using the
normal forms given by Dardanelli and van Geemen.
Note that surfaces with isomorphic automorphism groups can lie in
different classes, however.
\end{remark}

In characteristic $2$ and $3$, the behavior of partial derivatives of
cubic forms is very different.
In characteristic $2$, the Hessian matrix always has zero determinant so
Hessian surfaces do not exist. 
In characteristic $3$, the Hessian surfaces exist but the information
they encode is more naturally captured by the critical loci of the cubic
surfaces, which are non-empty even though the surfaces are smooth!
We discuss these differences in the next two sections.

\subsection{Canonical Point in characteristic 2}
\label{sec:strange}

Note that while the Hessian matrix in characteristic $2$ always has zero
determinant, the matrix itself is not necessarily trivial. In fact, if
$a_{123}, a_{023}, a_{013}, a_{012}$ are the coefficients of the
monomials $x_1x_2x_3, x_0x_2x_3, x_0x_1x_3, x_0x_1x_2$ of a cubic form
$F$ defining the surface, we have
\[
H(F) = (h_{ij})_{0\le i\le j\le 3},
\]
where $h_{ii} =0$ and $h_{ij} = h_{ji} = x_ka_{ijk}+x_la_{ijl}$
for $i,j,k,l$ distinct.
 
The Hessian matrix $H(F)$ evaluated at the point
$y= (y_0,y_1,y_2,y_3)$ is the symmetric matrix associated to the
symmetric bilinear form associated to the polar
quadratic form $P_y(F) = \sum y_i\frac{\partial F}{\partial x_i}$
(this can be checked directly or see \cite{CAG}, Proposition 1.1.17).
The singular locus of the corresponding quadric $V(P_y(F))$ is the
projective space associated to the null space of the matrix $H(F)(y)$.

Assume at least one of the $a_{ijk}$ is non-zero.
Let $v= (a_{123}, a_{023}, a_{013}, a_{012})$ and
$\bar{v}$ be the corresponding point in $\bbP^3$.
One checks that $H(F)(v)$ is the zero matrix, and so the singular
locus of the quadric $V(P_v(F))$ is the whole space.
Thus $P_{v}(F)$ is the square of a linear form.
We immediately check that, for all $y$, we have  
\begin{equation}\label{hessianeq}
H(F)(y)\cdot v = 0.
\end{equation}
We claim that the point $\bar{v}$ is the unique point with this property.
Viewing $y_0,y_1,y_2,y_3$ as algebraically independent indeterminates,
equation \eqref{hessianeq} is equivalent to a linear system of 16 equations
in the coordinates of $v$.
The rank of the corresponding $16\times 4$ matrix is $3$,
so we have a unique solution up to proportionality.
This establishes the claim.

Let us record what we have found.

\begin{lem}
Let $X = V(F)$ be a smooth cubic surface with $p=2$.
If the Hessian matrix $H(F)$  is non-zero,
then there is a unique vector $v \in \bbk^4$ up to scaling
such that $H(F)(y)\cdot v=0$ at all $y \in \bbk^4$.
Given a defining form $F$ for $X$ in coordinates $x_0,x_1,x_2,x_3$,
we find
\[
v = (a_{123}, a_{023}, a_{013}, a_{012})
\]
where $a_{ijk}$ is the coefficient of the monomial $x_ix_jx_k$ in $F$.
The polar quadratic form $P_v(F)$ is a square of a linear form.
\end{lem}

In the case when $(a_{123}, a_{023}, a_{013}, a_{012})\ne 0$ we denote
by $\sfc_X$ the  point $(a_{123}: a_{023}:a_{013}:a_{012})\in \bbP^3$
and call it the \emph{canonical point} of $X$
(note that this point does not necessarily lie on $X$).
The linear form whose square is $P_v(F)$ defines the \emph{canonical plane}. 
 
We find different normal forms for $X$ depending on the behavior of
this canonical point.  The following can be viewed as a rough analog to
Theorem~\ref{thm:DG} in characteristic $2$.

\begin{prop}\label{prop:5.4}
Assume $p = 2$. Then exactly one
of the following holds:
\begin{enumerate}
\item[(a)] $X$ has no canonical point and is isomorphic to the Fermat
surface defined by
\[
x_0^3+ x_1^3 + x_2^3+ x_3^3 = 0 \ .
\]
\item[(b)] $X$ has a canonical point but it does not lie on $X$;
the form $F$ has the normal form
\[
x_0(x_0+cx_1)^2 + C(x_1,x_2,x_3) = 0 \ .
\]
\item[(c)] $X$ contains its canonical point and the canonical plane is
not a tritangent plane; $F$ has the normal form
\[
x_0(x_0x_1+x_2^2) + C(x_1,x_2,x_3) = 0 \ .
\]
\item[(d)] the canonical plane of $X$ is a tritangent plane of $X$ with
the canonical point being its Eckardt point; $F$ has the normal form
\[
x_0x_1(x_0+x_1) + C(x_1,x_2,x_3) = 0 \ .
\]
\end{enumerate}
In each case, $c$ is a constant and $C$
is a homogeneous cubic form in $x_1,x_2,x_3$.
\end{prop}

\begin{proof}
If the canonical point does not exist, then there are no non-zero
coefficients of monomials of the form $x_ix_jx_k$ with $i,j,k$
distinct; this implies that one can write the equation of $X$ in the form
$x_0^2L_0+x_1^2L_1+x_2^2L_2+x_3^2L_3 = 0,$ where $L_i$ are linear forms.
Note that first partials will be a linear combination of squares
$x_i^2$.  In other words, the first partials are squares of linear forms. 
We may now conclude this is the Fermat surface by a result of
Beauville:
implication (iv) $\Rightarrow$ (v) of Th\'eor\`eme from \S 2
of \cite{Beauville}.

Otherwise, there exists a canonical point that we may assume to be
$(1:0:0:0)$.  Write the form $F$ defining $X$ as follows:
\[
F= a_0x_0^3+A_1x_0^2+B_2x_0+C_3=0
\]
where $a_0,A_1,B_2,C_3$ are constant, linear, quadratic, and cubic forms of $x_1,x_2,x_3$.
We compute the polar $P=P_{\sfc_X}(X)=a_0x_0^2+B_2$ and see that
$B_2$ must be the square of a linear form.

If $a_0 \ne 0$, then, by a change of coordinates in $x_0$ only,
we may assume $B_2$ is a multiple of $x_1$.
If $a_0=0$, then by a change of coordinates in $x_1,x_2,x_3$ we may
again assume $B_2$ is a multiple of $x_1$.
In either case, the canonical point remains $(1:0:0:0)$.
Thus:
\[
F = x_0(a_1x_0+a_2x_1)^2 + A_1x_0^2 + C_3 = 0
\]
for $a_1=\sqrt{a_0}$ and some constant $a_2$.
Note that $\sfc_X \in X$ if and only if $a_1 = 0$.

Suppose $a_1 \ne 0$.  Then we may assume $a_1=1$ and take $x_0 \mapsto
x_0 + A_1$ to obtain the form
\[
F = x_0^3+a_2^2x_1^2x_0+C_3, \ P=(x_0+a_2x_1)^2
\]
for case (b).

We now suppose $a_1 = 0$, so $\sfc_X \in X$ and $P=(a_2x_1)^2$.
Since the polar cannot be zero, we may assume $a_2=1$.
Either $A_1$ is a multiple of $x_1$ or it is not.  The two remaining cases
are the obtained from these two possibilities.
The facts about tritangent planes are verified using the normal forms.
\end{proof}

\subsection{Critical loci in characteristic 3}
\label{sec:criticalLoci}

The \emph{critical locus} $\Crit(X)$ of a projective hypersurface $Z$
is the subscheme defined by the ideal of partial derivatives
$\partial_0 F, \ldots, \partial_n F, $
where $F$ is the homogeneous form defining $X$.
If $X$ has degree $m$, then by Euler's formula
\[
\sum_{i=0}^n x_i \partial_i F = mF
\]
the critical locus of $Z$ coincides with the singular locus of $X$
when the degree $m$ is coprime to the characteristic.
However, when the characteristic divides the degree, the critical locus
may be larger than the singular locus --- there may be a non-trivial
critical locus even when the variety $Z$ is smooth.

\begin{prop} \label{prop:criticalLocus}
Let $Z$ be a smooth hypersurface of degree $m$ in $\bbP^n$
over a field $\bbk$ of characteristic $p$ dividing $m$.
If $m=2$ and $n$ is even then $\Crit(Z)$ is empty.
Otherwise, $\Crit(Z)$ is a $0$-dimensional closed subscheme of degree $h^0(\calO_{\Crit(Z)})$ equal to 
\begin{equation} \label{critical}
J_m(n):=\frac{(m-1)^{n+1}-(-1)^{n+1}}{m} \ .
\end{equation}
\end{prop}

\begin{proof}
Let $V$ be the vector space such that $\bbP^n=\bbP(V)$ and
consider $F\in S^m(V^\vee) \cong \Bbbk[t_0,\ldots,t_n]_m$.
The differential $dF$ defines a section of
$\Omega_{S(V^\vee)}^1(m)$.
The critical locus $\Crit(Z)$ is the scheme of zeros of this
section.
Note that the intersection $\Crit(Z) \cap X$ is the set of singular
points of the hypersurface $X$.
Since $X$ is smooth, $\Crit(Z)$ is either empty or
purely $0$-dimensional.
Thus $h^0(\calO_{\Crit(Z)})$ is equal to the Chern
number $c_{n}(\Omega_{\bbP^n}^1(m))$.

The known properties of Chern classes (see \cite{Fulton}, Example 3.2.2)
allow one to compute $h^0(\calO_{\Crit(Z)})$.
We have $h^0(\calO_{\Crit(Z)}) = J_m(n)$, where
\begin{align*}
J_m(n) &=
\sum_{i=0}^nc_i(\Omega_{S(V^\vee)}^1)c_1(\calO_{\bbP^n}(m))^{n-i}\\
&= 
\sum_{i=0}^n\tbinom{n+1}{i}(-1)^ic_1(\calO_{\bbP^n}(1))^i\cdot
(mc_1(\calO_{\bbP^n}(1))^{n-i})\\
&=  
\sum_{i=0}^n\tbinom{n+1}{i}(-1)^im^{n-i}\\
&=
\frac{(m-1)^{n+1}-(-1)^{n+1}}{m}\ ,
\end{align*}
as desired.
\end{proof}

\begin{remark} \label{rem:quadCritical}
Consider the case $m=2$ where $Z$ is a quadric and $F$ is a quadratic
form.
The critical locus of $Z$ is the set of zeroes of the alternating
bilinear form $B$ corresponding to $F$.
For a smooth form when $n$ is odd, the form $B$ is non-degenerate and so
the critical locus is always empty.
When $n$ is even, the form $B$ is degenerate with a one-dimensional
radical in the ambient vector space;
thus $\Crit(Z)$ consists of a single point in $\bbP^n$.
This explains why $J_2(n)=0$ if $n$ is odd, but $J_2(n)=1$ if $n$ is even.
\end{remark}

\begin{remark}
Suppose $X \subseteq \bbP^n$ is a hypersurface defined by a homogeneous
polynomial $F$ of degree $m$.
Recall that the \emph{discriminant} $\disc(f)$ is the resultant
of the partial derivatives $\{\partial_0 f, \ldots \partial_n f\}$.
When $m$ is coprime to the characteristic $p$, the hypersurface $X$ is
smooth if and only if $\disc(f)$ is non-zero.
However, $\disc(f)$ is always zero if $p$ divides $m$.

In \cite{Demazure}, Demazure defines the \emph{divided discriminant} of
a hypersurface.  If one takes the usual discriminant of a generic
homogeneous polynomial $P_m$ of degree $m$ over $\bbZ$, Demazure shows
that the gcd of the coefficients of $\disc(P_m)$ is $m^a$ where
$a=J_m(n)$ as in Proposition~\ref{prop:criticalLocus}.
The \emph{universal divided discriminant} is then
the integer polynomial $\disc(P_m)/m^a$, while the divided discriminant of $f$
is simply the specialization to the coefficients of $f$.
The divided discriminant is non-zero if and only if $X$ is smooth,
regardless of the characteristic of the ground field.

In the case when $m = 2$ and $n$ is even, the divided discriminant coincides with the half-discriminant of Kneser. As we observed in the previous remark, we have $J_2(n) = 1$ in this case.

In the particular case of a cubic surface $X$ in characteristic $p=3$,
we see that $J_3(3)=5$.  Since $\Crit(Z)$ is invariant under projective
transformations of $X$, it is a useful invariant in studying
automorphism groups of $X$.
\end{remark}

Another useful observation is that the subspace of perfect cubes spanned
by $x_0^3, \ldots, x_n^3$ form a
$\GL(V)$-invariant subspace of $S^3(V^\vee)$ when $p=3$.
Thus the quotient space $S^3(V^\vee)/\bfF(V^\vee)$ has a
$\GL(V)$-action, where $\bfF$ is the Frobenius map raising a linear
form to the third power.
When $\dim V = 4$, it was shown in \cite{Chen} and \cite{Cohen} that
the quotient space only has finitely many $\GL(V)$-orbits.
In particular, the quotient is a prehomogeneous vector space
unique to characteristic $3$.

These orbits were completely classified by Cohen and Wales in \cite{Cohen}
with the help of a computer.
Since we are interested only in forms $F$ defining a smooth surface,
we state only this part of their classification.

\begin{prop}[Cohen, Wales]
\label{prop:CohenWales}
Let $F$ be a homogeneous form defining a smooth cubic surface with $p=3$.
Then $F$ belongs to one of the following $\GL_4(\bbk)$-orbits of cubic forms:
\begin{itemize}
\item[(i)] $\left( \sum_{i=0}^3 c_ix_i^3 \right)
+ x_0x_1x_2+x_1x_2x_3+x_2x_3x_0+x_3x_0x_1$,
\item[(ii)] $\left( \sum_{i=0}^3 c_ix_i^3 \right)
+ x_0^2x_1+x_2^2x_3+x_3^2x_1$,
\item[(iii)] $\left( \sum_{i=0}^3 c_ix_i^3 \right)
+ x_0^2x_1+x_0^2x_2+x_0x_2^2+x_1^2x_2+x_0x_2x_3$,
\item[(iv)] $\left( \sum_{i=0}^3 c_ix_i^3 \right)
+ x_0^2x_1 + x_2^2x_3+x_3^2x_0$,
\end{itemize}
\end{prop}

Since the derivative of a perfect cube is trivial when $p=3$,
the critical locus (and Hessian surface) only depends on the orbit
and not the specific surface.
We can distinguish the orbits by the support of $\Crit(X)$.
We obtain points with multiplicities
$\{1,1,1,1,1\}$, $\{3,1,1\}$, $\{4,1\}$, and $\{5\}$
in cases (i),(ii),(iii), and (iv), respectively.

\begin{remark}
\label{rem:CohenWales}
Cohen and Wales~\cite{Cohen} actually computed the stabilizers of each of
these elements modulo $\bbF(V^\vee)$.
These stabilizers act as affine transformations on the coordinates
$c_1, \ldots, c_4$.
One can compute all the automorphism groups of smooth cubic
surfaces by identifying fixed points of subgroups of this action.
As with the proof of Proposition~\ref{prop:CohenWales},
this is probably best done with the aid of computers.
This can be seen as analogous to the approach of
Remark~\ref{rem:DGapproach}.
Again, surfaces with isomorphic automorphism groups can lie in
different classes.
\end{remark}

\section{The Fermat cubic surface}
\label{sec:Fermat}

The \emph{Fermat cubic surface} is the cubic surface defined by
\[
x_0^3 + x_1^3 + x_2^3 + x_3^3 = 0
\]
in $\bbP^4(k)$.
This is not even reduced when $p=3$, but is smooth in all other
characteristics.
We will see that the Fermat cubic is the unique smooth cubic surface
in the stratum 3C when $p\ne 3$.

There is an ``obvious'' group of automorphisms isomorphic to
$3^3 \rtimes \frakS_4$ generated by permutations of
the coordinates and multiplying coordinates by cube roots of unity.
In fact, this is the full group of automorphisms unless $p=2$.

\begin{lem} \label{lem:3Cautos}
Let $X$ be the Fermat cubic.  Then $\Aut(X)$ is generated by
reflections.
If $p \ne 2,3$ then $\Aut(X) \simeq 3^3 \rtimes \frakS_4$.
If $p = 2$ then $\Aut(X) \simeq \PSU_4(2)$.
\end{lem}

We will consider the case of $p=2$ below.
When $p \ne 2,3$, Lemma~\ref{lem:3Cautos} follows from the following:

\begin{lem} \label{lem:ShiodaNot2}
Suppose $d \ge 3$, $n \ge 4$, $p \nmid d-1,d$ and consider the
hypersurface $Z$ defined by
\[
F = x_0^d + \cdots x_n^d = 0
\]
in $\bbP^n$.  Then the group of (projective) automorphisms of $Z$ is
is isomorphic to $d^{n-1} \rtimes \frakS_n$.
\end{lem}

\begin{proof}
A slightly stronger result is proved in \cite{Shioda},
where the $d$ is only assumed to not be a power of $p$.
The difficulty here is showing that the given automorphism group is the
full group of automorphisms and not a proper subgroup. 

As defined in Section~\ref{sec:phenomena}, in this case
the Hessian hypersurface $\He(Z)$ is given by
\[
\left(d(d-1)\right)^n x_0^{d-2} \cdots x_m^{d-2} = 0 \ ,
\]
which is non-trivial when $p \nmid d,d-1$.
The Hessian surface is intrinsic to the projective embedding and not the
choice of coordinates; thus the projective automorphisms of $Z$ induce
projective automorphisms of $\He(Z)$.
The hypersurface $\He(Z)$ is defined by
\[
\left(d(d-1)\right)^n x_0^{d-2} \cdots x_m^{d-2} = 0 \ ,
\]
which is non-trivial when $p \nmid d,d-1$.
The automorphisms of $\He(Z)$ are clearly generated by permutations of
the coordinates and multiplying coordinates by constant factors.
To be an automorphism of $F$, these constant factors must be $d$th roots
of unity.
\end{proof}

The Eckardt points and exceptional divisors of the Fermat cubic surface
for $p  > 3$ are well known. The surface contains $27$ lines given by
equations $x_i+\epsilon x_j = 0 = x_k+\eta x_l = 0$, where $\{i,j,k,l\}
= \{0,1,2,3\}$ and $\epsilon^3 = \eta^3 = 1$.
Among the $45$ tritangent planes there are $18$ given by equations
$x_i+\epsilon  x_j  = 0, \ 0\le i< j\le 3$, each containing an Eckardt
point defined by equations $x_i+\epsilon x_j = 0,\   x_k = x_l = 0$.
The automorphism group acts transitively on these objects since it
contains transformations that permute the coordinates and multiply them
by third roots of unity.

\subsection{The Fermat surface in characteristic 2}

In the remainder of the section we establish Lemma~\ref{lem:3Cautos}
for $p = 2$ and describe the geometry of the Fermat cubic in more detail.

Recall that, using a result of A.~Beauville, we showed in
Proposition~\ref{prop:5.4} that the Fermat cubic is the unique smooth
cubic surface in characteristic $2$ without a canonical point.

Let $\bbF_4$ be the field of four elements and
let $V$ be a non-degenerate Hermitian inner product space on $\bbF_4$
of dimension $4$.
In other words, if
$\bfF : \bbF_4 \to \bbF_4$ is the Frobenius automorphism
defined by $\bfF(x)=x^2$, then we have a non-degenerate Hermitian
bilinear form $H$ on $\bbF_4^4$ which we may write as
\[
H(x, y) = \sum_{i=0}^3 x_i \bfF(y_i)
\]
where $x=(x_0,x_1,x_2,x_3)$ and $y=(y_0,y_1,y_2,y_3)$
are vectors in $\bbF_4^4$.

Note that
\[
H(x, x)  = x_0^3 + x_1^3 + x_2^3 + x_3^3
\]
is the defining equation for the Fermat cubic $X$ in $\bbP^3$.
Thus, the automorphism group $\Aut(X)$ of the Fermat cubic $X$
contains the projective special unitary group $\PSU_4(2)$
(we shall see shortly that this is the entire automorphism group).
This is a group of order $25920$ which has index $2$ in the Weyl group
$W(\sfE_6)$.

Consider the map $\sigma$ on $X$ defined on $\bbF_4^4$ by
\[
\sigma(x_0: x_1: x_2: x_3)=
(\bfF(x_0): \bfF(x_1): \bfF(x_2): \bfF(x_3)) \ ,
\]
which depends on a choice of coordinates.
The action of $\sigma$ preserves $X(\bbF_4)$ when written in its
diagonal form, permutes the $27$ lines and preserves their
configuration.
Thus $\sigma$ acts as an element of $W(\sfE_6)$.
On the other hand, there is no element $\sigma$ in $\PSU_4(2)$ acting in
this way --- the action of $\sigma$ is not geometric.
Thus $\sigma$ and $\PSU_4(2)$ must generate all of $W(\sfE_6)$.
In particular, we have proven Lemma~\ref{lem:3Cautos} since the automorphism
group of $X$ cannot be larger than $\PSU_4(2)$.

The following lemma characterizes the orbits of the various subspaces of
$\bbP^3(\bbF_4)$ connected to the geometry of the Fermat cubic.
A similar result holds for \emph{any} Hermitian surface
(see Theorems~19.1.7~and~19.1.9~of~\cite{Hirschfeld}),
but we emphasize the connection to cubic surfaces in the statement and
its proof.

\begin{lem} \label{lem:HermitianSubspaces}
The orbits of $\Aut(X)$ on the set of all subspaces of $\bbP^3(\bbF_4)$ are as
follows:
\begin{enumerate}
\item[($C_{1}$)]
45 Eckardt points (these are precisely the points which lie on $X$).
\item[($C_{0}$)]
40 points which do not lie on $X$.
\item[($C_{5}$)]
27 exceptional lines.
They each contain precisely 5 Eckardt points.
\item[($C_{3}$)]
240 trihedral lines.
These contain 3 Eckardt points and 2 non-Eckardt points.
\item[($C_{1}'$)]
90 lines that intersect $X$ at exactly one point.
\item[($C_{13}$)]
45 tritangent planes.
Each contains 13 Eckardt points, 8 non-Eckardt points, 3 exceptional lines,
16 trihedral lines, and 2 other lines.
\item[($C_{9}$)]
40 planes which intersect $X$ in a smooth cubic curve.
Each contains 9 Eckardt points, 12 non-Eckardt points, no exceptional lines,
12 trihedral lines, and 9 other lines.
\item[($C_{45}$)]
the whole space.
\end{enumerate}
\end{lem}

\begin{proof}
Recall that $\bbP^3(\bbF_4)$ contains 85 points, 357 lines and 85
planes; $\bbP^2(\bbF_4)$ contains 21 points and 21 lines;
and $\bbP^1(\bbF_4)$ contains 5 points.

Since $x^3=1$ for all nonzero $x \in \bbF_4$,
the points on $X$ are precisely those with an even number of non-zero
coordinates.
A short counting argument shows that there are 45 of them.
A tangent space computation reveals that any one of them is an Eckardt
point.
Since $\Aut(X)$ acts transitively on the Eckardt points, all 45 points
on $X$ must be Eckardt points.

There are 40 remaining points that do not lie on $X$, which we call
``non-Eckardt points.''
These points satisfy $\langle x, x \rangle=1$ for any
vector $x$ representing them.
By the Gram-Schmidt process, any four mutually orthogonal such vectors
can be brought to diagonal form by a unitary transformation.
In particular, $\Aut(X)$ acts transitively on the 40 non-Eckardt points.

Since all the Eckardt points are defined over $\bbF_4$, so must be the
27 exceptional lines between them.
They form an orbit of the automorphism group.
They each contain 5 Eckardt points since they are
projective lines containing only points of $X$.

The tangent space $P$ at each Eckardt point $p$ contains three exceptional
lines, so there are two more lines in $P$ passing through $p$
which do not contain any other Eckardt points.
Thus there are $2 \times 45 = 90$ such lines.
The stabilizer of $p$ in $\Aut(X)$ acts by $\frakS_3$ on the exceptional
lines through $p$.  Thus there is a faithful action of $\frakS_3$ on the
tangent space $P$ fixing $p$.
The action of $\frakS_3$ on $\bbP^1$ has no fixed points,
so the other two lines through $p$ are permuted by an element of $\Aut(X)$.
Since the Eckardt points are permuted,
$\Aut(X)$ acts transitively on the $90$ lines.

There are $357-27-90=240$ remaining lines.
Since all lines tangent to a
point of $X$ have been identified, only those passing through three distinct
points of $X$ remain.
These are precisely the lines on which the Hermitian form on the whole
space restricts to a nondegenerate Hermitian form
(vectors representing the two non-Eckardt points lift to an orthonormal
basis).
By the Gram-Schmidt process, these lines are all in the same $\Aut(X)$
orbit.

Given a subspace $U \subset \bbF_4^4$, we define the subspace
\[
U^\dagger := \{ v \in \bbF_4^4 : \langle v,u \rangle =0
\textrm{ for all } u \in U \}
\]
called the \emph{polar dual} of $U$.
There is a bijection between planes and points given by the polar dual
which commutes with the action of the automorphism group.
Thus we have the orbital dichotomy of the planes of $\bbP^3(\bbF_4^4)$.
It remains to characterize them.

For the 45 tritangent planes, we have already seen the 2 lines with a unique
Eckardt point and the three exceptional lines.  The remaining lines must
all pass through the three exceptional lines and thus are trihedral lines.
There are $21-5=16$ of them.

Finally, the remaining 40 planes are not tangent to $X$ and thus
each intersects $X$ in a smooth cubic curve.
Note that any plane containing two exceptional lines must be a
tritangent plane, so there are no exceptional lines on the 40 planes.
Thus, for each plane $P$,
there must exist three points which lift to linearly independent vectors
with non-zero norm.  Thus the restriction of the Hermitian form to $P$
is non-degenerate.
We immediately check that there are $9$ Eckardt points by writing out
all the solutions of the form.
The remaining counts of lines and points follow from polar duality
and the fact that there are 21 points and 21 lines in the whole space.
\end{proof}

The notation ``$C_n$'' used in the statement is meant to emphasize the
connection these subspaces have to the reflection subgroups of
$\PSU_4(2)$ constructed in Section~\ref{sec:EckardtCollections}.
Note that the Eckardt points correspond to isotropic points of the
Hermitian form and thus to reflections in $\PSU_4(2)$.
If one excludes $C_1'$, the points in each subspace correspond to
reflections which generate subgroups of $W(\sfE_6)$.
We will see that this gives a bijection between those orbits of subspaces of the
Hermitian space $\bbP^3(\bbF_4)$ and the possible automorphism groups of
an arbitrary smooth cubic surface in characteristic $2$.

Moreover, the Fermat cubic over $\bbF_4$ realizes all possible Eckardt
points, exceptional lines and trihedral lines.
Thus one can study the combinatorics of these objects in arbitrary cubic
surfaces over arbitrary fields by consulting the Fermat as a ``universal
model.''

\section{General forms}
\label{sec:generalForms}

We recall the following classical theorem of Sylvester, which can be
seen as showing that a generic cubic surface is a nice deformation
of the Clebsch.

\begin{thm}[Sylvester] \label{thm:Sylvester}
Suppose $p\ne 2,3$.
A general smooth cubic surface $X$ is isomorphic to one defined by the
forms
\begin{equation} \label{eq:SylvesterForm}
\sum_{i=0}^4 c_ix_i^3 = \sum_{i=0}^4 x_i = 0
\end{equation}
in $\bbP^4$ where $c_0, \ldots, c_4 \in \bbk$ are non-zero parameters.
The parameters are uniquely determined up to permutation and common
scaling by the isomorphism class of the surface.
Moreover, the automorphism group of any such surface is a subgroup
of the group $\frakS_5$ which acts by permuting coordinates.
\end{thm}

\begin{proof}
When $p=0$, a proof of the first assertion can be found in
Theorem~9.4.1~of \cite{CAG}.  An analysis of the proof shows that it
works in all characteristics different from $2$ and $3$.

It remains to determine the automorphism groups of these surfaces.
Let $X = Y \cap P$ be one such surface, where
$Y$ is the cubic threefold and $P$ is the hyperplane given by
$\sum_i x_i = 0$ from the defining equation \eqref{eq:SylvesterForm}.
Note that $X \hookrightarrow P$ is the anticanonical embedding of the
cubic surface.

Let $C$ be the union of the 5 coordinate hyperplanes.
The intersection $S = C \cap P$ is a union of $5$ planes,
no $4$ of which have a common point.
The collection $S$ is the famous \emph{Sylvester pentahedron}.
The intersections of the planes form 10 lines and 10 vertices.
The automorphisms of the pentahedron are simply the permutations of the
coordinates $x_0, \ldots, x_4$.

Let $H$ be the Hessian surface of $X$ in $P$
(this is \emph{not} the restriction of the Hessian of $Y$ constructed
in the proof of Lemma~\ref{lem:ShiodaNot2}).
From \S{}9.4.2~of~\cite{CAG}, the equation of the Hessian surface embedded in $\bbP^4$ is  
\[
\sum_{i=0}^4
(c_0\cdots \widehat{c_i} \cdots c_4)
(x_0\cdots \widehat{x_i} \cdots x_4)
= \sum_{i=0}^4 x_i = 0 \ .
\] 
Since all the parameters are assumed non-zero, we may rewrite this as
\[
x_0\cdots x_4\left( \sum_{i=0}^4 \frac{b_i}{x_i} \right)
= \sum_{i=0}^4 x_i = 0
\]
where $b_i=c_i^{-1}$.
We will describe the singular locus of $H$.
There are 10 additional conditions $C_{ij}$ defining the singular locus defined
by minors of a $5 \times 2$ matrix.  For example, we have
\[
C_{01} := (b_1-b_0)x_2x_3x_4 + (x_1-x_0)(x_2x_3b_4+x_2b_3x_4+b_2x_3x_4) \ .
\]
If all the $x_i$ are non-zero then we may use the condition
$\sum \frac{b_i}{x_i}=0$ to show that these conditions are equivalent to
$b_jx_i^2=b_ix_j^2$.
By \cite{CAG}, Proposition 9.4.5, these are precisely the conditions for the original cubic to be
singular at the given point;
thus, we may assume at least one $x_i$ is zero.
If one $x_i$ is zero, then another must be zero as well by the
original equation for $H$.
If $x_0=x_2=0$ then the condition $C_{01}$ shows that that a third
coordinate must be zero; this applies for any pair of coordinates.
Thus, we conclude that there are exactly $10$ singular points ---
namely the $\frakS_5$-orbit of $(1:-1:0:0:0)$.

The singular points of $H$ are precisely the 10 vertices
of the pentahedron, and the 10 lines lie on $H$.
Thus, one can reconstruct the pentahedron using the Hessian surface.
The surface $H$ is canonical, thus the pentahedron is as well.
Thus the expression \eqref{eq:SylvesterForm}
is unique up to permutation and multiplication by a common non-zero
scalar.
We conclude that automorphisms of $X$ are subgroups of $\frakS_5$
that act by permuting the coordinates.
\end{proof}

When $p=3$, Sylvester's theorem fails since the corresponding cubic
is not even reduced.
In $p=2$, the theorem also fails, but for a more subtle reason.
Viewing the equation \eqref{eq:SylvesterForm} over $\bbZ$, we may assume
without loss of generality that $c_4=1$ and eliminate $x_4$ to obtain
\[
\sum_{i=0}^3 (c_i-1)x_i^3
- \sum_{0 \le i \le j \le 3} 3x_i^2x_j
- \sum_{0 \le i< j<k \le 3} 6x_ix_jx_k = 0 \ .
\]
Since there are no terms of the form $x_ix_jx_k$ if $p=2$,
there is no canonical point.
Thus by Proposition~\ref{prop:5.4}, we see that any surface in
Sylvester's form is isomorphic to the Fermat cubic ---
a unique surface is not generic!

\begin{remark}
We define the \emph{Clebsch cubic surface} as the cubic surface
admitting a group of automorphisms isomorphic to $\frakS_5$.  
We shall see that it is unique up to isomorphism when it exists and corresponds to the
stratum 5A in the moduli space of cubic surfaces.
Classically, the Clebsch cubic surface is defined as the cubic surface
given by  equations \eqref{eq:SylvesterForm} with
$(c_0,c_1,c_2,c_3,c_4) = (1,1,1,1,1)$.
Here permutations of the coordinates generate a group
isomorphic to $\frakS_5$.
Note that the construction fails in characteristic $3$ and $5$.
In characteristic $3$ it is not even reduced;
while in characteristic $5$, the point $(1:1:1:1:1)$ is singular.
If $p\ne 5$, alternate equations for the Clebsch are
\[ \sigma_3 = \sigma_1 = 0, \]
where $\sigma_i$ denotes the elementary symmetric polynomial of degree $i$
in the variables $x_0,\ldots,x_4$.
This provides a cubic surface with an $\frakS_5$-action,
which is smooth when $p \ne 5$.
We shall see that there is no Clebsch surface in characteristic $5$.
\end{remark}

We would like a generic form that works in all characteristics.
The following normal form is
due to Emch \cite{Emch} who described it for cubic surfaces over $\bbC$.
As Sylvester's normal form is a nice deformation of the
Clebsch cubic surface, Emch's normal form is a nice deformation of the
Fermat cubic surface.

\begin{thm}[Emch] \label{thm:Emch}
A general smooth cubic surface over an algebraically closed field of
arbitrary characteristic can be defined by the form
\begin{equation} \label{eq:Emch}
\sum_{i=0}^3 x_i^3 + \sum_{0 \le i< j<k \le 3} a_{ijk}x_ix_jx_k
\end{equation}
in appropriate coordinates $x_0,\ldots,x_3$ where
$a_{012}, \ldots, a_{123}$ are parameters.
\end{thm}

\begin{remark} \label{rem:EmchCohen}
Some of these surfaces are singular.
For example, setting all parameters $a_{ijk}$ to $-1$,
the surface has a singular point $(1:1:1:1)$.
If all the parameters in \eqref{eq:Emch} are non-zero,
then by rescaling the coordinates we obtain
\begin{equation} \label{eq:EmchRescaled}
\sum_{i=0}^3 c_ix_i^3 + \sum_{0 \le i< j<k \le 3} x_ix_jx_k
\end{equation}
which is the first normal form from Proposition~\ref{prop:CohenWales}
in characteristic $3$.
\end{remark}

\begin{proof}
Our argument is inspired by the argument in \cite{Cohen} for the case $p=3$.
We will use the equivalent form \eqref{eq:EmchRescaled} throughout the
proof.
Let $V_0$ be the span of the monomials $x_0^3, \ldots, x_3^3$
in $\Sym^3(\bbk^4)$.
Define the form
\[
F_0 = \sum_{0 \le i< j<k \le 3} x_ix_jx_k \ .
\]
It suffices to show that a generic orbit of $G=\GL_4(\bbk)$ intersects the
affine subspace $F_0 + V_0$ in at most finitely many points.
We will do this by showing that the tangent space to the orbit at some $F$
in this subspace is transverse to the space $V_0$.

Let $F$ be a form as in \eqref{eq:EmchRescaled}.
Let $E_{ij}$ be the usual basis of the Lie algebra $\gl_4(\bbk)$
where $E_{ij}(x_j)=x_i$.
The induced action of $\gl_4(\bbk)$ on $\Sym^3(\bbk^4)$ is given by the
usual Leibniz rule
\[
E_{ij}(x_kx_lx_m)=E_{ij}(x_k)x_lx_m + x_kE_{ij}(x_l)x_m +
x_kx_lE_{ij}(x_m) \ .
\]
For $i,j,k,l$ distinct, we obtain
\begin{align*}
E_{ii}(F) &= 3c_ix_i^3 + x_ix_jx_k + x_ix_jx_l + x_ix_kx_l\\
E_{ij}(F) &= 3c_jx_j^2x_i + x_i^2x_k + x_i^2x_l + x_ix_kx_l.
\end{align*}
Writing
\[
M = \sum_{i,j} m_{ij}E_{ij}
\]
for a general element of $\gl_4(k)$.
The tangent space $T$ of the orbit of $G(F)$ is spanned by $M(F)$
as the coefficients $m_{ij}$ vary.

We find the coefficient of $x_i^2x_j$ in $M(F)$ is:
\begin{equation} \label{eq:iij}
3c_im_{ji} + m_{ik} + m_{il}
\end{equation}
while the coefficient of $x_ix_jx_k$ is:
\begin{equation} \label{eq:ijk}
m_{ii}+m_{jj}+m_{kk}+m_{il}+m_{jl}+m_{kl} \ .
\end{equation}
The zero set $Z_F$ of these expressions will give the infinitesimal
transformations along the space $V_0$.
As $F$ varies in the space $F_0 + V_0$, so does $Z_F$.
Since the dimension of $Z_F$ is upper semicontinuous in $F$,
it suffices to show $Z_F$ is zero-dimensional at some point $F$.

If $p \ne 2$, we consider the point $F=F_0$.
From \eqref{eq:iij} we obtain conditions $m_{ik}=-m_{il}$ for all $i,k,l$
distinct.
This forces $m_{ik}=-m_{il}=m_{ij}=-m_{ik}$ for all $i,j,k,l$ distinct,
thus all the off-diagonal elements of $M$ must be zero.
We are left with the conditions that $m_{ii}+m_{jj}+m_{kk}=0$
for all distinct $i,j,k$.
This has no non-trivial solution in characteristic $\ne 3$.
In characteristic $3$, we obtain only the scalar matrices ---
this tangent direction disappears in the projectivization.

We use a different point for characteristic $2$.
Consider $(c_0,c_1,c_2,c_3)=(\lambda,\lambda,\lambda,\lambda)$
where $\lambda \ne 0,1$.
Adding together the coefficients of $x_i^2x_j$, $x_i^2x_k$ and
$x_i^2x_l$ we obtain
\[ S_i := m_{ji}+m_{ki}+m_{li}=0\]
Adding together the coefficients of $x_i^2x_j$, $x_k^2x_j$ and
$x_l^2x_j$ we obtain
\[ \lambda(m_{ji}+m_{jk}+m_{jl})=m_{ji}+m_{jk}+m_{jl} \]
using $S_i$, $S_k$ and $S_l$ relations.
Since $\lambda \ne 1$, we conclude that
\[
T_j := m_{ji}+m_{jk}+m_{jl} = 0 \ .
\]
Summing together the coefficient of $x_i^2x_j$ with $T_i$ we see
that $\lambda m_{ji}=m_{ij}$.  Adding together $x_j^2x_i$ and $T_j$, we
conclude that $\lambda m_{ij} = m_{ji}$.
Thus all the off-diagonal elements are zero.
We conclude that the diagonal elements are also zero by the same argument
as before.
\end{proof}

\begin{remark}\label{rem:emchQuestion}
Above we show that a general smooth cubic surface can be written in Emch's
normal form.  If $p \ne 2$, then certainly not every smooth cubic
surface can be written in this form.  However, if $p = 2$, it may be
that \emph{every} smooth cubic surface can be written in Emch's normal
form!  If true, then the normal forms for $p=2$ in later sections could be
greatly simplified. 
\end{remark}

\begin{remark}\label{rem:emch2}
In characteristic zero, in a sequel paper \cite{Emch2}, Emch proves that
there are 40 different ways of writing a general cubic surface
in the normal form \eqref{eq:Emch}.
The different ways naturally correspond to
\emph{Steiner complexes} (see \S{}9.1.1~of~\cite{CAG})
or \emph{summit planes} (see \cite{Dixon}).
\end{remark}

Using the general normal forms we may produce large families of groups
with specified automorphism groups:

\begin{lem} \label{lem:SylvesterAutos}
If $p \ne 2$ then there exist $d$-dimensional strata of
smooth cubic surfaces $X$ such that
\begin{enumerate}
\item[(1A)] $\Aut(X)=1$ with $d=4$,
\item[(2A)] $\Aut(X)=\frakS_2$ with $d=3$,
\item[(2B)] $\Aut(X)=\frakS_2 \times \frakS_2$ with $d=2$,
\item[(3D)] $\Aut(X)=\frakS_3$ with $d=2$,
\item[(4B)] $\Aut(X)=\frakS_4$ with $d=1$,
\item[(6E)] $\Aut(X)=\frakS_3 \times \frakS_2$ with $d=1$, and
\item[(5A)] $\Aut(X)=\frakS_5$ with $d=0$ provided $p \ne 5$.
\end{enumerate}
\end{lem}

\begin{proof}
First, assume $p \ne 3$.
Using Sylvester's normal form,
we consider the induced action of $\frakS_5$ on the projective space of
parameters.  The stabilizers of this action correspond to cubic surfaces
with those groups as their full automorphism group.
The points with non-trivial stabilizers are obtained by setting certain
parameters equal to one another.
One obtains precisely the groups in the statement of the theorem
provided there exists a smooth cubic surface in each family.
One checks all families have a smooth member with the exception of
the $\frakS_5$-surface in characteristic $5$.

Now assume $p = 3$.
Consider the first normal form from Proposition~\ref{prop:CohenWales},
which after rescaling coordinates, is equal to Emch's normal form.
As noted in Section~\ref{sec:criticalLoci},
the critical locus consists of $5$ distinct points which span the
ambient projective space.
Thus the possible automorphisms are subgroups of $\frakS_5$. 

Apply the change of coordinates
\[ x_i \mapsto \left(\sum_{j=0}^3 x_j\right) + x_i \]
to obtain the normal form
\begin{equation} \label{eq:emchp3S5}
\left( \sum_{i=0}^3 d_ix_i^3 \right)
- \sum_{0 \le i \le j \le 3} x_i^2x_j
+ \sum_{0 \le i< j<k \le 3} x_ix_jx_k = 0 \ .
\end{equation}
Where the $d_i$'s are general parameters.
The advantage of this normal form is that the $\frakS_5$-action is more
transparent: it is generated by permutations of the $x_i$'s and
the matrix
\[
\begin{pmatrix}
0 & 1 & 0 & 0\\
0 & 0 & 1 & 0\\
0 & 0 & 0 & 1\\
-1 & -1 & -1 & -1
\end{pmatrix}
\]
of order $5$.
Moreover, $\frakS_5$ acts linearly on the vector space spanned by the
parameters $\{d_0, \ldots, d_3\}$
(the action is via affine transformations in the original normal form).
As for $p \ne 3$, the stabilizers are linear subspaces of the desired
dimensions and each contains a point which corresponds to a smooth cubic
surface.
\end{proof}

The link with partitions of $5$ is more tenuous when $p=2$ as the
automorphism groups are often larger.  However, we have the following:

\begin{lem} \label{lem:EmchAutos}
If $p = 2$ then there exist $d$-dimensional strata of
smooth cubic surfaces $X$ such that
\begin{enumerate}
\item[(1A)] $1 \subseteq \Aut(X)$ with $d=4$,
\item[(2A)] $\frakS_2 \subseteq \Aut(X)$ with $d=3$,
\item[(2B)] $\frakS_2 \times \frakS_2 \subseteq \Aut(X)$ with $d=2$,
\item[(3D)] $\frakS_3 \subseteq \Aut(X)$ with $d=2$,
\item[(4B)] $\frakS_4 \subseteq \Aut(X)$ with $d=1$,
\item[(5A)] $\frakS_5 \subseteq \Aut(X)$ with $d=0$.
\end{enumerate}
\end{lem}

\begin{proof}
As with the Sylvester form, if one sets the parameters of
\eqref{eq:EmchRescaled} equal to one another, we obtain subgroups of
$\frakS_4$ as automorphisms.  This exhibits automorphisms of the desired
forms for all except $\frakS_5$.
Each of these strata gives an affine subspace $W$ of $V_0+F_0$ with the
desired dimension $d$ from the proof of Theorem~\ref{thm:Emch}.
It was shown that a form with all parameters equal gives a point of
$V_0+F_0$ which meets transversely with the orbits of $\PGL_4(\bbk)$.
Such points are in $W$ for every case.
By upper semicontinuity, we conclude the orbits of $\PGL_4(\bbk)$ meet
$V_0+F_0$ in an open subset of the points of $W$ transversely as well.
The Fermat cubic has $\frakS_5$ symmetry when $p=2$ since
$\frakS_5 \subset \PSU_4(2)$.
\end{proof}

Lemmas~\ref{lem:SylvesterAutos}~and~\ref{lem:EmchAutos} give
open subsets of the
automorphism strata in the moduli space of cubic surfaces, which is
indicated by the labeling of the cases.
However, in general they are not the whole strata ---
the Sylvester or Emch form may only describe a general element in the
family.

\section{Rationality of the moduli space}
\label{sec:rationality}

Recall that the moduli space of complex cubic surfaces is known to be
rational. The proof either uses the explicit structure of the algebra of
invariants due to Clebsch, or the Sylvester form of a general cubic surface
(see \cite{CAG}, 9.4.3).
Using the latter, the proof goes as follows.
The space of cubics in Sylvester form can be
identified with an open subset of the affine space $\bbk^4$ spanned by
$c_1/c_0, \ldots, c_4/c_0$ (see Theorem~\ref{thm:Sylvester}).
Two forms correspond to isomorphic surfaces if and only if they are in
the same $\frakS_5$ orbit, where $\frakS_5$ acts on the affine space via
the standard action modulo the degree $1$ elementary symmetric function.
By the fundamental theorem of symmetric functions, the invariant ring is
a polynomial ring; thus the quotient is rational. Since Sylvester's
Theorem also applies in positive characteristic $p\ne 2,3$, the proof of
rationality carries over to these cases as well. In this section, using
results about normal forms of general cubic surfaces, we prove the
following: 

\begin{thm} \label{thm:rationality}
Over an algebraically closed field of arbitrary characteristic, the
coarse moduli space $\mathcal{M}_{\textrm{cub}}$ of smooth cubic
surfaces is a $4$-dimensional rational variety.
\end{thm}

\begin{proof}
By the above, we have to deal only with cases when $p = 2$ or $p = 3$.
First, we assume that $p = 3$.

A generic cubic surface can be put into the normal form
\ref{eq:emchp3S5}, where once again two normal forms correspond to
isomorphic surfaces if and only if they are in the same
$\frakS_5$-orbit.
Once again, the action on the affine space $\bbk^4$ spanned by
$d_0,\ldots,d_3$ is again the standard $4$-dimensional irreducible
representation of $\frakS_5$, so the quotient is rational.

Now we assume $p = 2$. 
Note that the existence of a canonical point is an open condition
on the moduli space, and the canonical point lying on the surface is a
closed condition.
Thus, the isomorphism classes of cubic surfaces from class (b) in
Proposition~\ref{prop:5.4} form a Zariski open subset in the moduli space.
By generality, we may assume that the coefficient $c$ in $(x_0+cx_1)^2$
is not zero, and by scaling  $x_1$, it is equal to $1$.
Write the cubic form $C(x_1,x_2,x_3)$ in the form
\[
C(x_1,x_2,x_3) = c_1x_1^3+x_1^2L(x_2,x_3)+x_1Q(x_2,x_3)+P(x_2,x_3).
\]
The surface is singular unless the homogeneous form $P$ has three
distinct roots in $\bbP(x_2:x_3)$,
which we may take to be $(1:0)$, $(0:1)$, and $(1:1)$.
Scaling, we may assume that $P(x_2,x_3) = x_2x_3(x_2+x_3)$.
By a linear change of coordinates via $x_2 \mapsto a_1x_1 + x_2$ and
$x_3 \mapsto a_2x_1 + x_3$, we may select $a_1$ and $a_2$ so that
the coefficients of $x_2^2$ and $x_2^3$ are zero in $Q(x_2,x_3)$.
Thus a general cubic surface is defined by
\[
x_0(x_0+x_1)^2+b_1x_1^3+x_1^2(b_2x_2+b_3x_3)+b_4x_1x_2x_3+x_2x_3(x_2+x_3) = 0
\]
for parameters $b_1,\ldots,b_4$.
Since the surface has a canonical point, we may assume $b_4 \ne 0$.
By scaling we obtain the form
\begin{equation} \label{eq:generic}
x_0(x_0+x_1)^2+(c_1/c_4)x_1^3+x_1^2(c_2x_2+c_3x_3)+x_1x_2x_3+c_4x_2x_3(x_2+x_3)
= 0 
\end{equation}
where $c_1,\ldots,c_4$ are parameters, which we may assume are all
non-zero.

Suppose two cubic surfaces in form \eqref{eq:generic} are isomorphic
(for possibly different values of $c_1,\ldots,c_4$).
Let $\psi \in \PGL_4(\bbk)$ be a map taking one form to the other.
Since the point $(1:0:0:0)$ is canonical, there is a unique
representative $A \in \GL_4(\bbk)$ for $\psi$ whose first column is the
first standard basis vector $(1,0,0,0)$.
We set $\psi^* : k[x_0,\ldots,x_3] \to k[x_0,\ldots,x_4]$
on the homogeneous coordinate ring
to be the dual action of $A$ on the corresponding total space $\bbA^4$.
Since the monomials $x_0^3$ and $x_0x_1^2$ are the only monomials in
\eqref{eq:generic} containing $x_0$, we see that $\psi^*(x_0)=x_0$
and $\psi^*(x_1)=x_1$.
Thus, we are left with
\begin{align*}
\psi^*(x_2) &= a_1x_1 + a_2x_2 + a_3x_3\\
\psi^*(x_3) &= a_4x_1 + a_5x_2 + a_6x_3
\end{align*}
where $a_1,\ldots, a_6$ are to be determined.

We explicitly determine the coefficients of the monomials in the
transformed equation:
\begin{align}
x_1x_2^2 &\colon c_4a_1a_5^2 + c_4a_2^2a_4 + a_2a_5 \label{eq:122} \\
x_1x_3^2 &\colon c_4a_1a_6^2 + c_4a_3^2a_4 + a_3a_6 \label{eq:133} \\
x_1x_2x_3 &\colon a_2a_6+a_3a_5 \ . \label{eq:123}
\end{align}

Let $B=\begin{pmatrix} a_2 & a_3 \\ a_5 & a_6 \end{pmatrix}$
be the bottom right $2 \times 2$ submatrix of $A$,
which represents the action of $\psi^*$ considering only the variables $x_2,
x_3$.
Since we assume $c_4 \ne 0$, the map $\psi^*$ must leave $x_2x_3(x_2+x_3)$ fixed.
From \eqref{eq:123}, we conclude that $B$ must have determinant $1$.
We conclude that $B$ is an invertible matrix consisting of $0$'s and
$1$'s.

The expressions \eqref{eq:122} and \eqref{eq:133} must be zero,
which amounts to a linear system of $a_1,a_4$ if the other variables are
fixed.
Since $\det(B) \ne 0$, this linear system has the unique solution
$a_1 = \frac{a_2a_3}{c_4}$ and
$a_4 = \frac{a_5a_6}{c_4}$.

Bearing in mind that $B$ is a non-singular matrix consisting only of
$0$'s and $1$'s,
one computes that $\psi$ induces the following action on the affine
space $\bbA^4$ spanned by $c_1,c_2,c_3,c_4$:
\begin{align*}
c_1 &\mapsto c_1 + a_2a_3c_2 + a_5a_6c_3\\
c_2 &\mapsto a_2c_2 + a_5c_3\\
c_3 &\mapsto a_3c_2 + a_6c_3\\
c_4 &\mapsto c_4 \ .
\end{align*}
These form a group isomorphic to $\frakS_3$ where, for example,
two generators are
obtained from
\[
B=\begin{pmatrix} 0 & 1 \\ 1 & 0 \end{pmatrix}
\textrm{ and }
B=\begin{pmatrix} 1 & 1 \\ 1 & 0 \end{pmatrix} \ .
\]
Using the coordinates $c_1'=c_1$, $c_2'=c_1+c_2$, $c_3'=c_1+c_3$,
$c_4'=c_4$, we see that the possible $\psi$'s act as the group $\frakS_3$ 
by permuting the parameters $c_1', c_2', c_3'$ and leaving $c_4'$ fixed.
By the fundamental theorem of symmetric functions,
the quotient $\bbA^4/\frakS_3$ is rational and, thus, so is the
moduli space of cubic surfaces.
\end{proof}

\section{Conjugacy classes of automorphisms}
\label{sec:ccs}

Our strategy for establishing the classification is as follows.
For each conjugacy class in $W(\sfE_6)$ we will describe the smooth cubic
surfaces that have an automorphism whose image in $W(\sfE_6)$ has that
class.
In the interest of brevity, we will write for example ``$X$ admits an
automorphism of class 3D'' or simply ``$X$ is a 3D surface.''
In each case, we will describe the automorphism group of each
surface and determine normal forms for the corresponding cubic surfaces.
In many cases, while we will actually describe the full automorphism
group initially, the proof that there are no further automorphisms will
be deferred until Section~\ref{sec:proof}.

Except for the cases 4A, 8A and 12A, we will show that the existence of
an automorphism of a given class implies the existence of a reflection
group acting on the surface which contains the given automorphism.
These will frequently turn out to be the full automorphism groups of the
generic surface admitting each class.

After we have gone through all the conjugacy classes, it will remain to
prove that we have indeed described the full automorphism groups of
these surfaces and that the corresponding strata in the moduli space have
the expected dimensions.  This matter will occupy
Sections~\ref{sec:EckardtCollections} and \ref{sec:proof}.

We recall the conjugacy classes of elements of $W(\sfE_6)$ in
Table~\ref{tbl:cyclicWE6} along with their trace and characteristic
polynomials of their actions on the standard $6$-dimensional vector
space (see Chapter~9~of~\cite{CAG}).

\begin{table}[h]
\begin{center}
\begin{tabular}{|c|r|l|c|l|}
\hline
Name & Trace & Char. Poly & Realizable & Powers \\
\hline
1A & $6$ & $\Phi_{1}^{6}$ & yes &  \\
\hline
2A & $-2$ & $\Phi_{1}^{2}\Phi_{2}^{4}$ & yes &  \\
2B & $2$ & $\Phi_{1}^{4}\Phi_{2}^{2}$ & yes &  \\
2C & $4$ & $\Phi_{1}^{5}\Phi_{2}^{1}$ &  &  \\
2D & $0$ & $\Phi_{1}^{3}\Phi_{2}^{3}$ &  &  \\
\hline
3A & $-3$ & $\Phi_{3}^{3}$ & yes &  \\
3C & $3$ & $\Phi_{1}^{4}\Phi_{3}^{1}$ & $p \ne 3$ &  \\
3D & $0$ & $\Phi_{1}^{2}\Phi_{3}^{2}$ & yes &  \\
\hline
4A & $2$ & $\Phi_{1}^{2}\Phi_{4}^{2}$ & yes & 2A \\
4B & $0$ & $\Phi_{1}^{2}\Phi_{2}^{2}\Phi_{4}^{1}$ & yes & 2B \\
4C & $-2$ & $\Phi_{1}^{1}\Phi_{2}^{3}\Phi_{4}^{1}$ &  & 2B \\
4D & $2$ & $\Phi_{1}^{3}\Phi_{2}^{1}\Phi_{4}^{1}$ &  & 2B \\
\hline
5A & $1$ & $\Phi_{1}^{2}\Phi_{5}^{1}$ & $p \ne 5$ &  \\
\hline
6A & $1$ & $\Phi_{3}^{1}\Phi_{6}^{2}$ & yes & 2A, 3A \\
6C & $1$ & $\Phi_{1}^{2}\Phi_{2}^{2}\Phi_{6}^{1}$ & $p \ne 3$ & 2A, 3C \\
6E & $-2$ & $\Phi_{2}^{2}\Phi_{3}^{1}\Phi_{6}^{1}$ & yes & 2A, 3D \\
6F & $-1$ & $\Phi_{1}^{2}\Phi_{2}^{2}\Phi_{3}^{1}$ & $p \ne 3$ & 2B, 3C \\
6G & $1$ & $\Phi_{1}^{3}\Phi_{2}^{1}\Phi_{3}^{1}$ &  & 2C, 3C \\
6H & $-2$ & $\Phi_{1}^{1}\Phi_{2}^{1}\Phi_{3}^{2}$ &  & 2C, 3D \\
6I & $0$ & $\Phi_{1}^{1}\Phi_{2}^{1}\Phi_{3}^{1}\Phi_{6}^{1}$ &  & 2D, 3D \\
\hline
8A & $0$ & $\Phi_{1}^{1}\Phi_{2}^{1}\Phi_{8}^{1}$ & $p \ne 2$ & 2A, 4A \\
\hline
9A & $0$ & $\Phi_{9}^{1}$ & $p \ne 3$ & 3A \\
\hline
10A & $-1$ & $\Phi_{1}^{1}\Phi_{2}^{1}\Phi_{5}^{1}$ &  & 2C, 5A \\
\hline
12A & $-1$ & $\Phi_{3}^{1}\Phi_{12}^{1}$ & yes & 2A, 3A, 4A, 6A \\
12C & $1$ & $\Phi_{1}^{1}\Phi_{2}^{1}\Phi_{4}^{1}\Phi_{6}^{1}$ &  & 2B, 3C, 4C, 6F \\
\hline
\end{tabular}
\end{center}
\caption{Cyclic subgroups of $W(\sfE_6)$}
\label{tbl:cyclicWE6}
\end{table}

Fortuitously, the characteristic polynomials uniquely determine the
conjugacy classes.
Using the characteristic polynomials, one can easily compute the classes
of the non-trivial powers of a given element.
The table also lists under what characteristic assumptions the given
classes can be realized by an action on a cubic surface;
we will explicitly construct realizations in the following sections.
However, it is easy to see that the many of the remaining classes are
\emph{not} realizable:

\begin{lem} \label{lem:excludedClasses}
The classes 2C, 2D, 4C, 4D, 6G, 6H, 6I, 10A and 12C are not realizable
for a cubic surface of any characteristic.
Additionally, if $p=2$ then 8A is not realizable;
if $p=3$ then 3C, 6C, 6F are not realizable.
\end{lem}

\begin{proof}
Suppose that $\alpha$ is an automorphism of a smooth cubic surface and
$g$ is the corresponding element of $W(\sfE_6)$, which is well-defined
up to conjugacy.
First, suppose that there exists an element $\tilde{g}$ of $W(\sfD_5)$
such that $g$ is in the image of a conjugate of the standard
embedding $W(\sfD_5) \hookrightarrow W(\sfE_6)$.
In this case, there exists an
exceptional divisor which can be blown down which realizes $\tilde{g}$
on a del Pezzo surface of degree $4$.
Consulting Table~\ref{tbl:D5ccs}, we may thus exclude
2C, 2D, 4C, 4D, 6G, 12C.  In characteristic $2$, we may exclude 8A.
In characteristic $3$, we may exclude 3C, 6C, and 6F.

The class 6H can be excluded since its cube has class 2C, which has
already been excluded.  Similarly, 6I can be excluded since its cube has
class 2D, and 10A can be excluded since its fifth power is 2C.
\end{proof}

There are two additional excluded cases not covered by
Lemma~\ref{lem:excludedClasses} since the corresponding proofs are a bit
more involved.
When $p=3$, we will see that the class 9A does not occur
by Lemma~\ref{lem:mush3C6C6F9A} below.
When $p=5$, we will see that the class 5A does not occur by
Lemma~\ref{lem:exclude5A} below.

\begin{remark} \label{rem:Hosoh}
In characteristic $0$, T.~Hosoh~\cite{Hosoh} determined all the possible
automorphism groups of a smooth cubic surface using computers.
His method was to explicitly construct the moduli space of marked
cubic surfaces and explicitly determine the action of $W(\sfE_6)$ on this
space.
The possible automorphism groups are the non-trivial stabilizers of this
action.
The same approach works in arbitrary characteristic, but is not feasible
to do by hand and does not produce normal forms.
\end{remark}

\begin{remark} \label{rem:Saito}
In Section~9.5~of~\cite{CAG}, all the 
automorphism groups of smooth cubic surfaces were classified in characteristic $0$.
The approach there was to first consider the cyclic automorphisms.
These were determined by considering every possible diagonal matrix with
order the same as an element of $W(\sfE_6)$ and determining whether they
leave a smooth cubic hypersurface invariant.
The class of each automorphism in $W(\sfE_6)$ was then computed using the
Lefschetz fixed point formula.
One may carry out an analogous procedure in
positive characteristic.
Here, one must also consider Jordan canonical forms rather than only
diagonal matrices and
to compute the classes in $W(\sfE_6)$ one can use a fixed point
formula of S.~Saito~\cite{Saito}.
We use some of these ideas in what follows, but we do not carry out this
program systematically.
\end{remark}

\section{Involutions}
\label{sec:involutions}

\begin{prop} \label{prop:involutions}
Let $X$ be a smooth cubic surface in the space $\bbP(V)$ where $V$
is a four-dimensional vector space over an algebraically closed field of
arbitrary characteristic.
Suppose $g$ is an involution of $\bbP(V)$ leaving $X$ invariant.
Then one of the following holds:
\begin{enumerate}
\item
Up to choice of coordinates,
$g$ has a lift $\widetilde{g} \in \GL(V)$ of the form
\begin{equation} \label{eq:2Amatrix}
\widetilde{g} =
\left(\begin{smallmatrix}
0&1&0&0\\
1&0&0&0\\
0&0&1&0\\
0&0&0&1
\end{smallmatrix}\right)\ .
\end{equation}
The element $g$ corresponds to an element of class 2A in $W(\sfE_6)$.
The automorphism $g$ only leaves invariant three exceptional lines, which lie in a
tritangent plane $T$.
The remaining 24 lines are partitioned into 12 pairs of lines,
where each pair lies in a tritangent plane containing one of the lines in
$T$;
the involution $g$ interchanges the exceptional lines in each pair.
The lines in any non-trivial orbit are incident.
\item
Up to choice of coordinates,
$g$ has a lift $\widetilde{g} \in \GL(V)$ of the form
\begin{equation} \label{eq:2Bmatrix}
\widetilde{g} =
\left(\begin{smallmatrix}
0&1&0&0\\
1&0&0&0\\
0&0&0&1\\
0&0&1&0
\end{smallmatrix}\right)\ .
\end{equation}
The element $g$ corresponds to an element of class 2B in $W(\sfE_6)$.
There is a unique $g$-invariant exceptional line $\ell$ which
is incident to every other $g$-invariant exceptional line.
The line $\ell$ is pointwise-fixed by $g$.
There are $5$ pairs of exceptional lines that,
along with the pointwise-fixed exceptional line $\ell$,
form $g$-invariant tritangent planes.
In 3 of these pairs each line is $g$-invariant; in the remaining 2
pairs the lines are interchanged by the action of $g$.
The remaining 16 lines are partitioned into pairs of skew lines
interchanged by the action of $g$.
\end{enumerate}
\end{prop}

\begin{proof}
The possible matrix representations follow immediately by writing down
all Jordan canonical forms for an element of order $2$.
In characteristic $\ne 2$, there are $4$ such forms but they reduce to
only $2$ distinct forms in $\PGL(V)$.
The matrices in the statement of the theorem were selected to be
independent of the characteristic.

Every involution on a cubic surface $X$ must leave invariant some
exceptional line since there is an odd number of them.
Blowing down one of these lines, we obtain a del Pezzo surface $Y$ of degree
$4$ with the action of an involution.
From Table~\ref{tbl:D5ccs}, we conclude that elements of class 2A on $X$
correspond to elements of class $\bar{1}\bar{1}\bar{1}\bar{1}1$ on $Y$,
and elements of class 2B on $X$ correspond to elements of class
$\bar{1}\bar{1}111$ or $221$ on $Y$.
As in Lemma~\ref{lem:excludedClasses}, these are the only possible classes.

Choosing a $g$-invariant line $\ell$, we may determine the descriptions
of the orbits of the $27$ lines by identifying $\ell$ with the
exceptional divisor $E_6$
and then considering the actions of involutions of the first and second
kind on the $16$ lines of the del Pezzo surface of degree $4$ obtained
by blowing down $\ell$.
The action on the remaining 11 lines can be determined from incidence
relations.

In this manner, the orbit structure can be verified on a single example.
The involution $\iota_{1234}$ has orbits
\begin{equation} \label{eq:27orbits2A}
\{ E_6\}, \{F_{56}\}, \{G_5 \}, \{ G_i, F_{i6} \},
\{E_5,G_6\}, \{E_i,F_{i5}\}, \{ F_{ij}, F_{kl} \}
\end{equation}
for $\{i,j,k,l\}=\{1,2,3,4\}$.
The invariant lines form a tritangent plane with the desired incidence
properties.
The involution $\iota_{12}$ has orbits
\begin{equation} \label{eq:27orbits2B}
\{E_6\}, \{F_{i6}\},\{G_i\}, \{G_a,F_{a6}\},
\{E_1,E_2\}, \{G_6,F_{12}\},\{E_i,F_{jk}\},
 \{F_{ai},F_{bi}\}
\end{equation}
for $\{a,b\}=\{1,2\}$ and $\{i,j,k\}=\{3,4,5\}$.
Here $E_6$ is the canonical invariant line,
the pairs $F_{i6},G_i$ give three tritangent planes,
while the orbits $\{G_a,F_{a6}\}$ give the other two.
Since $E_6$ is incident to $6$ invariant lines, it must have at least
$3$ $g$-fixed points and so must be pointwise-fixed by $g$.

It remains to connect the matrix representations to the conjugacy
classes.
Note that \eqref{eq:2Amatrix} contains a plane of fixed points,
while
\eqref{eq:2Bmatrix} contains at most two lines of fixed points.
Suppose $g$ has a fixed plane in $\bbP(V)$.
Then every exceptional line on the cubic surface $X$ contains at least
one fixed point.
Blowing down a $g$-invariant exceptional line, this implies that the
surface $Y$ cannot have an orbit of skew lines and so $g$ is of the
class 2A.
Conversely, suppose $g$ is of class $2A$.
The fixed point locus $C$ on $Y$ has dimension $1$.
If $p \ne 2$, then $C$ is a curve of genus $1$.
The proper transform of $C$ on $X$ cannot be contained in two lines,
so the matrix must be \eqref{eq:2Amatrix}.
If $p = 2$, then the matrix \eqref{eq:2Bmatrix} only has a single line
of fixed points.  This line cannot lie on $X$ since blowing it
down would result in a fixed point locus on $Y$ of dimension $0$.
But if it does not lie on $X$ then there are at most three fixed points ---
again the fixed point locus on $Y$ has dimension $0$.
We conclude that the matrix must be \eqref{eq:2Amatrix}.
\end{proof}

\subsection{Reflections and the class 2A}
\label{sec:2A}

Here we consider involutions $\widetilde{g}$ of the form \eqref{eq:2Amatrix}.
These involutions leave pointwise-fixed a proper subspace of maximal
dimension, so they are \emph{reflections}.
(Note that these reflections are \emph{not} the same as the reflections
in the standard representation of Weyl group $W(\sfE_6)$.)
The corresponding projective involution $g$ is classically called a
\emph{homology}, but we will use the term \emph{reflection} in the
projective case as well.
The locus of fixed points of $g$ consists of a hyperplane $A$ and a
unique point $q_0$,
which correspond respectively to the kernel and image of $\widetilde{g}-\id$.
The hyperplane $A$ is the \emph{axis} of the
reflection; the point $q_0$, the \emph{center} of the reflection.
If $p \ne 2$, then the center never lies on $A$.
If $p = 2$, then the center always lies on $A$.

\begin{thm} \label{thm:Eckardt}
The center $q_0$ of the reflection $g$ is an Eckardt
point on $X$.
Conversely, any Eckardt point on $X$ is the center of a
reflection leaving $X$ invariant.
The tritangent plane $T$ corresponding to $q_0$ contains precisely the
$g$-invariant exceptional lines.
The polar quadric $P_{q_0}(X)$ is supported on the union of the
tritangent plane $T$ and the axis $A$ of the reflection;
moreover, $T$ and $A$ may coincide only when $p=2$.
\end{thm}

\begin{proof}
Suppose $g$ is a reflection with center $q_0$.
We will show that any $g$-invariant line $\ell_0$ must contain $q_0$.
A $g$-invariant line must either be contained in the axis $A$ or contain $q_0$.
When $p \ne 2$,
blowing-down a pointwise-fixed line results in an isolated $g$-fixed
point on the corresponding del Pezzo surface of degree $4$,
contradicting Proposition~\ref{prop:inv12kind};
thus $\ell_0$ is not contained in the axis.
Thus, when $p \ne 2$, the line $\ell_0$ must contain $q_0$.
Now, suppose that $p=2$ and $\ell_0$ is contained in the axis $A$ of $g$.
Suppose $q_0\not \in \ell_0$.
Then any plane $H \ne A$ containing $\ell_0$ is not $g$-invariant.
In particular, every tritangent plane $H$ containing $\ell_0$
is not $g$-invariant.
But there are exactly $5$ tritangent planes containing a given
exceptional line; since $5$ is odd, at least one of them must be $g$-invariant
--- a contradiction.
Thus, regardless of the characteristic, $\ell_0$ contains the center $q_0$ of $g$.
From Proposition~\ref{prop:involutions}, there are $3$ $g$-invariant
lines and so they must all pass through $q_0$, which shows that $q_0$
is an Eckardt point.

Let us prove the converse. Assume first that $p\ne 2$.
We may use the same argument as~Theorem~9.1.23 from \cite{CAG}.
Let $q_0$ be the Eckardt point and let $P_{q_0}(X)$ be the polar quadric
at the point $q_0$.
It is equal to the union of the tangent plane $T$ at $q_0$ containing
three lines in $X$ passing through $q_0$ and a plane $A$ intersecting
$X$ along a cubic curve $C$.
We may choose coordinates such that $q_0 = (1:0:0:0)$
and $T= V(x_1)$.
If $q_0 \in A$ then we may take $A=V(x_1+\lambda x_2)$
and find that
\begin{equation}\label{bad}
x_0x_1(x_1+\lambda x_2)+a(x_1,x_2,x_3) = 0
\end{equation}
which is singular at $q_0$.
Thus $q_0 \notin A$ and we may assume $A = V(x_0)$.
The equation of $X$ can be written in the form
\begin{equation}\label{nform1}
x_0^2x_1+a(x_1,x_2,x_3) = 0.
\end{equation}
An involution acts as $(x_0:x_1:x_2:x_3)\mapsto (-x_0:x_1:x_2:x_3)$
leaving $X$ invariant and fixing $q_0$ as desired.

Assume $p=2$.
A similar argument works in this case.
Suppose $q_0=(1:0:0:0)$ is our Eckardt
point.  The equation for $X$ must have the form
\[
x_0^2L(x_1,x_2,x_3) + x_0Q(x_1,x_2,x_3) + C(x_1,x_2,x_3) =0
\]
for homogeneous forms $L,Q,C$ of degrees $1,2,3$ respectively.
Note that $L=0$ defines the tangent plane at $q_0$.
Since $q_0$ is an Eckardt point, the polar quadric $P_{q_0}(X)$ contains
the tangent plane at $q_0$.  Thus $Q$ is a product of $L$ and some
linear form which may assume is $x_1$.
We obtain the normal form
\begin{equation}\label{normzero}
x_0(x_0+x_1)L(x_1,x_2,x_3) + C(x_1,x_2,x_3) = 0,
\end{equation} 
which has an action of the involution
$(x_0:x_1:x_2:x_3)\mapsto (x_0+x_1:x_1:x_2:x_3)$
which has axis $x_1=0$ and center $q_0$.

By a linear change of coordinates, there are two different possibilities
for $L$, which correspond to different geometric situations.
Either we can reduce to the form
\begin{equation}\label{normone}
x_0(x_0+x_1)x_2 + C(x_1,x_2,x_3) = 0,
\end{equation}
where the axis and the tangent plane are distinct;
or we have
\begin{equation}\label{normtwo}
x_0(x_0+x_1)x_1+ C(x_1,x_2,x_3) = 0,
\end{equation}
where the axis and the tangent plane coincide.
\end{proof}

\begin{prop} \label{prop:normalForms2A}
Suppose $g$ is an involution of type $2A$ acting on a cubic surface $X$.
If $p \ne 2$ then $X$ can be defined by
\begin{equation} \label{eq:2Anormal0}
(x_0+x_1)(x_0x_1+c_0x_2^2+c_1x_3^2+c_2x_2x_3)+x_2x_3(x_2+x_3) = 0\ .
\end{equation}
If $p=2$ and none of the exceptional lines passing through the
corresponding Eckardt point lie on the axis, then $X$ is
defined by
\begin{equation}\label{eq:2Anormal2}
x_0x_1x_2 + (x_0+x_1)^2x_3 +
c_0(x_0+x_1)x_3^2 + c_1x_2^3 + c_2x_2x_3^2 + x_3^3 =0\ .
\end{equation}
In either case, $c_0,c_1,c_2$ are parameters and
$g(x_0:x_1:x_2:x_3)=(x_1:x_0:x_2:x_3)$. 
\end{prop}

We will see below that failure of the extra condition for $p=2$ is equivalent to
the existence of an involution of type 2B --- we will find normal forms
for the remaining surfaces in Proposition~\ref{prop:normalForms2B} below.

\begin{proof} 
Assume $p\ne 2$.
Replacing $g$ by a conjugate we may assume it acts by negation of $x_0$.
Then the equation can be written in the form
\[ x_0^2L(x_1,x_2,x_3)+F(x_1,x_2,x_3) = 0,\] 
where $L,F$ are a linear and cubic homogeneous forms.
We may assume that $L = x_1$.
The union of the exceptional lines passing through the Eckardt point 
$(1:0:0:0)$ is given by equation $x_1 = F(x_1,x_2,x_3) = 0$.
By a change of variables of $x_2,x_3$, we 
may assume that $F(0,x_2,x_3) = x_2x_3(x_2+x_3)$.
Thus the equation can be rewritten in the form
\[
x_0^2x_1+a_1x_1^3+x_1^2(a_2x_2+a_3x_3)+x_1(a_4x_2^2+a_5x_3^2+a_6x_2x_3)+x_2x_3(x_2+x_3)= 0.
\]
Via a change of coordinates $x_2 \mapsto x_2 + \lambda x_1$,
$x_3 \mapsto x_3 + \mu x_1$ we may assume $a_2=a_3=0$.
Note that now $X$ is singular unless $a_1 \ne 0$.
Thus, by rescaling $x_0$ and $x_1$ we may assume $a_1=1$ to obtain
\[
x_0^2x_1+x_1^3+x_1(a_4x_2^2+a_5x_3^2+a_6x_2x_3)+x_2x_3(x_2+x_3)= 0.
\]
Replacing $x_0,x_1$ with
$\frac{i}{\sqrt[3]{4}}(x_0-x_1),\frac{1}{\sqrt[3]{4}}(x_0+x_1)$ 
we arrive at the asserted normal form.

Assume $p = 2$.
We have seen already that the equation of $X$ can  be reduced to one of the
forms \eqref{normone} or \eqref{normtwo}.
By the given condition, we may assume it is \eqref{normone}.
Write $F$ in the form
\[F=
x_0(x_0+x_1)x_2+a_1x_2^3+x_2^2A(x_1,x_3)+x_2B(x_1,x_3)+C(x_1,x_3).
\]
In the tritangent plane $x_2=0$ there are $3$ distinct lines;
since none of them lie on the axis, the line $x_1=0$ is distinct from
all three.
These correspond to $4$ points in $\bbP^1(x_1,x_3)$ which we may assume
are given by $x_1=0$, $x_3=0$, $x_1+\lambda x_3=0$, $\lambda x_1 + x_3=0$
for some non-zero parameter $\lambda$.  Thus
$C(x_1,x_3)=x_3(x_1+\lambda x_3)(\lambda x_1 + x_3)
=\lambda x_3(x_1^2+(\lambda+\lambda^{-1})x_1x_3+x_3^2)$.
After rescaling $x_3$ we have
\begin{gather*}
F= x_0(x_0+x_1)x_2+a_1x_2^3+a_2x_1x_2^2+a_3x_2^2x_3\\
+x_2(a_4x_1^2+a_5x_1x_3+a_6x_3^2)
+x_3(x_1^2+a_7x_1x_3+x_3^2).
\end{gather*}
The change of coordinates $x_3 \mapsto x_3+\sqrt{a_3}x_2$
sets $a_3$ to $0$ without otherwise changing the form of the equation.
The change $x_0 \mapsto x_0+a_5x_3$ now kills $a_5$.
Now, $x_0 \mapsto x_0+a_2x_2$ kills $a_2$.
Finally, $x_0 \mapsto x_0+cx_1$ where $c^2+c=a_4$ kills $a_4$.
Replace $x_0,x_1$ with $x_0,x_0+x_1$ to arrive at the asserted normal form.
\end{proof}

\subsection{Involutions of class 2B}
\label{sec:2B}

Let $\ell$ be an exceptional line on a smooth cubic surface $X$.
Every plane containing $\ell$ intersects $X$ along $\ell$ and a
residual conic.  In this way, a pencil of planes containing $\ell$
defines a pencil of conics on $X$.  Restricting to $\ell$, we obtain a
linear system $g^1_2$ on $\ell$ of divisors of degree $2$ or,
equivalently, a regular map $g^1_2 : \ell \to \bbP^1$.
We call $\ell$ a \emph{separable} (resp. \emph{inseparable}) line
if $g^1_2$ is separable (resp. inseparable).

\begin{lem} \label{lem:excLineEck}
Let $\ell$ be an exceptional line on a smooth cubic surface $X$.
If $p \ne 2$, then $\ell$ contains $0$, $1$, or $2$ Eckardt
points.
If $p=2$, then $\ell$ contains $0$, $1$, or $5$ Eckardt
points.
The line $\ell$ is inseparable if and only if it contains $5$ Eckardt
points.
\end{lem}

\begin{proof}
The map $g^1_2 : \ell \to \bbP^1$
is ramified precisely when the corresponding conic is tangent to $\ell$.
A tangent conic can either be irreducible, or can be the union of two
lines that, together with $\ell$, span a tritangent plane with the
tangency point an Eckardt point.

When $p \ne 2$, the line $\ell$ is separable and it contains at most $2$
Eckardt points.
When $p=2$, if $\ell$ is separable, then there is at most one Eckardt
point since an involution acting on $\bbP^1$ has at most one fixed
point.
If $\ell$ is inseparable, then every tritangent plane containing $\ell$
must contain an Eckardt point on $\ell$.  There are always $5$ such
planes.
\end{proof}

\begin{lem} \label{lem:2Bautos}
Let $X$ be a smooth cubic surface admitting an automorphism $g$ of class 2B.
Let $\ell$ be the unique exceptional line pointwise-fixed by $g$.
There are $2$ Eckardt points on $\ell$ whose corresponding reflections
$h_1, h_2$ generate a group isomorphic to $2^2$
containing $g$.
Moreover, if $p = 2$, then $\ell$ contains $5$ Eckardt points whose corresponding
reflections generate a group isomorphic to $2^4$ containing $g$;
the line $\ell$ lies in the axis of every one of these reflections.
\end{lem}

\begin{proof}
From Proposition~\ref{prop:involutions}, there is a canonical
pointwise-fixed exceptional line $\ell$, which is incident to all other $g$-invariant
exceptional lines.
Blow down $\ell$ to a point $q$ on a del Pezzo surface $Y$ of degree $4$
where $g$ acts as an involution of the second kind.
There, $g$ is a unique product of two commuting involutions $h_1$ and
$h_2$ of the first kind.
Note that $h_1$ and $h_2$ both leave $\ell$ invariant since
$q$ is contained in the fixed locus of both.
From Proposition~\ref{prop:involutions}, involutions of class 2A leave
invariant only the exceptional divisors passing through their
corresponding Eckardt points.
Thus, the two Eckardt points lie on $\ell$.

Now we consider $p=2$.
By Proposition~\ref{prop:inv12kind}, $g$ has a unique fixed point on
$Y$, namely the canonical point.
The canonical point is fixed by all automorphisms of $Y$,
so the entire automorphism group of $Y$ acts on $X$.
In general, this is the group $2^4$ generated by the $5$ involutions of the
first kind.
As above, they must all leave $\ell$ invariant so $\ell$ contains all $5$
corresponding Eckardt points.
Since the $5$ reflections commute, they each fix the $5$ Eckardt points.
An automorphism cannot fix $5$ points on $\bbP^1$ unless it acts
trivially.
We conclude that $\ell$ is in the axis of each of these reflections.
\end{proof}

\begin{remark}
If $p\ne 2$, the fact that the pointwise-fixed line contains two Eckardt
points can be also deduced from the Lefschetz fixed-point formula.
Since the trace of an element of type 2B on the second cohomology group
is equal to $3$, the Euler-Poincar\'e characteristic of the fixed locus
$X^g$ is equal to $5$.
Thus $X^g$ consists of the line and 3 isolated fixed points.
There are five $g$-invariant tritangent trios containing the fixed line.
Since the singular points of tritangent trios containing the line
are fixed, at most three of the singular points do not lie on the fixed
line.  Thus there are two tritangent trios whose singular points lie on
the line --- these are the two Eckardt points.
\end{remark}

\begin{cor} \label{cor:2Bconverse}
Two reflections $r_1, r_2$ commute if and only if their corresponding
Eckardt points lie on a common exceptional line $\ell$;
in this case, the product $r_1r_2$ has class 2B and its canonical line
is $\ell$.
If $\ell$ is an exceptional line containing two (or more) Eckardt
points, then their corresponding reflections commute.
\end{cor}

\begin{prop} \label{prop:normalForms2B}
Let $X$ be a 2B surface. If $p\ne 2$, then
$X$ can be defined by the equation 
\begin{equation}\label{2B}
F = x_0^2(x_2+c_0x_3) + x_1^2(c_1x_2+x_3) + x_2x_3(x_2+x_3) = 0,
\end{equation}
where $g(x_0:x_1:x_2:x_3) = (-x_0:-x_1:x_2:x_3)$.
If $ p = 2$, then $X$ can be defined by the equation
\begin{equation}\label{eq:2Bnormal0}
F=x_0^3+x_1^3 + x_2^3+x_3^3 +
c_0x_2x_3(x_0+x_1)+c_1x_0x_1(x_2+x_3) = 0,
\end{equation}
where $g(x_0:x_1:x_2:x_3) = (x_1:x_0:x_3:x_2)$.
\end{prop}

\begin{proof}
Assume $p \ne 2$.
By Lemma~\ref{lem:2Bautos}, we may in fact assume that
$X$ is invariant under the group $H \cong 2^2$ where
$x_0 \mapsto \pm x_0$ and $x_1 \mapsto \pm x_1$ while
$x_2, x_3$ remain fixed.
The invariant cubic has the form
\[
x_0^2L_1(x_2,x_3) + x_1^2L_2(x_2,x_3) + C_3(x_2,x_3)
\]
where $L_1,L_2,C$ are homogeneous forms of degree $1,1,3$ respectively.
One checks that $L_1,L_2,C$ must have distinct roots or else $X$ is
singular.  Thus, we obtain the desired normal form.

Assume $p = 2$.
By Lemma~\ref{lem:2Bautos}, we may assume that $X$ is invariant under
the transposition that interchanges $x_0,x_1$ and also the transposition
that interchanges $x_2,x_3$;
together these generate a group $H$ isomorphic to $2^2$.
The invariant ring under $H$ is the polynomial ring
$\bbk[\sigma_1,\sigma_2,\tau_1,\tau_2]$ where
\[
\sigma_1 := x_0+x_1,\ \sigma_2 := x_0x_1,\ 
\tau_1 := x_2+x_3,\ \tau_2 := x_2x_3 \ .
\]
Thus $F$ has the form
\[
c_1\sigma_1\tau_2 + c_2\sigma_2\tau_1 + c_3\sigma_1\sigma_2 + c_4\tau_1\tau_2
+ c_5\sigma_1^3 + c_6\sigma_1^2\tau_1 + c_7\sigma_1\tau_1^2 + c_8\tau_2^3
\]
where $c_1, \ldots, c_8$ are parameters.

One finds that the partial derivatives at the point
$q=(\lambda,\lambda,\mu,\mu)$ are
\begin{align*}
\partial_{x_0}F|_q = \partial_{x_1}F|_q &= c_1\mu^2+c_3\lambda^2\\
\partial_{x_2}F|_q = \partial_{x_3}F|_q &= c_2\lambda^2+ c_4\mu^2
\end{align*}
and so the form is singular if $c_1c_2 = c_3c_4$.

Consider the linear map $m$ given by the matrix
\[
M = I_4 + \begin{pmatrix}
a_0J_2 & a_1J_2 \\ a_2J_2 & a_3J_2
\end{pmatrix}
\]
where $I_n$ is the $n \times n$ identity matrix, $J_2$ is the $2 \times
2$ matrix with all entries equal to $1$, and $a_0,\ldots,a_3$ are
constants.  Note that $M$ has order $2$, and is therefore invertible,
for all values of $a_i$.
We find that
\begin{align*}
m(\sigma_1) &= \sigma_1 \\
m(\sigma_2) &=
\sigma_2 + \wp(a_0)\sigma_1^2 + a_1\sigma_1\tau_1 + a_1^2\tau_1^2 \\
m(\tau_1) &= \tau_1 \\
m(\tau_2) &=
\tau_2 + a_2^2\sigma_1^2 + a_2\sigma_1\tau_1 + \wp(a_3)\tau_1^2
\end{align*}
where $\wp(t):=t^2+t$ is the Artin-Schreier map.

If $c_1', \ldots, c_8'$ are the corresponding coefficients of $m(F)$,
then we see that $c_1'=c_1,\ldots, c_4'=c_4$ and
\begin{align}
c_5' &= c_5 + c_1a_2^2 + c_3\wp(a_0) \label{eq:c5} \\
c_6' &= c_6 + c_1a_2 + c_2\wp(a_0) + c_3a_1 + c_4a_2^2 \label{eq:c6}\\
c_7' &= c_7 + c_1\wp(a_3) + c_2a_1 + c_3a_1^2 + c_4a_2 \label{eq:c7}\\
c_8' &= c_8 + c_2a_1^2 + c_4\wp(a_3) \label{eq:c8} \ .
\end{align}
We claim that we may select $a_0,\ldots,a_3$ so that
$c_5'=c_3, c_6'=c_7'=0, c_8'=c_4$.
In this case, the new form is
\[
c_1'\sigma_1\tau_2 + c_2'\sigma_2\tau_1 + c_3'(\sigma_1\sigma_2 +
\sigma_1^3)
+ c_4'(\tau_1\tau_2 +\tau_1^3)\ ,
\]
which can be rescaled to be of the form in the theorem
($c_3'$ and $c_4'$ are non-zero since $X$ is smooth).

It remains to prove the claim.
If $c_1=0$, then $c_3,c_4 \ne 0$ or else the form is singular.
Since \eqref{eq:c5} is now an equation only involving $a_0$, it has a
solution.  Substituting this solution into \eqref{eq:c6},
we see that $a_1 = \frac{c_4}{c_3}a_2^2 + K$ where $K$ is a constant.
Substituting this into \eqref{eq:c7} we obtain a univariate
polynomial in $a_2$ with highest term $\frac{c_4^2}{c_3}a_2^4$;
thus $a_1$ and $a_2$ have solutions.  Finally, \eqref{eq:c8}
is now a univariate polynomial in $a_3$ with highest term $c_4a_3^2$;
thus the whole system has a solution.

The case of $c_2=0$ follows by symmetry.
Thus, we may assume $c_1,c_2$ are both non-zero.
From \eqref{eq:c6} and \eqref{eq:c7}, we obtain expressions for
$\wp(a_0)$ and $\wp(a_3)$ in terms of $a_1$ and $a_2$.
Substituting these into $\eqref{eq:c5}$ and $\eqref{eq:c8}$ we now
simply must find solutions for $a_1,a_2$ for the system
\begin{align*}
(c_3c_4+c_1c_2)a_2^2 + c_1c_3a_2 + c_3^2a_1 &= K_1\\
(c_3c_4+c_1c_2)a_1^2 + c_2c_4a_1 + c_4^2a_2 &= K_2
\end{align*}
where $K_1,K_2$ are constants.
There is a solution if $c_3=c_4=0$, so we may assume, without
loss of generality, that $c_3 \ne 0$.
Recall $c_3c_4 + c_1c_2 \ne 0$ or else $X$ is singular.
The first equation now provides an expression for $a_1$
in terms of $a_2$ which has a quadratic highest term.
Substituting this into the second equation, we obtain a quartic in $a_2$.
Thus the system has a solution.
\end{proof}

\begin{remark}
When $p\ne 2$, in the normal form from Proposition~\ref{prop:normalForms2B},
the exceptional line $\ell$ is defined by
\[
x_2=x_3=0
\]
with two Eckardt points $(1:0:0:0)$ and $(0:1:0:0)$.
Note that the five tritangent planes containing $\ell$ are precisely
those defined by the linear forms $x_2$, $x_3$, $x_2+x_3$, 
$x_2+c_0x_3$ and $c_1x_2+x_3$.

For $p=2$, the exceptional line is defined by $x_0+x_1=x_2+x_3$
with the two associated Eckardt points $(1:1:0:0)$ and $(0:0:1:1)$;
and the canonical point is $(c_0:c_0:c_1:c_1)$.
\end{remark}

\section{Automorphisms of order 3}
\label{sec:order3}

\subsection{Trihedral Lines}
\label{sec:trihedralPairs}

Recall from section 9.1 of \cite{CAG} that
a \emph{tritangent trio} is a set of three exceptional lines which span
a tritangent plane.
A pair of tritangent planes intersecting along a line not in the surface
is classically known as a \emph{Cremona pair}.
A \emph{triad of tritangent trios} is a set of three tritangent trios
such that no two share a common exceptional line.
Given a pair of tritangent trios with no common exceptional lines,
there is a unique third tritangent trio forming a triad with the first
two.
Every triad of tritangent trios has a unique \emph{conjugate triad} which
contains the same $9$ exceptional lines.
We call these a \emph{conjugate pair of triads of tritangent trios}.

Geometrically, a \emph{trihedron} will be the set of three tritangent
planes corresponding to a triad of tritangent trios.
Similarly, one defines a \emph{conjugate trihedron} and a
\emph{conjugate pair of trihedra} (or simply \emph{trihedral pair}).

The Weyl group $W(\sfE_6)$ acts transitively on the set of conjugate pairs
of triads.
A representative example is the following set of $9$ exceptional lines:
\begin{equation} \label{eq:conjugateTriads}
\begin{matrix}
F_{14} & F_{25} & F_{36} \\
F_{26} & F_{34} & F_{15} \\
F_{35} & F_{16} & F_{24} \\
\end{matrix}
\end{equation}
where each row and column forms a tritangent trio,
and the rows form one triad with the columns its conjugate triad.

It is worth emphasizing that all of the aforementioned objects are
features of the abstract configuration of $27$ lines --- their existence
is completely independent of a choice of cubic surface.
On the other hand, the existence of Eckardt points and trihedral lines
depend on the specific cubic surface.

Recall that a \emph{trihedral line} $\ell$ is a line containing exactly
$3$ Eckardt points, which is not contained in the cubic surface.
Similarly to how Eckardt points correspond to the unusual situation
where the three exceptional lines in a tritangent plane have a common
point,
a trihedral line corresponds to the situation where three tritangent
planes in a trihedron contain a common (necessarily non-exceptional) line.

\begin{lem}
\label{lem:trihedra}
Let $X$ be a smooth cubic surface.
Consider a trihedral pair $((T_1, T_2, T_3), (T_1', T_2', T_3'))$,
where $\{T_i\},\{T'_i\}$ are tritangent trios.
The following are equivalent:
\begin{enumerate}
\item[(a)] $T_1 \cap T_2 \cap T_3$ is a line,
\item[(b)] $T_1', T_2', T_3'$ correspond to Eckardt points on a line, and
\item[(c)] two of $T_1', T_2', T_3'$ correspond to Eckardt points.
\end{enumerate}
\end{lem}

\begin{proof}
It suffices to assume that the 9 exceptional lines are those from
\eqref{eq:conjugateTriads} where $T_1, T_2, T_3$ correspond to the
columns and $T_1', T_2', T_3'$ correspond to the rows.

Assume (a).
If $T_1 \cap T_2 \cap T_3$ contain a line then
$T_1 \cap T_2 = T_1 \cap T_3$.
Thus
$F_{14} \cap F_{25}=T_1 \cap T_2 \cap T_1'=
T_1 \cap T_3 \cap T_1'=F_{14} \cap F_{36}$.
This means that $F_{14}\cap F_{25}\cap F_{36}$ is nonempty and so
$T_1'$ has an Eckardt point contained in $T_1 \cap T_2 \cap T_3$.
Similar arguments apply to $T_2'$ and $T_3'$ so (b) follows.

Clearly, (b) implies (c).  So it remains to show (c) implies (a).
Suppose $T_1', T_2'$ correspond to Eckardt points $p_1, p_2$.
Then $p_1 = F_{14} \cap F_{25} \cap F_{36}$
and $p_2 = F_{26} \cap F_{34} \cap F_{15}$.
Now $T_1 \cap T_2$ contain both $p_1$ and $p_2$
and the same is true of $T_1 \cap T_3$.
Thus $T_1 \cap T_2 \cap T_3$ contains both points and thus contains a
line.
\end{proof}

We also have the following, which generalizes Proposition~9.1.26~of~\cite{CAG}
to arbitrary characteristic (with a different proof).

\begin{cor} \label{cor:triadLineEck}
Let $X$ be a smooth cubic surface.
If $X$ contains two Eckardt points lying on a line $\ell$ not contained
in $X$ then there is a unique third Eckardt point such that $\ell$ is a
trihedral line.
\end{cor}

\subsection{Orbit structures}
\label{sec:Orbits}

Let $X$ be a smooth cubic surface with an automorphism $g$ of order $3$.
There are three types of orbits for the induced action of $g$ on the 27
lines:
\begin{itemize}
\item an invariant line,
\item a tritangent trio, or
\item a skew triple of lines.
\end{itemize}

For a $g$-invariant pair of conjugate triads,
the action of $g$ can only permute the rows and columns of the matrix
\eqref{eq:conjugateTriads}.
Thus, the possible orbit structures on the set of $9$ lines in the pair of
conjugate triads is as follows:
\begin{itemize}
\item $9$ invariant lines,
\item $3$ tritangent trios, or
\item $3$ skew triples.
\end{itemize}

\begin{lem}\label{lem:pointwise}
No automorphism of order $3$ fixes an exceptional line pointwise.
\end{lem}

\begin{proof}
Let $\ell$ be an exceptional line that is pointwise fixed under the
action of an automorphism $g$ of order $3$.
At most two other exceptional lines can meet $\ell$ at the same point,
so any exceptional line intersecting $\ell$ non-trivially must also be
$g$-invariant.
There are exactly $5$ tritangent planes containing $\ell$;
selecting a line from each plane,
we obtain $5$ skew exceptional lines that are each $g$-invariant.
Blowing these down we obtain a del Pezzo surface $Y$ of degree $8$ with $5$
$g$-fixed points in general position.
If $Y \cong \bbP^1 \times \bbP^1$, then there is no non-trivial
$g$-action with $5$ fixed points in general position
(all fixed points must lie on one or two fibers of a projection to $\bbP^1$).
Similarly, if $Y$ is a blow-up of $\bbP^2$ then we can blow-down
equivariantly to $\bbP^2$ with $6$ $g$-fixed points in general position,
which is again impossible.
\end{proof}

We now determine the possible orbit structures for all $27$ lines.

\begin{thm} \label{thm:3matrices}
Let $g$ be an automorphism of a smooth cubic surface $X$ of order $3$.
If $p \ne 3$, then let $\zeta$ be a primitive cube root of unity.
\begin{enumerate}
\item
The following are equivalent:
\begin{enumerate}
\item $g$ is of type 3A,
\item $X$ is a minimal $g$-surface,
\item the orbit structure on the $27$ lines is $9$ tritangent trios,
\item up to projective equivalence, the action of $g$ on $\bbP^3$ is
given by
\[
\begin{pmatrix}
\zeta & 0 & 0 & 0\\
0 & 1 & 0 & 0\\
0 & 0 & 1 & 0\\
0 & 0 & 0 & 1
\end{pmatrix}
\text{ if $p \ne 3$, or }
\begin{pmatrix}
1 & 1 & 0 & 0\\
0 & 1 & 0 & 0\\
0 & 0 & 1 & 0\\
0 & 0 & 0 & 1
\end{pmatrix}
\text{ if $p = 3$. }
\]
\end{enumerate}
\item
The following are equivalent:
\begin{enumerate}
\item $g$ is of type 3C,
\item $X$ is a $g$-equivariant blowup of $6$ points on $\bbP^2$ where three
points form an orbit and the other three are fixed,
\item the orbit structure on the $27$ lines is $6$ skew triples and
$9$ invariant lines,
\item up to projective equivalence, the action of $g$ on $\bbP^3$ is
given by
\[
\begin{pmatrix}
\zeta & 0 & 0 & 0\\
0 & \zeta & 0 & 0\\
0 & 0 & 1 & 0\\
0 & 0 & 0 & 1
\end{pmatrix} \ .
\]
\end{enumerate}
This is the Fermat cubic when $p\ne 3$;
This case does not occur when $p= 3$.
\item The following are equivalent:
\begin{enumerate}
\item $g$ is of type 3D,
\item $X$ is a $g$-equivariant blowup of $6$ points on $\bbP^2$ forming two
orbits each of size $3$,
\item the orbit structure on the $27$ lines is $3$ tritangent trios and
$6$ skew triples,
\item up to projective equivalence, the action of $g$ on $\bbP^3$ is
given by
\[
\begin{pmatrix}
0 & 0 & 1 & 0\\
1 & 0 & 0 & 0\\
0 & 1 & 0 & 0\\
0 & 0 & 0 & 1
\end{pmatrix} \ .
\]
\end{enumerate}
\end{enumerate}
\end{thm}

\begin{proof}

Recall from \cite{CAG}, 9.1.1 that, after fixing a geometric basis in
$\Pic(X)$ defined by a blowing-down morphism $X\to \bbP^2$, one defines
a natural bijection between 36 double-sixers of skew lines on $X$ and
positive roots in the lattice $\sfE_6$.
A sublattice of $\sfE_6$ spanned by two positive roots $\alpha,\beta$ with
$\alpha\cdot \beta = 1$ is isomorphic to the root lattice $\sfA_2$.
The 3 pairs of roots in this sublattice correspond to 3 distinct
double-sixers whose pairwise intersections each consist of 6 lines:
the corresponding collection of 18 lines is called an \emph{azygetic triad}.
The remaining 9 lines form a \emph{pair of conjugate triads}.
Three pairwise orthogonal root lattices $\sfA_2 \perp \sfA_2 \perp \sfA_2$
give rise to a triple of azygetic triads called a \emph{Steiner
complex}.
Each azygetic triad corresponds to the $9$ lines in the corresponding
pair of conjugate triads; the $27$ lines are the union of all three sets
of $9$ lines.

There are exactly $40$ Steiner complexes.
Thus any automorphism of order $3$ leaves one Steiner complex invariant.
Thus, the automorphism preserves a sublattice 
$\sfA_2\perp \sfA_2\perp \sfA_2$ of $\sfE_6$.
From Carter's classification, an element of Type 3A (resp. 3C, resp, 3D)
is conjugate to the product of cyclic permutations in three copies of
$A_2$ (resp. one copy, resp. two copies). This is reflected in Carter's
notation for the conjugacy classes $3A = 3A_2$, $3C = A_2$, $3D = 2A_2$.
Thus the $27$ lines can be partitioned into 3 $g$-invariant
pairs of conjugate triads.

Suppose there exists an invariant set of three skew lines.  Then these three
lines can be equivariantly blown down.
The resulting surface is a del Pezzo surface is of degree $6$, which is
never minimal for a group of order $3$.
Thus we can blow down three more lines.
We now have six skew lines permuted by $g$, thus the conjugacy class can
be easily computed via $\frakS_6 \subset W(\sfE_6)$.
In the case where we have two orbits of order $3$ we have 3D;
for one orbit of order $3$ and three invariant lines we have 3C.
This leaves 3A where there are no skew lines;
so the cubic surface is $g$-minimal.
We have shown the equivalence of (a) and (b) in each of the three cases.

Note that skew lines correspond to points in general position on $\bbP^2$
which are invariant under an automorphism of order $3$.
In characteristic $3$, any three distinct fixed points of an
automorphism of order $3$ in $\bbP^2$ must lie on a line.
Thus no cubic surface realizes an automorphism of type $3C$ when $p=3$.

We now describe the orbit structures of each type of automorphism on the
27 lines.
Since they are distinct, the orbit structures determine the conjugacy
class and vice versa.

The class 3A is minimal, so does not contain any orbits of skew lines.
The only possibility is 9 orbits contained in tritangent planes.

It is impossible to have more than $9$ invariant lines, since
then there would be at least two disjoint invariant pairs of conjugate
triads consisting of $9$ invariant lines each.  These 18 lines would
contain a double-sixer of invariant lines, which would force a
trivial action of $g$.

If there exists an orbit that forms a tritangent plane then there can
be no invariant lines.  Indeed, there are exactly $8$ other lines
incident to each line in the tritangent plane.  These partition the
remaining 24 lines so, in order to preserve incidence, they cannot be
invariant under $g$.

The automorphism of type $3D$ leaves invariant two orbits of three
points each in the plane model.
Partition these 6 points into 3 pairs, each pair containing exactly one
point from each orbit.  To each pair, associate the line passing through
these two points.  The strict transforms of these lines form
a $g$-invariant tritangent trio in $X$.
The 6 points also give rise to two distinct skew triples of lines in
$X$.  These orbits cannot lie in an invariant pair of conjugate triads,
so the only possibility is that they belong to two different
pairs of conjugate triads.
Thus $3D$ has the description given in the statement of the theorem.

We conclude that only $3C$ has invariant lines.
It cannot have more than $9$, so it must have exactly $9$.
It also cannot have any invariant tritangent planes, since they are
incompatible with the existence of invariant lines.
Thus it must have the description given in the statement of the theorem.

It remains to consider the matrix descriptions of each of the conjugacy
classes.
One can enumerate all the Jordan canonical forms of $4 \times 4$
matrices of order $3$.
Up to projective equivalence, they all occur in the statement of the
theorem except for the matrix
\[
\begin{pmatrix}
1 & 1 & 0 & 0\\
0 & 1 & 0 & 0\\
0 & 0 & 1 & 1\\
0 & 0 & 0 & 1
\end{pmatrix}
\]
in characteristic $3$.
From Section~1.12~of~\cite{Wehlau}, we reduce the equation to the form
\[
ax_0(x_0^2-x_1^2)+bx_2(x_2^2-x_3^2)+(x_0x_3-x_1x_2)L_1(x_1,x_3)+C_3(x_1,x_3) = 0.
\]
Computing the partial derivatives, we find that the point $(1,0,-1,0)$ is a singular point.
Thus this surface does not occur.

Suppose $g$ fixes pointwise a plane in $\bbP^3$.
Thus every exceptional line contains a fixed point.
Thus the orbit structure does not contain any skew triples.
We must be in the case 3A.
Conversely, suppose we are in the case 3A.  There are 9 invariant
tritangent planes so there are at least 9 fixed points in the dual space
of $\bbP^3$.
No more than $3$ such points can be collinear,
since no $4$ tritangent planes can share a common non-exceptional common line.
Thus, there must be a pointwise-fixed plane in the dual space.
Thus we have a fixed plane in the original space.
A pointwise fixed plane corresponds to a pseudoreflection,
so we have established the equivalence of all items for the case 3A.

Now suppose we are in case 3C.
We may assume we are not in characteristic $3$ so every invariant
exceptional line contains exactly 2 fixed points.
There are $9$ invariant lines so there must be at least two
pointwise-fixed lines in $\bbP^3$.
Thus, we must have the given matrix representation.
Conversely, if we have the given matrix representation then $X$
must be the Fermat by the following standard argument.
The invariant cubics are all of the form
\[
C_1(x_0,x_1) + C_2(x_2,x_3) = 0
\]
for cubic forms $C_1, C_2$.
Since the surface must be smooth, $C_1(x_0,x_1)$ has three distinct
roots in $\bbP^1=\bbP(x_0:x_1)$ and so can put in the form
$x_0^3+x_1^3$.
The same is true for $C_2$ so we have the Fermat cubic.
One checks that $g$ has the expected action on the exceptional lines
by an explicit computation in coordinates on the Fermat.

The remaining matrix must be for 3D since we have excluded all the other
possibilities.
\end{proof}

\subsection{Automorphisms of type 3A}
\label{sec:3A}

We say a smooth cubic surface $X$ is \emph{cyclic} if $X$ has a
triple cover of $\bbP^2$ with cyclic Galois group.
There may be more than one triple cover, but 
this only occurs for the Fermat surface.

\begin{lem}
A smooth cubic surface is cyclic if and only if it admits an
automorphism $g$ of class 3A.
Moreover, the automorphism $g$ generates the group of deck
transformations of a cover.
\end{lem}

\begin{proof}
By Theorem~\ref{thm:3matrices},
an automorphism of class 3A acts as a
pseudo-reflection on the ambient $\bbP^3$.
This gives a description of $X$ as a cyclic surface.
Conversely, if $X$ is a cyclic cubic surface with $g$ as its deck transformation
then it must have a pointwise fixed curve.  Thus 3A is the only
possibility.
\end{proof}

When $p\ne 3$, cyclic surfaces have the same normal form as in $\bbC$
by the same argument as in, for example, page~486~of~\cite{CAG}.
Namely, we may write the cyclic surface in the form
$x_3^3 + C(x_0,x_1,x_2)=0$
and then put the homogeneous cubic form $C(x_0,x_1,x_2)$
in Hesse standard form.
We obtain the following:

\begin{prop} \label{prop:normalForms3Apnot3}
Suppose $p \ne 3$, $X$ is a cyclic cubic surface, and $C$ is the
ramification divisor of the deck transformation $g$.
Then $X$ may be written in the form
\[
x_0^3 + x_1^3 + x_2^3 + x_3^3 + c x_0x_1x_2
\]
where $c$ is a parameter and $g$ acts via
$x_3 \mapsto \zeta x_3$ where $\zeta$ is a primitive third root of unity.
Here $C$ is a smooth genus $1$ curve in the plane defined by $x_3=0$.
The $9$ Eckardt points lying on $C$ are
\[
(1:-\epsilon:0:0), (0:1:-\epsilon:0), (1:0:-\epsilon:0)
\]
where $\epsilon$ is $1$, $\zeta$ or $\zeta^2$.
Taking one of these points to be the origin,
the points may be identified with the $3$-torsion subgroup
of the elliptic curve $C$, which is isomorphic to $3^2$.
\end{prop}

\begin{remark}
Suppose $p=2$ and $X$ is a cyclic surface other than the Fermat cubic
with $g$ the deck transformation.
In the above normal form, the canonical point is $(0:0:0:1)$
and the canonical plane is given by $x_3=0$.
Thus the canonical plane is the axis of the pseudoreflection $g$ and the
canonical point is its center.
\end{remark}

We defer the normal form for $p=3$ as we will use a description of some
of its automorphisms for the proof.

Recall that a \emph{tactical configuration} is a finite set of points and a
finite set of lines such that every point is contained in the same
number of lines and every line contains the same number of points.
Recall that the \emph{Hesse configuration} $(9_4\ 12_3)$ is the
configuration one obtains from the points and lines of the affine plane
over the field of three elements.
Equivalently, it can be obtained from the set of flexes of a smooth
curve of genus one over an algebraically closed field of characteristic
not $3$.
It contains $9$ points and $12$ lines where each point lies on $4$ lines
and each line contains $3$ points.

\begin{lem} \label{lem:3Aautos}
Let $g$ be an automorphism of type 3A.
Then $X$ contains $9$ coplanar $g$-fixed Eckardt points contained in $12$
trihedral lines, which form the Hesse configuration.
The corresponding reflections generate a group isomorphic to
$\calH_3(3) \rtimes 2$ with center $\langle g \rangle$.
\end{lem}

\begin{proof}
Note we do not assume anything about the characteristic!

Recall that $g$ fixes a plane $P$ pointwisely so every exceptional line must
have a $g$-fixed point on $P$.
If $g$ permutes the exceptional lines in a tritangent trio,
then this forces the constituent exceptional lines to share an Eckardt
point on $P$.
The orbit structure of the $27$ exceptional lines is a union of
$g$-invariant tritangent trios, so there is a set $\calE$ of $9$
Eckardt points on $P$ as desired.
Since the corresponding tritangent trios partition the exceptional
lines, there are no other Eckardt points on $P$.

Since there are no exceptional lines contained in $P$,
any two points in $\calE$ are contained in a trihedral line with a third.
Note that the set $\calE$ can be identified with a subset of the
tritangent planes and the trihedral lines correspond to triads; thus, the
configuration and its automorphism group are completely determined by
abstract considerations using $W(\sfE_6)$.
Thus, by establishing the lemma for any one cyclic cubic surface, we
prove it for all the others.
In particular, it is well known, and easy to check, for cyclic surfaces
over $\bbC$.
\end{proof}

In view of the previous description we may now establish the following:

\begin{lem} \label{lem:mush3A6A}
A smooth cubic surface admits an automorphism of class 3A if and only if
it admits an automorphism of class 6A.
\end{lem}

\begin{proof}
The square of an element of class 6A is of class 3A.  Conversely,
suppose $X$ admits an automorphism of class 3A.
The group from Lemma~\ref{lem:3Aautos} also admits an involution of class 2A
since there are $g$-fixed Eckardt points.  This involution must commute
with the center, which has class 3A.
Their product must be of class 6A by Table~\ref{tbl:cyclicWE6}.
\end{proof}

This last lemma allows us to find a normal form when $p=3$.

\begin{prop} \label{prop:normalForms3Ap3}
Suppose $p = 3$, $X$ is a cyclic cubic surface, and $C$ is the
ramification divisor of the deck transformation $g$.
Then $X$ may be written in the form
\begin{equation} \label{eq:3AnormalFormP3}
x_0^3 + x_0 x_3^2 - x_1x_2^2 + x_1^2x_3 + c x_1 x_3^2
\end{equation}
where $c$ is a parameter and $g$ acts via
$x_0 \mapsto x_0 + ix_3$ where $i$ is a primitive 4th root of unity.
Here $C$ is a cuspidal cubic curve with cusp at $(0:1:0:0)$.
The smooth locus of $C$ has a parametrization
\[
\phi : \bbG_a \to C \text{ via } \phi(t) = (t:t^3:1:0) \ .
\]
The solutions $\calE$ of $\alpha^9 - c \alpha^3 - \alpha = 0$
are a subgroup of $\bbG_a$ isomorphic to $3^2$.
Their images $\{ \phi(\alpha)\ : \ \alpha \in \calE \}$ are Eckardt points of $X$.
\end{prop}

\begin{proof}
By Lemma~\ref{lem:mush3A6A},
there exists an automorphism of class 6A represented by a matrix
\[
\begin{pmatrix}
1 & 0 & 0 & 1 \\
0 & 1 & 0 & 0 \\
0 & 0 &-1 & 0 \\
0 & 0 & 0 & 1
\end{pmatrix}
\]
where any invariant smooth cubic is:
\[
x_0^3-x_0x_3^2 + c_0 x_1 x_2^2 + c_1 x_3 x_2^2
+ c_2 x_1^3 + c_3 x_1^2 x_3 + c_4 x_1 x_3^2 + c_5 x_3^3 \ .
\]
Note that the change of coordinates $x_0 \mapsto x_0 - \sqrt[3]{c_2}x_1$
ensures that $c_2=0$. 
The fixed point locus of the cyclic cover is given by $x_3=0$
so now the ramification locus is
\[
x_0^3 + c_0 x_1 x_2^2 \ .
\]
An automorphism of class 3A has no pointwise-fixed lines, so the
ramification locus cannot contain any lines.  Thus $c_0 \ne 0$ and
the ramification curve is a cuspidal cubic as desired.
Since $c_0 \ne 0$, we may apply $x_1 \mapsto x_1 - \frac{c_1}{c_0}x_3$
to ensure $c_1 = 0$.
Now applying $x_0 \mapsto x_0 - r x_3$ where $r^3-r=c_5$ ensures that
$c_5=0$.
Now $c_3$ must be non-zero or else $(0:1:0:0)$ is singular.
Thus, by rescaling $x_1, x_2, x_3$ appropriately, we have the desired
normal form.

The ramification locus is the intersection of $X$ with the plane
$x_3=0$ and is defined by
\[
x_0^3 - x_1x_2^2 = 0 \ .
\]
The location of the cusp and the validity of the parametrization $\phi$
are verified by short computations.
The subset $\calE$ forms a subgroup of $\bbG_a$ since $x \mapsto x^3$
and $x \mapsto x^9$ are additive when $p=3$.
It remains to show that their images in $C$ are Eckardt points.

For each $q=\phi(\alpha)$, one verifies that the tangent space is
defined by
\begin{equation} \label{eq:3ATangent}
x_1 - \alpha^3 x_2 - \alpha^6 x_3 \ .
\end{equation}
Using this, we eliminate $x_1$ from the expression for $X$ to obtain
the following
\[
x_0^3 + x_0x_3^2 - \alpha^3 x_2^3 +(c \alpha^3-\alpha^9)x_2x_3^2
+(\alpha^{12}+c \alpha^6)x_3^3
\ ,
\]
which, along with \eqref{eq:3ATangent}, defines $T \cap X$.
One can rearrange this expression as
\[
(x_0-\alpha x_2)^3 + (x_0-\alpha x_2)x_3^2
+ (\alpha^4-c\alpha^6)x_3^3
\]
where we use the relation $\alpha^9 = c \alpha^3 + \alpha$.
Thus the intersection of the tangent space with $X$ is a union of three
exceptional lines each defined by
\[
x_0 - \alpha x_2 - \beta x_3 = 0
\]
as $\beta$ varies over the three solutions to
$\beta^3 + \beta + (c \alpha^2-1)\alpha^4 = 0$.
\end{proof}

\begin{remark} \label{rem:defOverF9}
By a rescaling of the variables, one can instead have a negative sign
for the monomial $x_0x_3^2$ in \eqref{eq:3AnormalFormP3}.
This has the pleasing consequence that the deck transformation is given
by $x_0 \mapsto x_0+x_3$ instead of requiring the fourth root of unity $i$.
In the current form, however, when $c=0$ the Eckardt points
are in direct bijection to the field $\bbF_9$ of $9$ elements
(note that $i \in \bbF_9$).
Up to projective equivalence, there are only two smooth cubic surfaces
defined over $\bbF_9$: this one and the Clebsch
(see Theorem~20.3.8~of~\cite{Hirschfeld}).
In Section~\ref{sec:8A}, we will see that this surface has an
automorphism group of order 216 --- our normal form makes it transparent
that all the automorphisms are also defined over $\bbF_9$.
\end{remark}

Since we know the Eckardt points, we may now explicitly describe
automorphisms for a cubic surface in form \eqref{eq:3AnormalFormP3}.
The following can be checked by
a tedious but straightforward calculation.

There is a group of automorphisms isomorphic to $\calH_3(3)$ given by the
matrices:
\[
\begin{pmatrix}
1 & 0 & \alpha & \beta \\
0 & 1 & \alpha^3 & -\alpha^6 \\
0 & 0 & 1 & \alpha^3 \\
0 & 0 & 0 & 1
\end{pmatrix}
\]
where $\alpha, \beta$ are the solutions to
\begin{align*}
\alpha^9-c \alpha^3 - \alpha &= 0 \\
\beta^3 + \beta + (c \alpha^2-1)\alpha^4 &= 0 \ .
\end{align*}

Note that $\alpha=0$ corresponds to the deck transformation subgroup.
The reflections given by the Eckardt points on $C$ are given by
$(x_0:x_1:x_2:x_3) \mapsto (x_0:x_1:-x_2:x_3)$ and its conjugates
under the group described above.

We record the following easy lemma since its consequences are
important later.

\begin{lem} \label{lem:invPoints}
If $h$ is an automorphism of a cyclic cubic surface which normalizes the
Galois group of the cover, then $h$ leaves invariant the ramification curve
and the $9$ Eckardt points within.
\end{lem}

\begin{proof}
The automorphism $h$ leaves invariant the ramification curve.
If $h$ normalizes $g$ then it must preserve the partition of the $27$
lines into $9$ $g$-invariant tritangent planes;
thus it must leave invariant the set of Eckardt points.
\end{proof}

As a converse, we will see that whenever there is an
automorphism $\tilde{h}$ of $(C, \calE)$, there is an automorphism $h$ of the
corresponding cyclic cubic surface inducing $\tilde{h}$.

When $p \ne 3$, recall that the automorphisms of a genus 1 curve are
generated by translations and group automorphisms (given a choice of origin).
The group of translations by $3$-torsion points is the image of the
group $\calH_3(3)$ above.
A general curve only admits group automorphisms of order $2$ --- these
are the reflections corresponding to Eckardt points above.
When furthermore $p \ne 2$,
the anharmonic curve admits a group automorphism of order $4$
discussed in Section~\ref{sec:12A};
an equianharmonic curve admits a group automorphism of order $3$
giving rise to the Fermat cubic surface discussed in the next
section.
When $p = 2$, the supersingular elliptic curve carries automorphisms of
order both $3$ and $4$ --- this can be seen as another explanation for
the fact that a surfaces admitting an automorphism of class 12A
is always the Fermat surface if $p=2$.

When $p = 3$, we are interested in automorphisms of the affine line that
preserve the subset defined by $\alpha^9 - c \alpha^3 - \alpha=0$.
Translation by elements in the subset correspond to the image of
$\calH_3(3)$ as before.  Up to translation, the remaining automorphisms are
of the form $\alpha \mapsto \lambda \alpha$ for $\lambda \in k^\times$.
Substituting into the defining equation for $\alpha$
we find that, when $c \ne 0$, we must have $\lambda^6=\lambda^8=1$.
Thus $\lambda^2=1$ and we only have automorphisms given by reflections
corresponding to Eckardt points.
However, when $c = 0$, we only require $\lambda^8=1$.
Thus we have an automorphism of order $8$ --- the
corresponding surface is discussed in Section~\ref{sec:8A}.

\begin{remark}
When $p \ne 2$, a smooth cubic surface is cyclic if and only if its
Hessian surface contains a hyperplane.
The Hessian surface of the Fermat cubic is simply the union of four
planes since it has $4$ cyclic structures.

When $p=3$, the Hessian surface of a cyclic cubic surface is a doubled plane
union a cone over a conic in $\bbP^2$.
In the coordinates of the normal form \eqref{eq:3AnormalFormP3},
the Hessian is $(x_1x_3+x_2^2)x_3^2$.
The doubled plane is defined by $x_3=0$,
the intersection of the components is defined by $x_3=x_2=0$,
and the vertex of the cone by $x_3=x_2=x_1=0$.
Thus there is a canonical flag associated to any cyclic cubic
surface, which is compatible with our choice of normal form.
\end{remark}

\subsection{Automorphisms of type 3C}
\label{sec:3C}

This case is the Fermat cubic, which has already been extensively discussed
in Section~\ref{sec:Fermat}.
We only prove the following, for later use:

\begin{lem} \label{lem:mush3C6C6F9A}
A cubic surface is the Fermat cubic surface if and only if it
admits an automorphism of class 3C, 6C, 6F, or 9A.
The Fermat cubic surface does not exist when $p=3$,
so in this case none of these automorphisms occur.
\end{lem}

\begin{proof}
We've seen above that admitting an automorphism of class 3C is
equivalent to being the Fermat cubic surface.
By considering the model over $\bbC$, we see that there are
automorphisms of class 6C, 6F, and 9A.
The class 6C and 6F have squares of class 3C so these certainly
correspond to the Fermat cubic.

It remains to show an automorphism $g$ of class 9A only acts on the Fermat.
Indeed, since the cube of $g$ is of class 3A, it must be a cyclic
surface.
In view of Lemma~\ref{lem:invPoints}, there is an induced action of
order $3$ on the ramification curve and the subgroup $G$ generated by
the reflections corresponding to the $9$ coplanar Eckardt points.
One can use the complex Fermat surface as a model to determine the
abstract action of $g$ on these $9$ points:
the automorphism $g$ acts via
\[ (x_0:x_1:x_2:x_3) \mapsto (x_1:x_2:\zeta_3 x_0:x_3) \]
on the Fermat, so $g$ does not act on $G$ as a translation.

When $p = 3$, since $g$ does not act as translation, and
the remaining possibilities have order dividing $8$, the surface
does not exist.
If $p \ne 3$, then $C$ is an elliptic curve and we have an automorphism
of order $3$ which is not translation; thus $C$ is equianharmonic.
Since such a curve is unique,
we conclude that $X$ is the Fermat cubic.
\end{proof}

\subsection{Automorphisms of type 3D}
\label{sec:3D}

\begin{lem} \label{lem:3Dautos}
Let $g$ be an automorphism of type $3D$.
Then $X$ contains a $g$-invariant trihedral line $\ell$
such that the reflections corresponding to the Eckardt points on $\ell$
generate a group isomorphic to $\frakS_3$ that contains $g$.
\end{lem}

\begin{proof}
There exist six skew lines $E_1, \ldots, E_6$ in $X$ such that $g$ acts
on them as the permutation $(123)(456)$.
The conjugate pair of triads given by \eqref{eq:conjugateTriads}
is invariant under this action.

Up to a choice of coordinates, we may blow down $E_1, \ldots, E_6$ to the points
$(1:0:0), (0:1:0), (0:0:1), (a:b:c), (c:a:b), (b:c:a)$.
Here $g$ acts by permuting the standard coordinate vectors.
There is another conjugacy class of automorphisms of order $3$ acting on
$\bbP^2$, but each of its orbits are contained in a line, so the $6$
points would not be of general position.

Let $T_1,T_2,T_3$ be the tritangent trios corresponding to the columns
of \eqref{eq:conjugateTriads}, and $T'_1,T'_2,T'_3$ be the conjugate
triad corresponding to the rows.
Recall that the lines $F_{ij}$ correspond to lines between the points in
the plane model corresponding to $E_i$ and $E_j$.
A calculation shows that $F_{14}, F_{26}, F_{35}$ all contain the point
corresponding to $(bc:ab:ac)$ in the plane model.
Thus, $T_1$ corresponds to an Eckardt point.
Similarly, there are two other Eckardt points corresponding to the other
two columns of \eqref{eq:conjugateTriads}.
They lie on a trihedral line $\ell$ since the corresponding tritangent
planes $T_1,T_2,T_3$ form a trihedron.

Let $r_i$ be the reflection corresponding to the Eckardt point
corresponding to $T_i$.
Recall that the action of $r_i$ leaves invariant only the exceptional
lines in the corresponding tritangent plane $T_i$,
while all other exceptional lines are taken to the other exceptional line
in the unique tritangent plane containing a line from $T_i$.
Note that rows of \eqref{eq:conjugateTriads} are tritangent trios
precisely of this form.
Thus $r_i$ acts by interchanging the other two columns of
\eqref{eq:conjugateTriads}.

Note that any particular line $E_i$ is incident to exactly $3$ lines in
\eqref{eq:conjugateTriads} --- exactly one in each row and column.
Thus, we may determine the action of $r_i$ on all 27 lines from its
action on the $9$ lines of the trihedral pair.
We conclude that $\langle r_1, r_2, r_3 \rangle \simeq \frakS_3$.
Setting $h=r_2r_1$, we find $h(T_i)=T_{i+1}$ and $h(T'_i)=T'_i$,
where $i$ is interpreted mod $3$.
From this we conclude that $h(E_i)=E_{\sigma(i)}$ where
$\sigma=(123)(456)$.  Thus $g=h \in \langle r_1, r_2, r_3 \rangle$.
\end{proof}

\begin{lem} \label{lem:normalForms3D}
Let $X$ be a cubic surface admitting an automorphism $g$ of type 3D.
If $p \ne 3$ and $X$ does not also admit an automorphism of type 6E,
then it can be defined by the equation
\begin{equation}\label{3d1}
x_0^3+x_1^3+x_2^3+x_3^3 +c_0x_0x_1x_2 +
c_1(x_0+x_1+x_2)x_3^2 = 0\ .
\end{equation}
If $p=3$, then $X$ can be defined by the equation
\begin{equation}\label{3d2}
(c_0(x_0+x_1+x_2)+c_1x_3)(x_0x_1+x_0x_2+x_1x_2)
+x_0x_1x_2+(x_0+x_1+x_2)x_3^2= 0
\ .
\end{equation}
In both cases $g(x_0:x_1:x_2:x_3)=(x_1:x_2:x_0:x_3)$.
\end{lem}

\begin{proof}
In the case $p\ne 3$, we may choose a normal form as in \cite{CAG}
equation (9.66).\footnote{Note that the coefficient at $t_0t_2t_3$ in the equation cannot be made equal to $1$.}
We would prefer a normal form which highlights the $\frakS_3$-symmetry.
Suppose $g$ acts via $g(t_0:t_1:t_2:t_3) = (t_0:\zeta t_1:\zeta^2 t_2:t_3)$.
By enumerating all homogeneous monomials of degree $3$ invariant under
this action we obtain the following normal form
\[
C(t_0,t_3) + L(t_0,t_3)t_1t_2 + t_1^3 + t_2^3 = 0
\]
for a smooth cubic surface where $C$ and $L$ are homogeneous forms.
One checks that $C$ must have distinct solutions in $\bbP^1$ or else the
surface is singular.
If $L=0$ then $X$ is the Fermat surface and has automorphisms of type 6E;
thus we may assume $L \ne 0$.
If $C$ and $L$ share a common zero then there exists an
automorphism of $\bbP^1$ which leaves the common zero fixed and
interchanges the remaining solutions of $C$;
this extends to an automorphism of $X$ of type 2A which commutes with $g$.
Thus, we may assume $C$ and $L$ have no common solutions or else
there exists an automorphism of type 6E.

Without loss of generality, we may assume $L=t_0$ and $C(t_0,t_3)$
has the monomials $t_0^3$ and $t_3^3$ in its support.
We now show that we may in addition assume $t_0^2t_3$ has coefficient
zero in $C$.
Let $p(t)=C(t,1)$ be the dehomogenization of $C$ and suppose
$r_1,r_2,r_3$ are its roots (non-zero by our assumption on the common
zeroes of $L$ and $C$).
We want to produce a fractional linear transformation $f(t)$
such that $p(f(t))$ is still cubic, has no $t^2$ term, and such that $f(0)=0$.
Consider the fractional linear transformation
\[
f : t \mapsto \frac{t}{at+1}
\]
for a parameter $a$.
We want to solve for the parameter $a$ such that
\begin{align} \label{eq:killSquareTerm}
\begin{split}
&0 = f(r_1)+f(r_2)+f(r_3)\\ =
&\frac{(r_1+r_2+r_3)+2a(r_1r_2+r_1r_3+r_2r_3)+3a^2r_1r_2r_3}{
(ar_1+1)(ar_2+1)(ar_3+1)} \ .
\end{split}
\end{align}
If $p \ne 2$, we may assume $r_1=1,r_2=-1$ and this becomes
\[
r_3-2a-3a^2r_3=0\ ,
\]
which has a solution such that the denominator of
\eqref{eq:killSquareTerm} is then non-zero whenever $r_3 \ne 1,-1,0$.
If $p = 2$, the equation \eqref{eq:killSquareTerm} becomes
\[
(r_1+r_2+r_3)+a^2r_1r_2r_3 = 0
\]
which has a solution since $r_1,r_2,r_3$ are non-zero;
the denominators of \eqref{eq:killSquareTerm} are non-zero whenever
$r_1,r_2,r_3$ are distinct and non-zero.

Thus, our normal form is
\[
t_0^3 + t_1^3 + t_2^3 + c_0t_0t_1t_2 + c_1t_0t_3^2 + t_3^3 = 0
\]
for parameters $c_0,c_1$.
We now make the change of coordinates
\begin{align*}
t_0 & \mapsto x_0+x_1+x_2\\
t_1 & \mapsto x_0+\zeta^2x_1+\zeta x_2\\
t_2 & \mapsto x_0+\zeta x_1+\zeta^2x_2\\
t_3 & \mapsto x_3
\end{align*}
to arrive at the equation
\[
(3+c_0)(x_0^3+x_1^3+x_2^3) +(18-3c_0)x_0x_1x_2 +
c_1(x_0+x_1+x_2)x_3^2+x_3^3 \ .
\]
The surface is singular if $c_0=-3$, so we may rescale the variables
to obtain the desired normal form.


Now assume $p = 3$. By Lemma~\ref{lem:3Dautos}, we may assume that the
equation is invariant with respect to permutation of the first three
coordinates.  We can write $F\in \Bbbk[x_3][x_0,x_1,x_2]^{\frakS_3}$ in
the form
\begin{equation} \label{eq:3Dp3startForm}
F= a\sigma_1^3+b\sigma_1\sigma_2+c\sigma_3+d\sigma_1^2x_3+e\sigma_2x_3
+f\sigma_1x_3^2+gx_3^3,
\end{equation}
where $\sigma_1, \sigma_2, \sigma_3$ are the elementary symmetric
polynomials in $x_0,x_1,x_2$.
The point $(1:1:1:0)$ is singular unless $c \ne 0$, so we may assume
without loss of generality that $c=1$.

We use a linear transformation $\tau$, which commutes with $\frakS_3$,
defined by the matrix
\begin{equation} \label{eq:3Dp3matrix}
M = 
\begin{pmatrix}
1+a_2&a_2&a_2&a_4\\
a_2&1+a_2&a_2&a_4\\
a_2&a_2&1+a_2&a_4\\
a_5&a_5&a_5&1
\end{pmatrix} \ ,
\end{equation}
which has determinant $1$.
Note that neither $\tau(\sigma_1)$ nor $\tau(\sigma_2)$ contain pure
powers of $x_3^2$, so the coefficient of $x_3^3$ in $F$ is $a_4^3+g$.
Taking $a_4=\sqrt[3]{-g}$, we may thus assume that $g=0$.
Now, $f \ne 0$ or else $X$ is singular, so we may assume $f=1$.

Setting $a_4=0$, the transformation $\tau$ induces the following action:
\begin{align*}
\sigma_1 &\mapsto \sigma_1\\
\sigma_2 &\mapsto \sigma_2 - a_2\sigma_1^2 \\
\sigma_3 &\mapsto \sigma_3 + a_2^2(a_2+1)\sigma_1^3 + a_2\sigma_1\sigma_2\\
x_3 &\mapsto a_5\sigma_1 + x_3 \ .
\end{align*}

Assuming $a_4=0$ we obtain a coefficient of
$-ea_2-a_5+d$ for $\sigma_1^2x_3$ in $\tau(F)$.
Thus we may assume $d=0$ in $F$.
Applying $\tau$ again with $a_4=0$ and $a_5=-ea_2$, the coefficient of
$\sigma_1^2x_3$ remains $0$ and the coefficient of $\sigma_1^3$ becomes
\[
a_2^3 + (1-e^2)a_2^2 - ba_2 + a,
\]
and so one can choose $a_2$ to eliminate the coefficient of $\sigma_1^3$
also.
We are left with the expression
\[
F = b\sigma_1\sigma_2 + \sigma_3 + e\sigma_2 x_3 + \sigma_1
x_3^2\ ,
\]
which is the desired normal form.
\end{proof}

\section{Automorphisms of order 4}

Let $X$ be a smooth cubic surface and $g$ be an automorphism of order $4$.

\begin{lem}
One of the following holds:
\begin{enumerate}
\item
Up to choice of coordinates,
$g$ has a lift $\widetilde{g} \in \GL(V)$ of the form
\begin{equation} \label{eq:4Amatrix}
\widetilde{g} =
\left(\begin{smallmatrix}
i&0&0&0\\
0&-1&0&0\\
0&0&1&0\\
0&0&0&1
\end{smallmatrix}\right)
\text{ if $p \ne 2$, or }
\left(\begin{smallmatrix}
1&1&0&0\\
0&1&1&0\\
0&0&1&0\\
0&0&0&1
\end{smallmatrix}\right)
\text{if $p=2$.}
\end{equation}
Here $i$ denotes a primitive $4$th root of unity.
The element $g$ corresponds to an element of class 4A in $W(\sfE_6)$.
\item
Up to choice of coordinates,
$g$ has a lift $\widetilde{g} \in \GL(V)$ of the form
\begin{equation} \label{eq:4Bmatrix}
\widetilde{g} =
\left(\begin{smallmatrix}
0&0&0&1\\
1&0&0&0\\
0&1&0&0\\
0&0&1&0
\end{smallmatrix}\right)\ .
\end{equation}
The element $g$ corresponds to an element of class 4B in $W(\sfE_6)$.
\end{enumerate}

\end{lem}

\begin{proof}
From Lemma~\ref{lem:excludedClasses}, only classes 4A and 4B may occur,
so we only need to determine the corresponding classes of matrices
of order $4$ which leave invariant a smooth cubic.
If $p \ne 2$ then one can diagonalize the matrix.
The analysis of the possibilities is then the same as the case over $\bbC$
given in \S{}9.5.1~of~\cite{CAG}.
The class 4A can be distinguished from 4B by the eigenvalues of
$\tilde{g}^2$.
If $p=2$, then there are only two matrices of order $4$ up to conjugacy:
$J_4(1)$ and $J_3(1) \oplus J_1(1)$ where $J_k(\lambda)$ is a
$k \times k$ Jordan block with eigenvalue $\lambda$.
A matrix $J_4(1)$ has square $J_2(1) \oplus J_2(1)$ and thus corresponds
to 4B.
A matrix $J_3(1) \oplus J_1(1)$ has square
$J_2(1) \oplus J_1(1)^{\oplus 2}$ and thus corresponds to 4A.
\end{proof}

\subsection{Class 4A}
\label{sec:4A}

\begin{lem} \label{lem:4Adesc}
Let $X$ be a smooth cubic surface with an automorphism $g$ of class 4A.
There are exactly three invariant exceptional lines
$\ell_0, \ell_1, \ell_2$, which form a tritangent plane.
The remaining $24$ lines are partitioned into $6$ orbits of $4$ lines each.
Each orbit consists of two pairs of incident lines, each of which form
tritangent planes with one of $\ell_0, \ell_1, \ell_2$.
\end{lem}

\begin{proof}
From Table~\ref{tbl:D5ccs}, we see that $X$ is obtained by blowing up a
point on a del Pezzo surface $Y$ of degree $4$ where $g$ acts on $Y$
with class $\bar{2}\bar{2}1$.
Let $E_1, \ldots, E_5$ correspond to a skew set of exceptional lines on
$Y$ and let $E_6$ be the line corresponding to the blown up point.
We may assume $g$ acts as $\iota_{13}(12)(34)$ on $Y$.
Note that $g^2$ is the involution $\iota_{1234}$
and we know its orbit structure from \eqref{eq:27orbits2A}.
We check that the each line in the tritangent trio $E_6, F_{56}, G_5$ is
$g$-invariant.
One checks that $\iota_{13}(12)(34)$ acts freely on the $16$ lines on $Y$
using, for example, the $\bbF_2^5$ model discussed in
Section~\ref{sec:DP4}.
One checks that $g(G_1)=G_2$ and $g(G_3)=G_4$ by exploiting incidence
relations with the lines on $Y$.  These are in different $g^2$ orbits,
so $g$ acts transitively on the $24$ non-invariant lines.
Since $g^2$-orbits are pairs of incident lines as stated in the theorem,
we have proved the result.
\end{proof}

\begin{lem}
Suppose $p=2$, $X$ is a smooth cubic surface, and $g$ is an automorphism
of class 4A.
An invariant exceptional line on $X$ is blown down to the canonical
point of a del Pezzo surface $Y$ of degree $4$.
\end{lem}

\begin{proof}
From Table~\ref{tbl:D5ccs}, the involution $g^2$ acts as an involution
of first kind on $Y$.
Blowing down to a quartic del Pezzo surface $Y$, we have a point $p$ on the
fixed locus $F$ of $g^2$, which is an involution of the first kind.
The existence of $g$, an element of class $\bar{2}\bar{2}1$, implies
that $Y$ is also invariant under an involution of class $221$.
Thus, by Proposition~\ref{prop:inv12kind},
the locus $F$ is a pair of rational curves tangent at the canonical point.
Note that $F$ is the intersection of $Y$ with a hyperplane.
This is still true after projecting away from $p$ in $\bbP^4$ to
recover $X$ in $\bbP^3$.
The preimage of $F$ under the blowup is thus a tritangent plane.
Since the original curves were tangent, it has an Eckardt point.
Since the three exceptional curves are all $g$-invariant,
the original point $p$ must have been a common point of the
components of $F$ --- this is the canonical point.
\end{proof}

We will see in Lemma~\ref{lem:mash} below that
in characteristic $2$, the surfaces admitting an automorphism of class
4A coincide with those admitting automorphisms of class 4B and 6E.
So, in the next Lemma we state the result only in the case $p\ne 2$
deferring the case $p = 2$ until the next section.

\begin{lem} \label{lem:normalForms4A}
Suppose $p \ne 2$ and $X$ is a smooth cubic surface with an automorphism
$g$ of class 4A.
Then we have the normal form
\begin{eqnarray}\label{eq:4Anormal0}
x_3^2 x_2 + x_2^2x_0 + x_1 (x_1-x_0) (x_1 - c x_0) = 0
\end{eqnarray}
where $g:(x_0:x_1:x_2:x_3)\mapsto (x_0:x_1:-x_2:ix_3)$.
The stratum of such surfaces is $1$-dimensional
and is isomorphic to the moduli space of elliptic curves.
\end{lem}

\begin{proof}
If $p\ne 2,3$, a normal form for can be also be found in (9.67) of
\cite{CAG}.
A general $g$-invariant cubic has the form
\[
a_1x_2x_3^2 + (a_2x_0+a_3x_1)x_2^2 + C(x_0,x_1)
\]
for parameters $a_1,a_2,a_3$ and a homogeneous binary form $C$ of degree $3$.
After a linear change of coordinates and scaling variables we may assume
the cubic is of the form
\[
x_2x_3^2 + x_0x_2^2 + C(x_0,x_1)
\]
where, since the surfaces is smooth, $C$ must have distinct roots and
$x_0$ cannot divide $C$.
Setting $x_3=0$, we simply have an elliptic curve, which we may put into
Legendre form.
This produces the desired normal form for the cubic surface.
\end{proof}

\subsection{Class 4B}
\label{sec:4B}

\begin{lem} \label{lem:4Bautos}
If $X$ admits an automorphism $g$ of class $4B$, then there is a canonical
$g$-invariant tritangent plane $T$.

In characteristic $\ne 2$,
there are six Eckardt points on $T$, with two on each exceptional line
of $T$, whose corresponding reflections generate a group $G$ isomorphic to
$\frakS_4$ containing $g$.
The group $G$ has a normal subgroup $N$ isomorphic to $2^2$ generated by
elements of the form $2B$ which leave the exceptional lines in $T$
invariant.
The quotient group $G/N \cong \frakS_3$ acts by permuting the lines of $T$.

In characteristic $2$, the three exceptional lines on $T$ meet at a
canonical Eckardt point.
There are $4$ other Eckardt points on each line of $T$ giving $13$ in total.
Their corresponding reflections generate a group of order $192$
isomorphic to $2^3 \rtimes \frakS_4$, which contains $g$.
\end{lem}

\begin{proof}
Recall that $g^2$ is an involution of type 2B.
Thus there is a canonical exceptional line $\ell$ containing two Eckardt
points $p_1, p_2$ whose corresponding reflections $r_1, r_2$ have
product $r_1r_2=g^2$.
The action on the 27 lines is the same regardless of the cubic surface
so it suffices to consider the situation of the Fermat cubic:
\[
x_1^3 + x_2^3 + x_3^3 + x_4^3 = 0 \ .
\]
We may assume $g$ is the permutation $(1234)$.  Thus $g^2$ is $(13)(24)$
and the reflections are $(13)$ and $(24)$.
The line $\ell$ is defined by $x_1+x_3=x_2+x_4 = 0$
in the Fermat model.
The reflections generate the group $\frakS_4$.
The reflections are the $6$ transpositions, so we
obtain $6$ Eckardt points.
The other two involutions $(12)(34)$ and $(14)(23)$
have canonical exceptional lines $\ell'$ and $\ell''$.
The lines $\ell, \ell', \ell''$ form a tritangent plane $T$,
which is canonical by construction.
We see that $N = \langle (12)(34), (14)(23) \rangle \simeq 2^2$
acts trivially on the set of three lines.

If the characteristic is $2$, then $\ell$ must contain $5$ Eckardt
points.  By symmetry, so must $\ell'$ and $\ell''$.
Since all possible Eckardt points are realized, one of them must be the
intersection $\ell \cap \ell' \cap \ell''$.
Blowing down $\ell$, we obtain a quartic del Pezzo surface of degree $4$
whose automorphism group contains $2^4 \rtimes 2^2$ of order $64$.

The reflection corresponding to the canonical point commutes with all
other reflections so it generates a central subgroup of $G$ of order
$2$.  Let $H$ be the quotient of $G$ by this group of order $2$.
Let $r_1, \ldots, r_4$ denote the images in $H$ of the reflections
corresponding to the remaining Eckardt points on $\ell$.
Label those from $\ell'$ by
$s_1,\ldots,s_4$; and from $\ell''$, by $t_1,\ldots,t_4$.
Note $r_1r_2r_3r_4=1$ and $r_ir_j=r_jr_i$ for all $i,j$;
the same is true for the $s_i$ and $t_i$.

Since all the points are coplanar, a line passing through Eckardt
points on $\ell$ and $\ell'$ passes through a unique third point on
$\ell''$.  This line is a trihedral line and the corresponding
reflections generate a group isomorphic to $\frakS_3$.
If the corresponding reflections are $r_i, s_j, t_k$,
then $r_is_j=t_kr_i=s_jt_k$.

We may keep track of the 16 trihedral lines via four permutations
$\rho_1,\ldots, \rho_4$ of $\{1,2,3,4\}$ defined such that
$\{r_i,s_j,t_{\rho_i(j)}\}$ describe all possible trihedral lines as
$i,j$ vary over $\{1,2,3,4\}$.
We may label the reflections such that $\rho_1$ is the identity.
Since $r_1,\ldots,r_4$ commute, the computation
\[
s_{\rho_i(j)}r_1r_i=
r_1t_{\rho_i(j)}r_i=
r_1r_is_j=
r_ir_1s_j=
r_it_jr_1=
s_{\rho_i^{-1}(j)}r_ir_1
\]
shows that the remaining $\rho_i$'s are involutions.
A similar computation with $r_ir_js_k=r_jr_is_k$ show that
the $\rho_i$'s commute with one another.

Consider the subgroup $K$ of $H$ generated by
$R=r_1r_2$ and $S=s_1s_{\rho_2(1)}$.
Note that $S=s_is_{\rho_2(i)}$ for all $i$.
One computes the product
\[
RS = r_1r_2s_1s_{\rho_2(1)}=r_2r_1s_{\rho_2(1)}s_1
= r_2t_{\rho_2(1)}r_1s_1 = t_{\rho_2(1)}s_1s_1t_1
= t_1t_{\rho_2(1)} =: T
\]
and checks that $K \cong 2^2$.
By checking $s_iRs_i=RS=T$ and similar conditions,
we see that $K$ is normal in $H$.
One could have also defined $R$ as $r_1r_3$ or $r_1r_4$,
to obtain different but isomorphic normal subgroups.
Thus $K$ is normal, but not characteristic in $H$.
(Note that $R=r_1r_2=r_3r_4$, though, so there are essentially only the
three such normal subgroups.)

In the group $H/K$, we have $r_1\equiv r_2$ and $r_3 \equiv r_4$;
similar equations are true for the $s_i$'s and $t_i$'s.
Thus the $6$ reflections generating $\frakS_4$ above generate $H/K$.
We obtain the semidirect product description described above for the
original group $G$.
\end{proof}

\begin{remark}
From Theorem~20.3.4~of~\cite{Hirschfeld},
up to projective equivalence, there is a unique smooth cubic surface
over the field $\bbF$ of $8$ elements;
its automorphism group is precisely the group of order $192$ described
in Lemma~\ref{lem:4Bautos}.
Of course, over the algebraic closure, this is simply the Fermat cubic
surface and the geometric automorphism group is larger.
\end{remark}

We will see in Lemma~\ref{lem:mash} below that
in characteristic $2$, the surfaces admitting an automorphism of class
4A coincide with those admitting automorphisms of class 4B and 6E.

\begin{prop} \label{prop:normalForms4B}
Let $X$ be a 4B surface.  Then $X$ is isomorphic to a surface given by equation
$$
x_0^3+x_1^3+x_2^3+x_3^3+c(x_0x_1x_2 + x_0x_1x_3 + x_0x_2x_3
 + x_1x_2x_3) = 0.$$
\end{prop}

\begin{proof}
We know that $\Aut(X)$ contains $\frakS_4$ that acts by permuting the variables.   The equation must be of the form
$$a_1\sigma_1^3+a_2\sigma_1\sigma_2+a_3\sigma_3 = 0,$$
where $\sigma_i$ are elementary symmetric polynomials in the coordinates.

A change of variables that commutes with the permutations of the coordinates  is given by a matrix
\[
M:= \begin{pmatrix}1+\beta&\beta&\beta&\beta\\
\beta&1+\beta&\beta&\beta\\
\beta&\beta&1+\beta&\beta\\
\beta&\beta&\beta&1+\beta\end{pmatrix}
\]
with determinant $D:= (1+4\beta)$.
It acts via
\begin{align*}
\sigma_1 &\mapsto (1+4\beta)\sigma_1\\
\sigma_2 &\mapsto \sigma_2 + (6\beta^2+3\beta)\sigma_1^2 \\
\sigma_3 &\mapsto \sigma_3 + (4\beta^3+3\beta^2)\sigma_1^3
+ 2\beta \sigma_1\sigma_2
\end{align*}
which transforms the equation to the form
\begin{equation} \label{eq:new4B}
a_1'\sigma_1^3+a_2'\sigma_1\sigma_2+a_3'\sigma_3 = 0,
\end{equation}
where
\begin{align*}
a_1' &= (1+4\beta)^3a_1 + (6\beta^2+3\beta)(1+4\beta) a_2 +
(4\beta^3+3\beta^2)a_3\\
a_2' &= (1+4\beta)a_2+2\beta a_3\\
a_3' &= a_3 \ .
\end{align*}

Note that the coefficient of any monomial $x_i^2x_j$ with $i \ne j$
is equal to $a_2'+3a_1'$.
Let $p_3:=x_1^3+x_2^3+x_3^3+x_4^3$, which, by Newton's identities,
is $\sigma_1^3-3\sigma_1\sigma_2+3\sigma_3$.
If $a_2'=-3a_1'$, then the transformed cubic \eqref{eq:new4B} can be written
as
\[
a_1'p_3 + (a_3'-3a_1')\sigma_3\ .
\]
After noting that $a_1' \ne 0$ or else $X$ is singular,
we see that we have precisely the desired normal form.

It remains to show that we can find a transformation so that
$a_2'+3a_1'=0$.
Note that $a_2'+3a_1'=f(\beta)$ where 
\begin{align} \label{eq:killCondition}
f(\beta) :=\ & 12(16a_1+6a_2+a_3)\beta^3
+ 9(16a_1+6a_2+a_3)\beta^2\\ \notag
& + (36a_1+13a_2+2a_3)\beta
+ (3a_1+a_2). \notag
\end{align}
We claim that there is a choice of $\beta$ such that $f(\beta)=0$
and $D=(1+4\beta)\ne 0$.
Note that $D \ne 0$ when $p=2$.
When $p \ne 2$, we compute that $f(-\frac{1}{4})=-\frac{1}{8}a_3$;
if $a_3 = 0$ then the cubic surface is reducible, so $D \ne 0$ whenever
$f(\beta)=0$.
When $p=2$, we see that $f(\beta)=a_3\beta^2 +a_2\beta+(a_1+a_2)$
and conclude that $f(\beta)$ always has a root since $a_3 \ne 0$.
If $p=3$, then $f(\beta)=(a_2-a_3)\beta+a_2$; this always has a root
since $a_2=a_3$ implies that $(1:-1:-1:1)$ is a singular point.
If $p \ne 2,3$, then $f(\beta)$ has no roots only if all the
coefficients of powers of $\beta$ are zero.
This occurs only when $a_3=8a_1$ and $a_2=-4a_1$.
In this case, the cubic is singular at $(1:1:1:1)$.

Since we can always find a root for $f(\beta)$ with $D\ne 0$,
we have the desired transformation.
\end{proof}

\begin{remark} \label{rem:strange4A}
From the normal form in $p=2$, we see immediately that a 4B surface $X$
is either the Fermat cubic, or has a canonical point that coincides with
the canonical Eckardt point in the canonical tritangent plane.
The tritangent plane here is the plane $x_0+x_1+x_2+x_3 = 0$ and the
canonical point is $(1:1:1:1)$.
\end{remark}

\section{Automorphisms of higher order}
\label{sec:higherOrder}

\subsection{Class 5A}
\label{sec:5A}

Recall that the \emph{Clebsch cubic surface} is a smooth cubic surface
whose automorphism group contains a subgroup isomorphic to $\frakS_5$.
When $p \ne 5$, a normal form for such a surface is given by
\begin{equation} \label{eq:charFreeCubic}
\sum_{i} x_i = \sum_{i < j < k} x_ix_jx_k = 0
\end{equation}
in $\bbP^4=\bbP(x_0: \ldots : x_4)$.
Due to the Newton identities, in characteristic $\ne 3,5$ the Clebsch
surface can be defined by the more familiar equations
\[
\sum_{i=0}^4 x_i = \sum_{i=0}^4 x_i^3 = 0 \ .
\]
There is an obvious action of $\frakS_5$ by permuting the coordinates.

In characteristic $5$, the variety defined by the equations
\eqref{eq:charFreeCubic} has a
singular point $(1:1:1:1:1)$.  Indeed, there is no Clebsch surface
when $p=5$:

\begin{lem} \label{lem:exclude5A}
When $p=5$, there is no smooth cubic surface admitting an automorphism
of order $5$.
\end{lem}

\begin{proof}
The class 5A is the unique conjugacy class order $5$ in $W(\sfE_6)$.
Since $X$ has $36$ double-sixers, one of them is invariant.
An automorphism of order 5 cannot switch the two sixers of skew lines.
Thus, blowing down a sixer, we realize $g$ as a projective automorphism
of $6$ points, which must fix one of the points.

There are two possible Jordan canonical forms for a $3 \times 3$ matrix
of order $5$:
\[
M_1 = 
\begin{pmatrix}
1 & 1 & 0 \\
0 & 1 & 0 \\
0 & 0 & 1
\end{pmatrix}
\text{ and }
M_2 = 
\begin{pmatrix}
1 & 1 & 0 \\
0 & 1 & 1 \\
0 & 0 & 1
\end{pmatrix} \ .
\]

Let $x_1, x_2, x_3$ generate the dual basis and view them as elements of
the homogeneous coordinate ring of $\bbP^2$.
From Proposition~7.6.1~of~\cite{Wehlau},
the possible invariant conics in each case are
\[
c_0 x_2^2 + c_1 x_2x_3 + c_2 x_3^3
\]
for the matrix $M_1$ and
\[
c_0 (x_1 x_3 + 2x_2^2 + 3 x_2x_3) + c_1 x_3^2
\]
for the matrix $M_2$ where $c_0,c_1,c_2$ are parameters.

In the first case, the conics are never smooth.
In the second case, the unique fixed point $(1:0:0)$
lies on any invariant conic.
Thus, it is impossible to find $5$ invariant points lying on a smooth
conic such that there is a fixed point not lying on the conic.
This is a contradiction.
\end{proof}

We now have the following:

\begin{thm}
If $X$ is a smooth cubic surface with an automorphism of order $5$ then
$X$ is isomorphic to the Clebsch cubic surface.
In characteristic $5$, there is no such surface.
\end{thm}

\begin{proof}
The class 5A is the unique conjugacy class of order $5$ in $W(\sfE_6)$.
It can be represented by the permutation $(12345)$ on six skew lines
$E_1, \ldots, E_6$.
Since $E_6$ is invariant, it can be blown down equivariantly to a del
Pezzo surface $Y$ of degree $4$.
The surface $Y$ is uniquely determined by a set $S$ of $5$ distinct points
in $\bbP^1$ up to projective equivalence.
The automorphism group of such a surface is of the form $2^4 \rtimes G$
where $G$ is the embedded automorphism group of $S$.
Up to projective equivalence, there is a unique set of $5$ points in
$\bbP^2$ invariant under an automorphism of order $5$.
Thus, there is a unique surface $Y$ and so there can be at most one
smooth cubic surface $X$ with the desired property.
When the characteristic is not $5$, $X$ must be the Clebsch surface.
\end{proof}

\begin{remark}
A group of order $5$ has two invariant lines and no invariant tritangent
planes. This implies that the invariant lines are skew.
Given two skew lines on a cubic surface, there is a set of five mutually
skew lines that intersect both given lines.
Blowing the latter down, the two invariant exceptional lines are no
longer exceptional, so they are not invariant lines and the surface is
not a blowup of a single point on $\bbP^2$.
Thus the cubic surface is equivariantly isomorphic to the blow-up of
five points on a nonsingular quadric $Q$.
We can start from an orbit of five points on a quadric and consider a
rational map $Q\dasharrow \bbP^3$ given by linear system of quadric
sections of $Q$ passing through the points in the orbit.
The image of this map is a cubic surface.
The previous theorem implies the unexpected fact
that, in characteristic 5, the image must be a singular cubic surface.
\end{remark} 

\begin{lem} \label{lem:mush3C5A}
In characteristic $2$, the Fermat cubic surface and the Clebsch cubic
surface are isomorphic.
\end{lem}

\begin{proof}
Note that the Fermat cubic surface has an automorphism of degree $5$, so
this follows immediately from the previous result.
This can also be seen directly using a careful choice of matrix
as in the proof of Proposition~\ref{prop:normalForms4B}.
\end{proof}

\begin{lem} \label{lem:5Aautos}
If $p \ne 5$ and $g$ is an automorphism of order $5$,
then there is a set of $10$
Eckardt points whose corresponding reflections generate a group
isomorphic to $\frakS_5$ containing $g$.
\end{lem}

\begin{proof}
This is a Clebsch cubic surface which can be put in the form
\eqref{eq:charFreeCubic} where $g$ acts by permuting the coordinates.
The reflections correspond to transpositions in $\frakS_5$.
There are 10 of them and they generate $\frakS_5$.
\end{proof}

\begin{prop} \label{prop:normalForms5A}
The Clebsch cubic surface has the normal form
\[
\sigma_1\sigma_2 - \sigma_3
= \left(\sum_{i \ne j} x_i^2x_j \right)
+ 2\left(\sum_{i < j < k} x_ix_jx_k\right)
\]
in $\bbP(x_0:x_1:x_2:x_3)$ in all characteristics $p \ne 5$.
\end{prop}

\begin{proof}
Make the substitution $x_4=-x_0-x_1-x_2-x_3$ in
\eqref{eq:charFreeCubic}.
\end{proof}

\begin{remark} \label{rem:ClebschOtherForm}
One could also put the Clebsch in the normal form from
Proposition~\ref{prop:normalForms4B} for an appropriate choice of $c$.
From the proof, a tedious calculation
shows that $c$ satisfies
$5c^3 + 18c^2 - 189c + 378=0$
when $p=0$.
\end{remark}

\subsection{Class 6E}
\label{sec:6E}

\begin{lem} \label{lem:6Eautos}
Let $X$ be a smooth cubic surface and suppose $g$ is an automorphism of
class 6E.
Then there exist four Eckardt points on $X$, of which three lie on a trihedral 
line contained in the tritangent plane of the fourth.
Their corresponding reflections generate a group $G$ isomorphic to
$\frakS_3 \times 2$ containing $g$.
\end{lem}

\begin{proof}
From Table~\ref{tbl:cyclicWE6},
an automorphism $g$ has class 6E if and only if it is a product $g=ab$
of commuting automorphisms $a$ of class $3D$ and
$b$ of class $2A$.
The automorphism $a$ has a canonical trihedral line $\ell$ which must be invariant
under $b$ since $a$ and $b$ commute.
If $b$ acted nontrivially on the Eckardt points on $\ell$
then it would have to leave one fixed since $b$ has order $2$;
this kind of action cannot commute with the cyclic permutation induced
by the action of $a$.  Thus $b$ fixes each Eckardt point of $\ell$.
In particular, $b$ commutes with their corresponding reflections.
Thus there are exceptional lines connecting each Eckardt point on
$\ell$ to the point corresponding to $b$.
Since $b$ commutes with $a$ and the reflections on its trihedral line,
they generate a group isomorphic to $\frakS_3 \times 2$ as desired.
\end{proof}

\begin{lem} \label{lem:mash}
When $p=2$, a smooth cubic surface admits an automorphism
whose class is one of 4A, 4B, 6E if and only if it admits automorphisms of all
three classes.
\end{lem}

\begin{proof}
We begin by considering the classes 4A and 4B.
In both cases, the blow-down of any invariant exceptional line is the
canonical point on a quartic del Pezzo surface.
The class 4A corresponds to $\bar{2}\bar{2}1$,
while the class 4B corresponds to $2\bar{2}\bar{1}$.
Any surface admitting one of these must admit the other.
Since they both preserve the canonical point,
blowing it up we obtain a cubic surface admitting automorphisms of both
classes.

If the surface admits an automorphism of class 4A or 4B, then it admits
a group of automorphisms of order $192$ as described in
Lemma~\ref{lem:4Bautos}.
There is an automorphism $a$ of class 3D --- for example, the product of any two
non-commuting reflections.
The center of the group is an involution $b$ of class
2A.
Since $b$ commutes with $a$, the product $g=ab$ is an element of order $6$.
From Table~\ref{tbl:cyclicWE6}, we see that $g$ must be of class 6E on
account of its powers $a=g^4$ and $b=g^3$.

Now suppose there is an automorphism of class 6E.
From Lemma~\ref{lem:6Eautos}, we have an Eckardt point where the
exceptional lines in its tritangent plane each contain another Eckardt
point.
Since $p=2$, there are $5$ Eckardt points on each of these exceptional
lines and we have precisely the situation of Lemma~\ref{lem:4Bautos}.
In particular, there are automorphisms of class 4A and 4B.
\end{proof}

By the previous Lemma, a normal form for 6E surfaces when $p=2$ was
exhibited in Proposition~\ref{prop:normalForms4B}.

\begin{prop} \label{prop:normalForms6E}
Let $X$ be a 6E surface.  If $p \ne 2,3$, then $X$
has the normal form
\[
x_0^3+x_1^3+x_2^3+(x_0+x_1+x_2)x_3^2+cx_0x_1x_2 = 0.
\]
If $p=3$ then $X$ has the normal form
\[
c(x_0+x_1+x_2)(x_0x_1+x_0x_2+x_1x_2) + (x_0x_1x_2)
+ (x_0+x_1+x_2)x_3^2 =0 \ .
\]

\end{prop}

\begin{proof}
If $p\ne 2,3$, we can diagonalize $g$ to assume that it acts as 
\[
(t_0:t_1:t_2:t_3)\mapsto (t_0:\zeta t_1:\zeta^2t_2,-t_3)\ ,
\]
where $\zeta$ is a primitive third root of unity.
The equation becomes
\[
t_3^2t_0+t_0^3+t_1^3+t_2^3+ct_0t_1t_2 = 0
\]
(c.f. \cite{CAG}, (9.7.4)).
After the change of coordinates 
\[
(t_0:t_1:t_2:t_3)\mapsto (x_0+x_1+x_2:x_0+\zeta^2x_1+\zeta
x_2:x_0+\zeta x_1+\zeta^2x_2:x_3)\ ,
\]
we arrive at the equation
\[
(c+3)(x_0^3+x_1^3+x_2^3)+(18-3c)x_0x_1x_2+(x_0+x_1+x_2)x_3^2 = 0.
\]
The surface is singular if $c+3=0$, so we may rescale the variables to
obtain the desired normal form.

For $p=3$, we may simply follow the same proof as in
Lemma~\ref{lem:normalForms3D} demanding in addition that
the form is invariant under $x_3 \mapsto -x_3$.
Now the initial form \eqref{eq:3Dp3startForm} has $d=e=g=0$ and
all the transformations using the matrix \eqref{eq:3Dp3matrix}
can be done assuming $a_4=a_5=0$.
\end{proof}

\subsection{Class 8A}
\label{sec:8A}

\begin{lem}
If $p=2$, no smooth cubic surface admits an automorphism of class 8A.
\end{lem}

\begin{proof}
By checking Jordan canonical forms,
there is no element of $\GL_4(\bbk)$ of order $8$.
\end{proof}

\begin{lem} \label{lem:8Aautos}
Suppose $p \ne 2$ and $X$ is a smooth cubic surface admitting an
automorphism $g$ of class 8A.
There is a unique such surface defined by the normal form:
\begin{equation} \label{eq:8AnormalForm}
x_0^3 + x_0x_3^2 - x_1x_2^2 + x_1^2x_3 = 0
\end{equation}
where $g$ acts via a diagonal matrix with entries
$(1,\epsilon^6,\epsilon, \epsilon^4)$
where $\epsilon$ is a primitive eighth root of unity.

If $p=3$, then the surface is cyclic and contains a subgroup isomorphic
to $\calH_3(3) \rtimes 8$.
\end{lem}

\begin{proof}
Since $p \ne 2$, we can assume $g$ is defined using a diagonal matrix.
The argument is then the same as on page 489 of \cite{CAG}.
When $p=3$, this is a special case of the normal form
\eqref{eq:3AnormalFormP3}.
The element $g$ normalizes the Galois group of the cover and so
$g$ leaves invariant the group $N=\calH_3(3) \rtimes 2$
by Lemma~\ref{lem:invPoints}.
Note that $g^4$ is an element of $N$ of order $2$.
Since $\calH_3(3)$ is a normal subgroup of $N$ and is a Sylow $3$-group,
it is a characteristic subgroup of $N$.
Thus, it is a normal subgroup of $\langle N, g \rangle$.
By Schur-Zassenhaus, we conclude that
$\langle N, g \rangle \cong \calH_3(3) \rtimes 8$.
\end{proof}

\begin{lem} \label{lem:mush8A12A}
If $p=3$, a cubic surface admits an automorphism of class 8A if and only
if it admits an automorphism of class 12A.
\end{lem}

\begin{proof}
If there is an automorphism $g$ of class 8A, then $g^2$ has class 4A.
The Galois group of the cyclic cover is generated by an element
$h$ and $ghg^{-1}=h^{-1}$.
Thus $g^2$ and $h$ commute.  Thus $g^2h$ is of class 12A.
Conversely, if there is an automorphism of class 12A then there exist
elements of class 4A and 3A that commute with each other.
Thus $X$ is cyclic and has an element of order $4$ which normalizes the
Galois group of the cover.
By Lemma~\ref{lem:invPoints}, there is an automorphism of order $4$ of
the $9$ Eckardt points.
This is only possible if $c=0$ in \eqref{eq:3AnormalFormP3};
thus there is also an automorphism of order $8$.
\end{proof}

\subsection{Class 12A}
\label{sec:12A}

We already saw above that, when $p=3$, this case coincides with those of
8A by Lemma~\ref{lem:mush8A12A}.

\begin{lem} \label{lem:mush3C5A12A}
When $p=2$, if a smooth cubic surface admits an automorphism
whose class is one of 3C, 5A, 12A then it admits automorphisms of all
three classes.
\end{lem}

\begin{proof}
We have already seen that in Lemma~\ref{lem:mush3C5A} that the classes
3C and 5A coincide.
Note that 3C is a cyclic surface whose ramification divisor is a
supersingular elliptic curve.
The surface 12A is also a cyclic surface; by Lemma~\ref{lem:invPoints}
there is an automorphism of order $4$ acting faithfully on its
ramification divisor.  Thus 12A also has a supersingular ramification
divisor.
Thus their ramification divisors are isomorphic and so are the
corresponding surfaces.
\end{proof}

\begin{lem} \label{lem:12Aautos}
Suppose $p \ne 2,3$ and $X$ is a smooth cubic surface admitting an
automorphism $g$ of class 12A.  Then $X$ is cyclic and its ramification
divisor is an anharmonic cubic.
The surface $X$ has normal form
\[
x_0^3 + x_1^3 + x_2^3 + x_3^3 + \lambda x_0x_1x_2
\]
where $\lambda=3(\sqrt{3}-1)$.
Here $\Aut(X) \cong \calH_3(3) \rtimes 4$.
\end{lem}

\begin{proof}
Since $g^4$ is of class 3A, we have a cyclic cubic surface.
Since $g$ acts with order $4$ on the $3$-torsion subgroup of the
ramification divisor $C$, we conclude $C$ is anharmonic
by Lemma~\ref{lem:invPoints}.
It remains only to find an automorphism of order $4$ on a cyclic
surface.
The following matrix, from \S 3.1.4 of \cite{CAG},
\[
\frac{1}{\sqrt{3}}
\begin{pmatrix}
1 & 1 & 1 & 0 \\
1 & \zeta & \zeta^2 & 0 \\
1 & \zeta^2 & \zeta & 0 \\
0 & 0 & 0 & \sqrt{3}
\end{pmatrix}\ ,
\]
where $\zeta$ is a primitive third root of unity,
has order $4$ and leaves invariant the given normal form.
Note that $\calH_3(3) \rtimes 4$ must be the full automorphism group
since there is only one cyclic surface structure on $X$.
\end{proof}

\section{Collections of Eckardt points}
\label{sec:EckardtCollections}

The possible automorphisms of cubic surfaces are closely controlled by
the subgroup generated by reflections.
From Theorem~\ref{thm:Eckardt}, along with
Lemmas~\ref{lem:3Cautos},
\ref{lem:2Bautos},
\ref{lem:3Aautos},
\ref{lem:3Dautos},
\ref{lem:4Bautos},
\ref{lem:5Aautos} and
\ref{lem:6Eautos},
we have the following:

\begin{lem}
Let $X$ be a smooth cubic surface with an automorphism $g$ of any class
except 4A, 8A or 12A.
Then the reflections corresponding to the Eckardt points on $X$ generate
a subgroup containing $g$.
\end{lem}

\subsection{Reflection groups acting on cubic surfaces}
\label{sec:reflectionGroups}

In this section, we classify the possible collections of Eckardt points
that may occur on a smooth cubic surface 
along with the reflection groups they generate.
Over finite fields, the possibilities are classified
in Lemma~20.2.10~of~\cite{Hirschfeld}.
However, Hirschfeld's proof does not enumerate the corresponding reflection groups.
Our proof follows similar lines.

Let $X$ be a smooth cubic surface.
Let $\calE$ be the set of Eckardt points of $X$.
Let $\calL$ (resp. $\calT$) be the set of lines passing through two
points of $\calE$ that lie on $X$ (resp. \emph{do not} lie on $X$).
The lines in $\calL$ are all exceptional lines; they contain exactly $2$
points of $\calE$ if $p\ne 2$ or exactly $5$ points of $\calE$ if
$p = 2$.
The lines in $\calT$ are precisely the trihedral lines of $X$ and they
contain exactly three points of $\calE$.

Since Eckardt points can always be identified with tritangent planes,
one can view a triple $(\calE, \calL, \calT)$ as a kind of incidence
structure where $\calE$ is a subset of the 45 tritangent trios
and $\calL$ and $\calT$ are subsets of the power set of $\calE$
corresponding to the Eckardt points lying on the corresponding line.
We say two such triples are \emph{equivalent} if they are in the same
$W(\sfE_6)$-orbit.

\begin{lem} \label{lem:configurations}
Let $X$ be a smooth cubic surface with corresponding triple
$(\calE, \calL, \calT)$.
Let $G$ be the subgroup of $\Aut(X)$ generated by reflections
(equivalently, $G$ is generated by the reflections corresponding to the
Eckardt points $\calE$).
Up to equivalence, the possibilities are enumerated in
Table~\ref{tbl:configurations}.  
\end{lem}

\begin{table}[h]
\centering
\scalebox{0.75}{%
\begin{tabular}{|c|c|ccc|c|c|c|}
\hline
Name & $\dim \Span \calE$ & $|\calE|$ & $|\calL|$ & $|\calT|$ & $p$ & G
& Surfaces\\
\hline
$C_0$ & -1 & 0 & 0 & 0 & any & $1$ & 1A\\
\hline
$C_1$ & 0 & 1 & 0 & 0 & any & $2$
& 2A, 4A ($p \ne 2$), 8A ($p \ne 2,3$)\\
\hline
$C_2$ & 1 & 2 & 1 & 0 & $\ne 2$ & $2^2$ & 2B ($p \ne 2$)\\
$C_3$ & 1 & 3 & 0 & 1 & any & $\frakS_3$ & 3D\\
$C_5$ & 1 & 5 & 1 & 0 & $=2$ & $2^4$ & 2B ($p = 2$)\\
\hline
$C_4$ & 2 & 4 & 3 & 1 & $\ne 2$ & $\frakS_3 \times \frakS_2$
& 6E ($p \ne 2$)\\
$C_6$ & 2 & 6 & 3 & 4 & $\ne 2$ & $\frakS_4$
& 4B ($p \ne 2$)\\
$C_9$ & 2 & 9 & 0 & 12 & any & $H_3(3) \rtimes 2$
& 3A, 8A=12A ($p = 3$), 12A ($p \ne 2,3$)\\
$C_{13}$ & 2 & 13 & 3 & 16 & $=2$ & $2^3 \rtimes \frakS_4$
& 4A=4B=6E ($p=2$)\\
\hline
$C_{10}$ & 3 & 10 & 15 & 10 & $\ne 2,5$ & $\frakS_5$
& 5A ($p \ne 2$)\\
$C_{18}$ & 3 & 18 & 27 & 42 & $\ne 3$ & $3^3 \rtimes \frakS_4$
& 3C ($p \ne 2,3$)\\
$C_{45}$ & 3 & 45 & 27 & 240 & $=2$ & $\PSU_4(2)$
& 3C=5A=12A ($p =2$) \\
\hline
\end{tabular}
}
\caption{Possible collections of Eckardt points.}
\label{tbl:configurations}
\end{table}

\begin{proof}
We proceed by building the triples inductively by adding points
to existing triples taking into account the constraints imposed
on collinear Eckardt points by
Lemmas~\ref{lem:excLineEck}~and~\ref{cor:triadLineEck}.

The cases $C_0$ and $C_1$ are immediate.
The classification of the collinear collections $C_2$,
$C_3$, $C_5$ and their automorphisms follow immediately from
Lemmas~\ref{lem:excLineEck}~and~\ref{cor:triadLineEck}.

Suppose $\calE$ spans a plane $P$.
Let us consider the collection corresponding to $C_9$.
The reflections corresponding to any three non-collinear Eckardt points
in $C_9$ generate the entire reflection group associated to $C_9$.
Indeed, any two Eckardt points give rise to an element of order $3$
which has $3$ orbits each of $3$ collinear Eckardt points.
Since we have two pairs lying on different lines, we get a transitive
action of the group on the Eckardt points and thus their corresponding
reflections.

Thus, if a general planar collection $\calE$ can be embedded in $C_9$
as an abstract incidence structure, then we obtain $C_9$ itself.
By Lemma~\ref{lem:HermitianSubspaces}, $C_9$ and $C_{13}$ are the maximal
possible planar collections.
For the remaining planar cases, we may assume $\calE$ is a subset
of $C_{13}$; thus $P$ is a tritangent plane.

All of the Eckardt points must lie on the three exceptional lines in the
tritangent plane.  Since the points of $\calE$ are not collinear,
at least two points $p,q$ of $\calE$ lie on distinct exceptional lines.
The line $\overline{pq}$ is a trihedral line in $\calT$ whose third point
$r$ lies on the remaining exceptional line.
These points generate a group $\frakS_3$ which permutes the exceptional
lines, so there are an equal number of points of $\calE$ on each.
Since $p,q,r$ are collinear, there must be at least one more point.
If the characteristic is not $2$, there can only be two points on each
exceptional line.  If the point $t$ is the intersection of the three
lines we have the case $C_4$, otherwise, there must be 6 points in
$\calE$ and we are in case $C_6$.
If the characteristic is $2$, we must have $5$ points on every line and
we are in case $C_{13}$.

Now assume $E$ spans the whole space.
If two lines $\ell, \ell'$ from $\calT$ are skew, then we claim $X$ is
the Fermat surface.
Suppose there exists a third trihedral line $\ell''$ that passes through
Eckardt points of both $\ell$ and $\ell'$.
In this case $\{ \ell, \ell'' \}$ and $\{ \ell', \ell'' \}$ span
two distinct planes that each contain the $C_9$ collection.
This means that there are two distinct cyclic surface structures on $X$,
which is only possible if $X$ is the Fermat cubic surface.
We are left with only the possibility that all the lines between Eckardt
points on $\ell$ and $\ell'$ are exceptional.  This means that the
reflections corresponding to the Eckardt points on $\ell$ commute with
those from $\ell'$.
Thus, there is a group isomorphic $\frakS_3 \times \frakS_3$ acting on $X$.
This implies the existence of an automorphism of class 3C and thus
we must have the Fermat cubic surface by 
Theorem~\ref{thm:3matrices}.
This produces case $C_{18}$ when $p>3$
and $C_{45}$ when $p=2$.
This case does not occur when $p=3$.
Note that $C_{18}$ has two points on every exceptional line so it is
certainly maximal when $p \ne 2$.

We may now assume that any pair of trihedral lines in $\calE$ must intersect
in a point outside $X$.
Whenever this occurs, the two lines span a plane $P$ and $P \cap \calE$
must be a configuration of type $C_6$.
Thus the plane $P$ contains exactly 4 trihedral lines from $\calT$.
Let  $\calP$ be the subsets of $\calL$ corresponding to such planes $P$.
We thus want subsets of four elements where each pair of subsets
contains exactly one common element.
By directly trying to build such a collection, we see there is
only one possibility for $\calP$:
\[
\{1,2,3,4\}, \{1,5,6,7\}, \{2,5,8,9\}, \{3,6,8,10\}, \{4,7,9,10\} \ .
\]
This last case is $C_{10}$.
This is the Clebsch cubic surface.  If $p \ne 2$, then
the automorphism group of such a
surface is maximal, so the collection $C_{10}$ must be as well.
If $p=2$, then we obtain the collection $C_{45}$.
\end{proof}

\begin{remark} The Clebsch surface is classically called the
\emph{Clebsch diagonal cubic surface}.
We recall the reason for this name (see also \S{}9.5.4 of \cite{CAG}).
Consider the Clebsch in the coordinates \eqref{eq:charFreeCubic}.
Let $P$ be the linear subspace given by $\sum_{i=0}^4 x_i = 0$,
which is isomorphic to $\bbP^3$.
The \emph{Sylvester Pentahedron} is the set of 5 planes in $P$
defined by $x_i=0$ for $i=0,\ldots, 4$.
The planes are called \emph{faces} of the pentahedron,
intersections of pairs of planes form the \emph{edges},
and intersections of triples of planes are the \emph{vertices}.
The 10 vertices are Eckardt points $\calE$ of the Clebsch surface,
which are all of them when $p \ne 2,5$.
The 10 edges are the trihedral lines $\calT$ in $C_{10}$ described above.

Each face contains 6 coplanar Eckardt points as in the
$C_6$ arrangement described above.  The three exceptional lines on each
face are precisely the three diagonals of the tetrahedron defined by the other
4 faces.  Thus the 15 exceptional lines $\calL$ in $C_{10}$ are the
``diagonals'' from which the surface gets its name.
The remaining $12$ exceptional lines form a double-sixer.

When $p=2$, a Sylvester pentahedron can be constructed, but it is not
unique.
In the coordinates \eqref{eq:charFreeCubic}, we see that the three
diagonals of each quadrilateral meet at a point, which produces
additional Eckardt points.
This is yet another explanation for why the Fermat and the Clebsch surfaces
are isomorphic when $p=2$.
 \end{remark}
 
\subsection{Automorphisms not arising from reflection groups}
\label{sec:nonReflection}

It remains to consider automorphism groups that are \emph{not} generated
by reflection groups.  Since surfaces admitting automorphisms of class
8A or 12A also admit automorphisms of class 4A, the following
essentially settles the matter.

\begin{lem} \label{lem:handle4A}
Suppose $X$ admits an automorphism $g$ of class $4A$.
If $p=2$, then $g$ is contained in a reflection subgroup of $\Aut(X)$.
If $p \ne 2$, then either $\Aut(X)$ is cyclic, or $X$ is the 12A surface.
\end{lem}

\begin{proof}
When $p=2$, we know from Lemma~\ref{lem:mash} that $X$
is also a surface admitting an automorphism of class 4B and that $g$ is
contained in a reflection group.

Now suppose $p \ne 2$ and $p_0$ is the Eckardt point corresponding to
the reflection $g^2$.
From Lemma~\ref{lem:4Adesc}, we see that if $q$ is an Eckardt point
lying on an exceptional line $\ell$ containing $p_0$ then $g(q)$ is a
third Eckardt point on $\ell$ --- a contradiction when $p \ne 2$.
If $q$ is an Eckardt point lying on a trihedral line $\ell$ containing
$p_0$ then $g^2(q)$ is the third Eckardt point on the line.
Since there at most three Eckardt points on $\ell$, we see that
$g(q)$ must lie on a distinct trihedral line.
We conclude that if there is more than one Eckardt point then the
configuration of Eckardt points can only be $C_9$.
Since $g$ does not leave invariant any trihedral line, we have an induced action
of order $4$ on the ramification locus of a cyclic cubic surface.
Thus, $X$ is the 12A surface.

It remains to consider the case where the only Eckardt point is $p_0$.
Aside from 2A, 4A, 8A, all the other realizable classes imply the
existence of other Eckardt points.
Thus we may assume $\Aut(X)$ consists of elements only of classes
2A, 4A, 8A that leave invariant the tritangent plane containing $p_0$.
One now uses a similar argument as page~495~of~\cite{CAG}.
Let $H$ be the image of $\Aut(X)$ on the projectivized tangent space of
$p_0$.
Since $\Aut(X)$ must have order a power of $2$, the order is coprime to
the characteristic.
Thus, the morphism $\Aut(X) \to H$ has cyclic kernel.
Any other automorphism can only interchange two of the three lines
passing through $p_0$.
Since $\Aut(X)$ has order a power of $2$, this implies that $H=2$ or $H=1$.
Any automorphism interchanging two lines must have order $8$
by the description in Lemma~\ref{lem:4Adesc}.
Thus $\Aut(X)$ is a non-split central extension of cyclic groups
and so must be cyclic.
\end{proof}

\section{Proof of the Main Theorem}
\label{sec:proof}

\begin{proof}[Proof of Theorem~\ref{thm:main}]
Let $X$ be a smooth cubic surface.
Lemmas~\ref{lem:excludedClasses},
\ref{lem:mush3C6C6F9A},
and \ref{lem:exclude5A}
tell us which automorphism classes do not occur.
Lemmas~\ref{lem:mush3A6A},
\ref{lem:mush3C6C6F9A},
\ref{lem:mash},
and~\ref{lem:mush3C5A12A}
tell us that certain strata coincide with other strata. 

We now prove that the automorphism groups that appear in
Table~\ref{tbl:cubicAutos} actually act on a smooth cubic surface and
that no larger groups occur.
Let $R$ be the subgroup of $\Aut(X)$ generated by reflections.
The possible reflection groups are enumerated in
Lemma~\ref{lem:configurations} and coincide with the automorphism groups
described in the normal forms above.
Thus, Lemmas~\ref{lem:3Cautos},
\ref{lem:2Bautos},
\ref{lem:3Aautos},
\ref{lem:3Dautos},
\ref{lem:4Bautos},
\ref{lem:5Aautos}, and
\ref{lem:6Eautos}
tell us that the only automorphism classes not necessarily contained in
$R$ are 4A, 8A and 12A.
Aside from these cases, $R=\Aut(X)$ so we have a full description of
their automorphism groups.

From Lemma~\ref{lem:handle4A}, we have either that the surface has
automorphism generated by an automorphism of class 4A or 8A,
or is the surface 12A.
The surface 12A coincides with 3C when $p=2$, so this has already been
handled.
In Lemma~\ref{lem:12Aautos}, the full automorphism group of 12A is
computed when $p \ne 2, 3$.
Finally, when $p=3$, the 12A stratum coincides with the 8A stratum ---
its automorphism group is described completely in
Lemma~\ref{lem:8Aautos}.

It remains to prove that the strata from Table~\ref{tbl:cubicAutos}
are all distinct and have the expected dimensions.
In other words, we need to prove that there are no
other unusual coincidences like the Fermat and the Clebsch coinciding in
characteristic $2$.
Note that each stratum is irreducible.
Indeed, each conjugacy class of cyclic group can be represented by only
one matrix in $\GL_4(k)$.
Thus each stratum is the image of an eigenspace of a linear endomorphism
of $S^3(\bbk^4)$.
We need to show that each of the arrows in
Figures~\ref{fig:cubicSpecialization},
\ref{fig:cubicSpecializationP2},
\ref{fig:cubicSpecializationP3}, and
\ref{fig:cubicSpecializationP5} are not actually equalities.

First we consider $p \ne 2$.
By Lemma~\ref{lem:SylvesterAutos},
the strata 1A, 2A, 2B, 3D, 4B, 6E, and 5A are distinct and have the
expected dimensions.
When $p \ne 2$, there is a $1$-dimensional family of 4A surfaces
distinct from the other strata by Lemma~\ref{lem:normalForms4A}.
The remaining surfaces 3A, 3C, 12A, 8A are distinct, equal, or
non-existent by the characterization of cyclic surfaces in
Section~\ref{sec:3C}.

Now we consider $p = 2$.
The surfaces 2B, 4A, 3C are in bijection with del Pezzo surfaces of
degree 4 via blowing up the canonical point.
Thus they are distinct and their strata have the appropriate dimensions
by the results of Section~\ref{sec:DP4}.
The 3A stratum is in bijection with isomorphism classes
of cubic plane curves and so has dimension $1$ and does not coincide
with 4A.
The class 3D does not coincide with 2B since the latter does not admit
automorphisms of order $3$.
Since the 3D stratum is connected and contains the distinct $1$-dimensional
strata 4B and 3A, it is in fact $2$-dimensional in view of the normal
form from Lemma~\ref{lem:normalForms3D}.
Finally, the 2A stratum has dimension $3$ since it contains
the distinct $2$-dimensional strata 2B and 3D.
\end{proof}

\begin{cor} \label{cor:reflectionIndex}
Let $X$ be a smooth cubic surface.  Then the subgroup of $\Aut(X)$
generated by reflections has index $1$, $2$, or $4$ in $\Aut(X)$.
The surfaces of index $2$ are 4A (if $p \ne 2$) and 12A (if $p \ne
2,3$).
The surface of index $4$ is 8A.
\end{cor}

Since we have classified all the strata, from the explicit normal forms
we have the following:

\begin{cor} \label{cor:unirationalStrata}
Every nonempty stratum of the coarse moduli space $\calM_{cub}$ is unirational.
\end{cor}

\section{Lifting to characteristic zero}
\label{sec:lifting}

In this section we prove Theorem~\ref{thm:liftingIntro}.

Let $\bbk$ be an algebraically closed field of positive characteristic
$p$ and let $A$ be a complete discrete valuation ring of characteristic
zero with residue field $\bbk$.
Such an $A$ always exists by the Witt vector construction
(see Theorem~II.5.3~in~\cite{SerreLF}).
Let $X$ be a smooth projective $k$-variety with an action of a finite
group $G$.

We say that \emph{$(X,G)$ lifts to $A$} if there exists an $A$-scheme
$X_A$, smooth and projective over $A$, with a $G$-action such that
the special fiber is $G$-equivariantly isomorphic to $X$.
We say \emph{$(X,G)$ lifts to characteristic $0$} if there exists a lift
to $(X,G)$ to $A$ for some $A$ as above.

Note that not every smooth cubic $G$-surface in positive characteristic
lifts to a $G$-surface in characteristic $0$.
An obvious obstruction to lifting is that $G$, as a subgroup of
$W(\sfE_6)$, may not be realizable as automorphisms on \emph{any}
cubic surface in characteristic $0$.
In this section, we show that this is the only obstruction.

\begin{thm} \label{thm:lifting}
Let $X$ be a smooth cubic surface defined over $\bbk$ with an action of
finite group $G$ such that $G \subset W(\sfE_6)$ also occurs as a group
of automorphisms for some cubic surface in characteristic $0$.
Then, after possibly replacing $A$ with a totally ramified quadratic
extension, $(X,G)$ lifts to $A$.  
\end{thm}

\begin{proof}
Let $K$ be the field of fractions of
$A$ and $\frakm$ the maximal ideal of $A$.
In \S{}5~of~\cite{SerreBourbaki}, Serre showed that any rational $G$-surface lifts to
characteristic $0$ in the tame case (in other words, when $p$ is coprime
to $|G|$.)
Thus, it remains only to consider the cases where $p$ divides $|G|$.

In what follows, we consider a subgroup $G \subset W(\sfE_6)$ that
is realized by automorphisms on some smooth cubic surface $Y$ of characteristic $0$.
Then we consider a smooth cubic surface $X$ over $\bbk$ that realizes
the group $G$ (assuming such a surface exists).
We then show that $(X,G)$ lifts to $A$ or a quadratic extension of $A$.
Looking at the classification, it suffices to assume that $G=\Aut(Y)$ for an
appropriate $Y$.

Consider the standard representation of $\frakS_5$ on $V=\bbZ^5$ that
acts by permuting basis elements.
Given a partition $\lambda$ of $5$, let $H_\lambda$ be the subgroup
of $\frakS_5$ isomorphic to $\frakS_{\lambda_1} \times \cdots
\frakS_{\lambda_n}$ that permutes the basis elements in each part of
$\lambda$.
By the fundamental theorem of symmetric functions,
$\bbZ[V]^{H_\lambda}$ is a polynomial ring generated by elementary
$\lambda$-symmetric functions.
If $W$ is the quotient of $V$ by the invariant linear form $x_1 + \cdots + x_5$,
then $\bbZ[W]^{H_\lambda}$ is a polynomial ring with a generating
set $f_1, \ldots, f_s$ where each $f_i$ contains a monic monomial.
Thus $f_1, \ldots, f_s$ generate $R[W_R]^{H_\lambda}$ for any ring $R$.

Now, consider a cubic $G$-surface $(X,G)$ over $\bbk$ that is one
of the cases 2A, 2B, 3D, 4B, 6E and 5A
appearing in Lemma~\ref{lem:SylvesterAutos}.
In order to lift $G$ to characteristic $0$, we may assume that
$G$ is of the form $H_\lambda$ described above.
Moreover, the normal forms show that we may assume $H_\lambda$ acts on
the space $W_\bbk$ in the manner described above.
Thus $X$ is defined by a cubic polynomial that can be written as a
linear combination of monomials in $f_1, \ldots, f_s$.
Lifting the coefficients to $A$ we have a lift of $(X,G)$.

Now, the class 3C corresponds to the Fermat cubic surface.  From the
normal form and the description of the automorphism group $3^3 \rtimes
\frakS_4$, it evidently lifts to characteristic $0$ when $p=2$.
Since it is not defined when $p=3$, there is nothing to check.

The automorphism group for 8A is only wild if $p=2$, but then it doesn't exist.

We now turn to the group 4A.  This is only wild when $p=2$.
By possibly replacing $A$ with a totally ramified quadratic extension
(see Proposition~II.2.3~in~\cite{SerreLF}),
we can assume that $A$ contains a primitive $4$th root of unity $i$.
Fix a primitive third root of unity $\zeta$ in $A$,
which exists by Hensel's Lemma since $\bbk$ is algebraically closed.
Consider the element $r=i(2\zeta+1)$, which is a square root of $3$.
Note that $r^{-1}=r/3 \in A$ and $r \equiv 1 \mod \mathfrak{m}$.

Consider the group $G$ of order $4$ generated by the matrix
\[
\frac{1}{r}
\begin{pmatrix}
1 & 1 & 1 \\
1 & \zeta & \zeta^2 \\
1 & \zeta^2 & \zeta
\end{pmatrix}\ ,
\]
in $A$.
Diagonalizing in $K$, or finding the Jordan canonical form in
$k$, we find that this group corresponds to the non-trivial $3 \times 3$
block of a matrix of class 4A.
One checks in these more convenient bases that, over both fields, the
invariants of degree $\le 3$ are generated by elements of degree $1$,
$2$, and $3$ that have no relations in degree $\le 3$.

The following three polynomials in $A[x_0,x_1,x_2]$ are $G$-invariant:
\begin{align*}
f_1 &= (r+1)x_0 + x_1 + x_2\\
f_2 &= x_0(x_0 + x_1 + x_2)\\
f_3 &= x_0^3 + x_1^3 + x_2^3 + 3(r-1)x_0x_1x_2 \ .
\end{align*}
Their leading terms mod $\frakm$ are $\{x_1,x_0^2,x_0^3\}$,
so they have no relations in degree $\le 3$.
Thus, $\{f_1^3,f_1f_2,f_3\}$ is a basis for the invariants of degree
$3$ over both fields $k$ and $K$.

Extending $G$ to act on $A[x_0,x_1,x_2,x_3]$ by the trivial action on
$x_3$, we conclude that the vector space of invariants of degree $3$ is
spanned by 
\[ f_1^3, f_1f_2, f_3, f_1^2x_3, f_2x_3, f_1x_3^2, x_3^3 \]
over both fields $K$ and $k$.
Thus we may lift 4A to characteristic $0$ from characteristic $2$.

It remains to consider the cases $3A$ and $12A$.
From Proposition~\ref{prop:normalForms3Apnot3},
the cyclic surfaces in characteristic $2$ have the same normal form as
in characteristic $0$:
\[
x_0^3 + x_1^3 + x_2^3 + x_3^3 + c x_0x_1x_2
\]
and the automorphism groups are defined over $\bbZ[\zeta]$ as well.
Thus they lift.
For the case 12A, we may already lift the natural automorphism 4A
by setting $c = 3(r-1)=3(\sqrt{3}-1)$ and noticing
that our normal form is $f_3 + x_3^3$.
Thus the  group $\calH_3(3) \rtimes 4$ from class 12A lifts when $p=2$.

If $p=3$,
by possibly replacing $A$ with a totally ramified quadratic extension,
we may assume that $A$ contains a primitive third root of unity
$\zeta$.  Note that $\zeta \equiv 1 \mod \frakm$.
Consider the group $G$ of order $3$ generated by the matrix
\[
\begin{pmatrix}
1 & 0 & 0 & 1 \\
0 & 1 & 0 & 0 \\
0 & 0 & 1 & 0 \\
0 & 0 & 0 & \zeta
\end{pmatrix}\ ,
\]
in $A$.  As above, this is of type $3A$ in both $K$ and $k$.
Consider the following cubic homogeneous form
\begin{gather*}
F = \zeta x_0^3 + (\zeta-1)x_0^2x_3 - x_0x_3^2
+ (1-\zeta)x_0x_1^2 + x_1^2x_3 - x_1x_2^2\\
+ c ( -3\zeta x_0^2x_1 + 2(1-\zeta)x_0x_1x_3 + x_1x_3^2 )
\end{gather*}
where $c$ is a parameter.
One checks that $F$ is $G$-invariant.
Passing to the residue field $k$, we have:
\[
F_k = x_0^3 - x_0x_3^2 + x_1^2x_3 - x_1x_2^2 + c x_1x_3^2
\]
so for any $3A$ surface we may construct a lift of $F$.
Note that we use a slightly modified normal form here,
see Remark~\ref{rem:defOverF9}.

Since both $F_K$ and $F_k$ are $3A$-surfaces, they each have exactly $9$
Eckardt points lying in the plane $P: x_3=0$.
The Hessian cubic of $X|_P$ is given by
\[
24\left[(3\zeta+3)c^2+2\zeta+1)x_0^2x_1 + (2\zeta+1)cx_0x_1^2 +\zeta x_0x_2^2
+\zeta x_1^3 - \zeta c x_1x_2^2\right] = 0
\]
and its intersection with $X \cap P$ gives the $9$ Eckardt points of
$X_K$.  Dividing first by $24$, the Hessian cubic reduces to
$x_0x_2^2+x_1^3-cx_1x_2^2$
modulo $\frakm$, which one checks cuts out the $9$ Eckardt points in
$X_k \cap P_k$.  Thus, the nine Eckardt points on $X_K$ and $X_k$ lift
to $A$-points of $X$.

Following the proof of Theorem~\ref{thm:Eckardt}, an Eckardt point in
$X_K$ produces an axis plane and a tangent plane, which are defined by
linear forms $L_1$ and $L_2$.
By rescaling by elements of $\frakm$ if necessary, these linear forms
can be defined over $A$ such that they are non-zero modulo $\frakm$.
Moreover, these correspond to distinct planes in $X_k$ since the
axis and tangent plane exist and are distinct here as well.
Thus we may set $x_0=L_1$ and $x_1=L_2$ and
extend to a basis for $A^4$ to write $F$ in the form
\[
x_0^2x_1 + C(x_1,x_2,x_3) = 0
\]
as in Theorem~\ref{thm:Eckardt}.
This is clearly invariant under the reflection taking $x_0 \mapsto
-x_0$.
Thus, all $9$ reflections on $X_k$ can be lifted to $X_A$.
This means that group of order $54$ that they generate lifts as desired.

Finally, if we consider an element of order $12$
generated by
\[
\begin{pmatrix}
1 & 0 & 0 & 1 \\
0 & -1 & 0 & 0 \\
0 & 0 & i & 0 \\
0 & 0 & 0 & \zeta
\end{pmatrix}\ ,
\]
then the same argument works for the $12A$ surface where $c=0$.
\end{proof}

\begin{remark}
From the proof of Theorem~\ref{thm:lifting}, we can make precise the
conditions required of $A$ so that we do not have to replace it with an
extension.
If $p=2$ for the classes 4A and 12A, then $A$ must contain a primitive
$4$th root of unity $i$;
moreover, this is necessary since a matrix of class 4A has eigenvalues $i,-1,1,1$
over $K$.
If $p=3$ for the classes 3A and 12A, then $A$ must contain a primitive
third root of unity; again this is necessary since a matrix of class 3A
has eigenvalues $\zeta,1,1,1$ over $K$.
In all remaining cases, the result holds with no conditions on $A$.
\end{remark}

\newpage
\section*{Appendix}
\label{sec:appendix}

 \begin{table}[ht]
\begin{tabular}{|| c |c| c | c |c | c | c | c | c| l||}
\hline
&Atlas &Carter&Manin&Ord&$\#C$ &$\Tr$&Char\\ \hline
x&1A&$\emptyset$&$c_{1}$&1&\footnotesize{51840}&6&$\Phi_{1}^{6}$\\ \hline
x&2A&$4A_1$&$c_3$&2&1152&-2&$\Phi_{1}^{2}\Phi_{2}^{4}$\\ \hline
x&2B&$2A_1$&$c_2$&2&192&2&$\Phi_{1}^{4}\Phi_{2}^{2}$\\ \hline
&2C&$A_1$ &$c_{16}$&2&1440&4&$\Phi_{1}^{5}\Phi_{2}^{1}$\\ \hline
&2D&$3A_1$&$c_{17}$ &2&96&0&$\Phi_{1}^{3}\Phi_{2}^{3}$\\ \hline
x&3A&$3A_2$&$c_{11}$&3&648&-3&$\Phi_{3}^{3}$\\ \hline
x&3C&$A_2$&$c_6$&3&216&3&$\Phi_{1}^{4}\Phi_{3}^{1}$\\ \hline
x&3D&$2A_2$&$c_9$&3&108&0&$\Phi_{1}^{2}\Phi_{3}^{2}$\\ \hline
x&4A&$D_4(a_1)$&$c_4$&4&96&2&$\Phi_{1}^{2}\Phi_{4}^{2}$\\ \hline
x&4B&$A_1+A_3$&$c_5$&4&16&0&$\Phi_{1}^{2}\Phi_{2}^{2}\Phi_{4}^{1}$\\ \hline
&4C&$2A_1+A_3$&$c_{19}$&4&96&-2&$\Phi_{1}^{1}\Phi_{2}^{3}\Phi_{4}^{1}$\\ \hline
&4D&$A_3$&$c_{18}$&4&32&2&$\Phi_{1}^{3}\Phi_{2}^{1}\Phi_{4}^{1}$\\ \hline
x&5A&$A_4$&$c_{15}$&5&10&1&$\Phi_{1}^{2}\Phi_{5}^{1}$\\ \hline
x&6A&$E_6(a_2)$&$c_{12}$& 6&72&1&$\Phi_{3}^{1}\Phi_{6}^{2}$\\ \hline
x&6C&$D_4$&$c_{21}$& 6&36&1&$\Phi_{1}^{2}\Phi_{2}^{2}\Phi_{6}^{1}$\\ \hline
x&6E&$A_1+A_5$&$c_{10}$&6&36&-2&$\Phi_{2}^{2}\Phi_{3}^{1}\Phi_{6}^{1}$\\ \hline
x&6F&$2A_1+A_2$&$c_{8}$&6&24&-1&$\Phi_{1}^{2}\Phi_{2}^{2}\Phi_{3}^{1}$\\ \hline
&6G&$A_1+A_2$&$c_{7}$&6&36&1&$\Phi_{1}^{3}\Phi_{2}^{1}\Phi_{3}^{1}$\\ \hline
&6H&$A_1+2A_2$&$c_{22}$&6&36&-2&$\Phi_{1}^{1}\Phi_{2}^{1}\Phi_{3}^{2}$\\ \hline
&6I&$A_5$&$c_{23}$&6&12&0&$\Phi_{1}^{1}\Phi_{2}^{1}\Phi_{3}^{1}\Phi_{6}^{1}$\\ \hline
x&8A&$D_5$&$c_{20}$&8&8&0&$\Phi_{1}^{1}\Phi_{2}^{1}\Phi_{8}^{1}$\\ \hline
x&9A&$E_6(a_1)$&$c_{14}$&9&9&0&$\Phi_{9}^{1}$\\ \hline
&10A&$A_1+A_4$&$c_{25}$&10&10&-1&$\Phi_{1}^{1}\Phi_{2}^{1}\Phi_{5}^{1}$\\ \hline
x&12A&$E_6$&$c_{13}$&12&12&-1&$\Phi_{3}^{1}\Phi_{12}^{1}$\\ \hline
&12C&$D_5(a_1)$&$c_{24}$&12&12&1&$\Phi_{1}^{1}\Phi_{2}^{1}\Phi_{4}^{1}\Phi_{6}^{1}$\\ \hline
\end{tabular}
 \centering
 \caption{Conjugacy classes of elements of finite order in $W(E_6)$}\label{tableconj}
\end{table}
Here in the first column we mark the conjugacy classes of elements which
may arise in automorphism groups of nonsingular cubic surfaces (in any
charactersitic).

\begin{figure}[!htbp]
\centering
\scalebox{0.88}{%
$$ \xymatrix{
& 1A \ar[d] & \\
& 2A \ar[dl] \ar[dr] & \\
2B \ar[d] & &
3D \ar[dll] \ar[d] \\
4A=4B=6E \ar[dr] & &
3A \ar[dl] \\
& 3C=5A=12A
}
$$
}
\caption{Specialization of strata in $\mathcal{M}_{\textrm{cub}}$ when
$p=2$.}
\label{fig:cubicSpecializationP2}
\end{figure}

\begin{figure}[!htbp]
\centering
\scalebox{0.88}{%
$$ \xymatrix{
& 1A \ar[d] \\
& 2A \ar[dl] \ar[d] \ar[ddrr] \\
2B \ar[d] \ar[dr] &
3D \ar[dl] \ar[d] \ar[dr] \\
4B \ar[dr] &
6E \ar[d] &
3A \ar[d] &
4A \ar[dl] \\
& 5A & 8A=12A\\
}
$$
}
\caption{Specialization of strata in $\mathcal{M}_{\textrm{cub}}$ when
$p=3$.}
\label{fig:cubicSpecializationP3}
\end{figure}

\begin{figure}[!htbp]
\centering
\scalebox{0.88}{%
$$ \xymatrix{
& 1A \ar[d] \\
& 2A \ar[dl] \ar[d] \ar[ddrr] \\
2B \ar[d] \ar[dr] &
3D \ar[dl] \ar[d] \ar[dr] \\
4B \ar[dr] &
6E \ar[d] &
3A \ar[dl] \ar[d] &
4A \ar[dl] \ar[d] \\
 & 3C & 12A & 8A\\
}
$$
}
\caption{Specialization of strata in $\mathcal{M}_{\textrm{cub}}$ when
$p=5$.}
\label{fig:cubicSpecializationP5}
\end{figure}

\begin{landscape}
\begin{table}[htbp]
\centering
\renewcommand{\arraystretch}{1.3}
\scalebox{0.88}{%
\begin{tabular}{|c|c|cc|cccccccccccccccc|}
\hline
Name & $\operatorname{char}$ & $\Aut(X)$ & Order & 1A & 2A & 2B & 3A
& 3C & 3D & 4A & 4B & 5A & 6A & 6C & 6E & 6F & 8A & 9A & 12A\\
\hline
1A & any & $1$ & 1 & 1 &  &  &  &  &  &  &  &  &  &  &  &  &  &  & \\
\hline
2A & any & $2$ & 2 & 1 & 1 &  &  &  &  &  &  &  &  &  &  &  &  &  & \\
\hline
2B & $\ne 2$ & $2^2$ & 4 & 1 & 2 & 1 &  &  &  &  &  &  &  &  &  &  &  &  & \\
 & 2 & $2^4$ & 16 & 1 & 5 & 10 &  &  &  &  &  &  &  &  &  &  &  &  & \\
\hline
3A & any & $\calH_3(3) \rtimes 2$ & 54 & 1 & 9 &  & 2 &  & 24 &  &  &  & 18 &  &  &  &  &  & \\
\hline
3C & $\ne 2,3$ & $3^3 \rtimes \frakS_4$ & 648 & 1 & 18 & 27 & 8 & 6 & 84 &  & 162 &  & 72 & 36 & 36 & 54 &  & 144 & \\
 & 2 & $\PSU_4(2)$ & 25920 & 1 & 45 & 270 & 80 & 240 & 480 & 540 & 3240 & 5184 & 720 & 1440 & 1440 & 2160 &  & 5760 & 4320\\
\hline
3D & any & $\frakS_3$ & 6 & 1 & 3 &  &  &  & 2 &  &  &  &  &  &  &  &  &  & \\
\hline
4A & $\ne 2$ & $4$ & 4 & 1 & 1 &  &  &  &  & 2 &  &  &  &  &  &  &  &  &
\\ & 2 & $2^3 \rtimes \frakS_4$ & 192 & 1 & 13 & 30 &  &  & 32 & 12 & 72 &  &  &  & 32 &  &  &  & \\
\hline
4B & $\ne 2$ & $\frakS_4$ & 24 & 1 & 6 & 3 &  &  & 8 &  & 6 &  &  &  &  &  &  &  & \\
 & 2 & \multicolumn{2}{c|}{\cellcolor{gray!20}(same as 4A)} & \multicolumn{16}{c|}{\cellcolor{gray!20}} \\
\hline
5A & $\ne 2,5$ & $\frakS_5$ & 120 & 1 & 10 & 15 &  &  & 20 &  & 30 & 24 &  &  & 20 &  &  &  & \\
 & 2 & \multicolumn{2}{c|}{\cellcolor{gray!20}(same as 3C)} & \multicolumn{16}{c|}{\cellcolor{gray!20}} \\
\hline
6E & $\ne 2$ & $\frakS_3 \times \frakS_2$ & 12 & 1 & 4 & 3 &  &  & 2 &  &  &  &  &  & 2 &  &  &  & \\
 & 2 & \multicolumn{2}{c|}{\cellcolor{gray!20}(same as 4A)} & \multicolumn{16}{c|}{\cellcolor{gray!20}} \\
\hline
8A & $\ne 2,3$ & $8$ & 8 & 1 & 1 &  &  &  &  & 2 &  &  &  &  &  &  & 4 &  & \\
 & 3 & $\calH_3(3) \rtimes 8$ & 216 & 1 & 9 &  & 2 &  & 24 & 18 &  &  & 18 &  &  &  & 108 &  & 36\\
\hline
12A & $\ne 2,3$ & $\calH_3(3) \rtimes 4$ & 108 & 1 & 9 &  & 2 &  & 24 & 18 &  &  & 18 &  &  &  &  &  & 36\\
 & 3 & \multicolumn{2}{c|}{\cellcolor{gray!20}(same as 8A)} & \multicolumn{16}{c|}{\cellcolor{gray!20}} \\
 & 2 & \multicolumn{2}{c|}{\cellcolor{gray!20}(same as 3C)} & \multicolumn{16}{c|}{\cellcolor{gray!20}} \\
\hline
\end{tabular}
}
\caption{Automorphism groups of cubic surfaces with enumeration of
conjugacy classes.}
\label{tbl:cubicAutosWithCCs}
\end{table}
\end{landscape}

\begin{landscape}
\begin{table}[htbp]
\centering
\renewcommand{\arraystretch}{1.3}
\scalebox{0.84}{%
\begin{tabular}{|c|c|c|l|}
\hline
Name & $\operatorname{char}$ & Normal Form & Reference \\
\hline
2A & $\ne 2$ &
$(x_0+x_1)(x_0x_1+c_0x_2^2+c_1x_3^2+c_2x_2x_3)+x_2x_3(x_2+x_3)$
& Proposition~\ref{prop:normalForms2A} \\

& $2$ &
$x_0x_1x_2 + (x_0+x_1)^2x_3 +
c_0(x_0+x_1)x_3^2 + c_1x_2^3 + c_2x_2x_3^2 + x_3^3$
& Proposition~\ref{prop:normalForms2A} \\ [-1ex]
& & {\tiny (Note: $p=2$ normal form does not include 2B surfaces)} & \\

\hline

2B & $\ne 2$ &
$x_0^2(x_2+c_0x_3) + x_1^2(c_1x_2+x_3) + x_2x_3(x_2+x_3) = 0$
& Proposition~\ref{prop:normalForms2B} \\

& $2$ &
$x_0^3+x_1^3 + x_2^3+x_3^3 +
c_0x_2x_3(x_0+x_1)+c_1x_0x_1(x_2+x_3)$
& Proposition~\ref{prop:normalForms2B} \\

\hline

3A & $\ne 3$ &
$x_0^3 + x_1^3 + x_2^3 + x_3^3 + c x_0x_1x_2$
& Proposition~\ref{prop:normalForms3Apnot3} \\

& $3$ &
$x_0^3 + x_0 x_3^2 - x_1x_2^2 + x_1^2x_3 + c x_1 x_3^2$
& Proposition~\ref{prop:normalForms3Ap3} \\

\hline

3C & $\ne 3$ &
$x_0^3 + x_1^3 + x_2^3 + x_3^3$ 
& Lemma~\ref{lem:mush3C6C6F9A} \\

\hline

3D & $\ne 3$ &
$x_0^3+x_1^3+x_2^3+x_3^3 +c_0x_0x_1x_2 + c_1(x_0+x_1+x_2)x_3^2$
& Lemma~\ref{lem:normalForms3D} \\ [-1ex]
& & {\tiny (Note: $p \ne 3$ normal form does not include 6E surfaces)} &
\\ [-1ex]

& $3$ &
$(c_0(x_0+x_1+x_2)+c_1x_3)(x_0x_1+x_0x_2+x_1x_2)
+x_0x_1x_2+(x_0+x_1+x_2)x_3^2$
& Lemma~\ref{lem:normalForms3D} \\

\hline

4A & $\ne 2$ &
$x_3^2 x_2 + x_2^2x_0 + x_1 (x_1-x_0) (x_1 - c x_0)$
& Lemma~\ref{lem:normalForms4A} \\

 & 2 & \cellcolor{gray!20}(same as 4B) & \\

\hline

4B & any &
$x_0^3+x_1^3+x_2^3+x_3^3+c(x_0x_1x_2 + x_0x_1x_3 + x_0x_2x_3
 + x_1x_2x_3)$
& Proposition~\ref{prop:normalForms4B} \\

\hline

5A & $\ne 2,5$ & $\left(\sum_{i \ne j} x_i^2x_j \right)
+ 2\left(\sum_{i < j < k} x_ix_jx_k\right)$
& Proposition~\ref{prop:normalForms5A} \\
 & 2 & \cellcolor{gray!20}(same as 3C) & \\

\hline

6E & $\ne 2,3$ &
$x_0^3+x_1^3+x_2^3+(x_0+x_1+x_2)x_3^2+cx_0x_1x_2$
& Proposition~\ref{prop:normalForms6E} \\

& $3$ &
$c(x_0+x_1+x_2)(x_0x_1+x_0x_2+x_1x_2) + (x_0x_1x_2)
+ (x_0+x_1+x_2)x_3^2$
& Proposition~\ref{prop:normalForms6E} \\

 & 2 & \cellcolor{gray!20}(same as 4B) & \\

\hline

8A & $\ne 2$ &
$x_0^3 + x_0x_3^2 - x_1x_2^2 + x_1^2x_3$
& Lemma~\ref{lem:8Aautos} \\

\hline

12A & $\ne 2,3$ &
$x_0^3 + x_1^3 + x_2^3 + x_3^3 + 3(\sqrt{3}-1) x_0x_1x_2$
& Lemma~\ref{lem:12Aautos} \\

 & 3 & \cellcolor{gray!20}(same as 8A) & \\

 & 2 & \cellcolor{gray!20}(same as 3C) & \\

\hline
\end{tabular}
}
\caption{Normal forms of cubic surfaces for each automorphism group stratum.}
\label{tbl:Normal Forms}
\end{table}
\end{landscape}

 \end{document}